\documentclass{elsarticle} 
\usepackage[outline]{contour} 
\contourlength{0.4pt}
\makeatletter
	\def\ps@pprintTitle{%
 	\let\@oddhead\@empty
	\let\@evenhead\@empty
	\def\@oddfoot{\centerline{\thepage}}%
	\let\@evenfoot\@oddfoot}
\makeatother
\usepackage[dvipsnames]{xcolor}
\usepackage[english]{babel}
\usepackage{amsfonts}
\usepackage[T1]{fontenc}
\usepackage{algorithm}
\usepackage{algpseudocode} 

\usepackage[margin=2cm]{geometry}
\usepackage{amsthm}
\usepackage{array,makecell} 
\usepackage{subcaption}     
\usepackage{tikz}
\usepackage[utf8]{inputenc}
\usepackage{pgfplots} 
\usepackage{pgfgantt}
\usepackage{pdflscape}
\pgfplotsset{compat=newest} 
\pgfplotsset{plot coordinates/math parser=false} 
\newlength\fwidth
\newlength\fheight
\usepackage{amsmath}
\usepackage{amsthm}
\usepackage{graphicx}
\usepackage{subcaption}
\usepackage{overpic}
\usepackage{siunitx}
\usepackage{mathtools}
\newtheorem{definition}{Definition}
\usepackage[textsize=tiny]{todonotes}

\newcommand{\Ra}[1]{\color{black} {#1}}
\newcommand{\Rb}[1]{\color{black} {#1}}
\newcommand{\Rc}[1]{\color{black} {#1}}
\newcommand{\Rd}[1]{\color{black} {#1}}
\usepackage[colorlinks=true, allcolors=blue]{hyperref}
\newtheorem{theorem}{Theorem}
\theoremstyle{definition}
\newtheorem{remark}{Remark}
\usepackage{booktabs}
\newtheorem{assumption}{Assumption}

\DeclareMathOperator{\sech}{sech}
\newcommand{\real}{\mathbb{R}}
\newcommand{\bzero}{\mathbf{0}}

\newcommand{\bvarphi}{ \boldsymbol \varphi}
\newcommand{\bvarphihat}{\widehat{ \bvarphi}}
\newcommand{\bphihat}{\widehat{ \bphi}}

\newcommand{\dt}{\Delta t}
\newcommand{\y}{\mathbf y}
\newcommand{\dd}{\rm d}
\newcommand{\f}{\mathbf{f}}

\newcommand{\q}{\mathbf q}
\newcommand{\w}{\mathbf w}
\newcommand{\kapbar}{\bar{\kappa}}
\newcommand{\qdot}{\dot{\q}}
\newcommand{\pdot}{\dot{\p}}
\newcommand{\ydot}{\dot{\y}}
\newcommand{\yhat}{\widehat{\y}}
\newcommand{\ybar}{\bar{\y}}
\newcommand{\nbar}{\bar{n}}
\newcommand{\what}{\widehat{\w}}
\newcommand{\phat}{\widehat{\p}}
\newcommand{\phatdot}{\dot{\phat}}

\newcommand{\p}{\mathbf p}
\newcommand{\x}{\mathbf x}

\newcommand{\bphi}{\boldsymbol \phi}
\newcommand{\bPhi}{\boldsymbol \Phi}

\newcommand{\bq}{\q}

\newcommand{\V}{\mathbf V}
\newcommand{\fn}{\f_{\text{non}}}
\newcommand{\Y}{\mathbf Y}

\newcommand{\A}{\mathbf A}

\newcommand{\D}{\mathbf D}
\newcommand{\Dhat}{\widehat{\D}}

\newcommand{\B}{\mathbf B}
\newcommand{\Q}{\mathbf Q}
\newcommand{\qhat}{\widehat{\q}}

\newcommand{\qhatdot}{\mathbf{\dot{\widehat q}}}

\newcommand{\Pp}{\mathbf P}

\newcommand{\In}{\mathbf I_n}
\newcommand{\Jn}{\mathbf J_{2n}}
\newcommand{\Jr}{\mathbf J_{2r}}

\begin{document}
\begin{frontmatter}
\title{Nonlinear energy-preserving model reduction with lifting transformations that quadratize the energy}

 		\author[affil1]{Harsh Sharma\corref{cor1}}
		 		\cortext[cor1]{Corresponding author}
		\ead{hasharma@ucsd.edu}
 		\author[affil1]{Juan Diego Draxl Giannoni}
 		\author[affil1]{Boris Kramer}

			\address[affil1]{Department of Mechanical and Aerospace Engineering, University of California San Diego, CA, United States}

\begin{abstract}
Existing model reduction techniques for high-dimensional models of conservative partial differential equations (PDEs) encounter computational bottlenecks when dealing with systems featuring non-polynomial nonlinearities. This work presents a nonlinear model reduction method that employs lifting variable transformations to derive structure-preserving quadratic reduced-order models for conservative PDEs with general nonlinearities. We present an energy-quadratization strategy that defines the auxiliary variable in terms of the nonlinear term in the energy expression to derive an equivalent quadratic lifted system with quadratic system energy. The proposed strategy combined with proper orthogonal decomposition model reduction yields quadratic reduced-order models that conserve the quadratized lifted energy exactly in high dimensions. We demonstrate the proposed model reduction approach on four nonlinear conservative PDEs: the one-dimensional wave equation with exponential nonlinearity, the two-dimensional sine-Gordon equation, the two-dimensional Klein-Gordon equation with parametric dependence,  and the two-dimensional Klein-Gordon-Zakharov equations. The numerical results show that the proposed lifting approach is competitive with the state-of-the-art structure-preserving hyper-reduction method in terms of both accuracy and computational efficiency in the online stage while providing significant computational gains in the offline stage.
 \end{abstract}
\end{frontmatter}
\section{Introduction}
\label{sec:introduction}
High-dimensional nonlinear full-order models (FOMs) of conservative PDEs appear in a wide variety of science and engineering areas ranging from climate modeling to plasma physics. These nonlinear FOMs are typically derived via structure-preserving methods~\cite{bridges2006numerical} that discretize conservative PDEs such that the resulting space-discretized models satisfy nonlinear conservation laws. To meet the rapidly increasing demand for accurate numerical simulations of complex physical systems, the field of projection-based model reduction~\cite{benner2015survey} has developed principled techniques to derive computationally efficient reduced-order models (ROMs) via the projection of FOM operators on low-dimensional subspaces. A major challenge in the context of model reduction of  conservative nonlinear FOMs is  preserving the qualitative properties of the dynamics in the reduced setting, especially the energy conservation law, as it plays a crucial role in characterizing the nonlinear dynamics.

Structure-preserving model reduction was first explored in the context of Lagrangian mechanical systems~\cite{lall2003structure} where the authors derived nonlinear Lagrangian ROMs that conserved the nonlinear FOM energy. For nonlinear FOMs of interconnected systems modeled using the port-Hamiltonian framework,  structure-preserving model reduction via preservation of the underlying Dirac structure has been presented in, e.g.,~\cite{polyuga2010structure, hartmann2010balanced, gugercin2012structure, polyuga2012effort}. In the context of the Hamiltonian formulation of conservative FOMs, symplectic model reduction of Hamiltonian systems was introduced in~\cite{peng2016symplectic} where the authors derived Hamiltonian ROMs by projecting the Hamiltonian FOM operators onto a symplectic subspace obtained via proper symplectic decomposition. Building on this work,  reduced basis methods for structure-preserving model reduction of parametric Hamiltonian systems and Poisson systems have been presented in~\cite{afkham2017structure} and~\cite{hesthaven2021structure}, respectively. A  more general symplectic model reduction approach with non-orthonormal bases was proposed in~\cite{buchfink2019symplectic}. In a similar direction, the authors in~\cite{gong2017structure} modified the standard POD-Galerkin approach such that the Hamiltonian structure is preserved after the Galerkin projection step. The above methods laid the foundation for recent developments in
structure-preserving model reduction for Hamiltonian systems and more general gradient systems in both intrusive~\cite{hesthaven2021structureb,hesthaven2022rank, schein2021preserving, buchfink2023symplectic, sharma2023symplectic, gruber2024variationally} and nonintrusive settings~\cite{sharma2022hamiltonian,sharma2024preserving,gruber2023canonical,geng2024gradient,filanova2023operator,sharma2024lagrangian}. However, all of the aforementioned structure-preserving approaches suffer from computational efficiency issues as the evaluation of the nonlinear components of the structure-preserving ROM vector field still scales with the FOM dimension. 

This computational bottleneck is a well-known issue in the nonlinear model reduction community and has led to the development of hyper-reduction methods. These methods introduce a second level of approximation to reduce the computational cost involved in the evaluation of the nonlinear terms. In particular, the discrete empirical interpolation method (DEIM)~\cite{barrault2004empirical,chaturantabut2010nonlinear,drmac2016new} has been shown to be effective for nonlinear model reduction over a range of applications. However, the standard hyper-reduction approaches for general nonlinear systems do not preserve the underlying Hamiltonian structure. 

While the subfield of hyper-reduction methods has grown considerably in the last decade, progress on structure-preserving hyper-reduction has been less rapid. The authors in~\cite{farhat2015structure} presented a cubature approach, the energy-conserving sampling and weighting method, where the vector field obtained from the nonlinear conservative FOM is approximated with a weighted average of the nonlinear vector field components on a coarser mesh. However, this approach is limited to nonlinear conservative FOMs obtained via finite element discretization of conservative PDEs. In~\cite{carlberg2015preserving}, the authors presented a DEIM method for nonlinear mechanical oscillators in structural dynamics where the trajectory does not deviate drastically from equilibrium. The underlying assumption about trajectories being localized around an equilibrium point, however, is generally not met by nonlinear conservative FOMs considered herein. A modification of the DEIM method for nonlinear port-Hamiltonian systems has been presented in~\cite{chaturantabut2016structure} where the nonlinear gradient of the Hamiltonian function is approximated in the space where the DEIM projection is orthogonal. Even though this method can be adapted to the Hamiltonian setting to derive nonlinear ROMs that preserve the underlying Hamiltonian structure, the paper does not provide state space error bounds for the resulting ROM. For dynamical systems with a first integral, an energy-conserving  hyper-reduction approach via preservation of the skew-symmetric structure has been proposed in~\cite{miyatake2019structure} but the cost of deriving the ROM operators for this method is computationally prohibitive as it scales with the FOM dimension. In~\cite{klein2024energy}, the authors presented an energy-conserving hyper-reduction method for energy-and momentum-conserving ROMs of incompressible Navier-Stokes equations. However, the problem-specific formulation of this method does not generalize to many nonlinear conservative PDEs. The authors in~\cite{wang2021structure} obtained a structure-preserving DEIM approximation for general nonlinear Hamiltonian systems by decomposing the Hamiltonian into the Euclidean product of the nonlinear vector field with a constant vector. Building on this idea, the authors in~\cite{pagliantini2023gradient} developed a gradient-preserving hyper-reduction approach that guarantees the preservation of the FOM Hamiltonian asymptotically. Moreover, this work also provided state space error bounds by extending the a priori error estimate result for the standard POD-DEIM in~\cite{chaturantabut2012state} to the Hamiltonian setting. This structure-preserving DEIM (spDEIM) method, however, requires a computationally demanding offline phase for parametric problems.

Projection-based nonlinear model reduction via lifting~\cite{gu2011qlmor, benner2015two, kramer2019nonlinear,benner2018mathcal,kramer2019balanced} has emerged as a promising research direction. These methods first introduce auxiliary variables to transform a general nonlinear system into an equivalent finite-dimensional lifted system with quadratic dynamics and then project the quadratic operators of the high-dimensional lifted system onto a reduced space to obtain a quadratic ROM. The key advantage of quadratizing (polynomializing) the dynamics before nonlinear model reduction is that it eliminates the need for additional hyper-reduction/interpolation techniques. This lifting approach has also been extended to the nonintrusive setting in~\cite{qian2020lift, qian2019transform,gosea2018data,swischuk2020,bychkov2024exact} where lifting transformations are exploited to learn low-order polynomial ROMs of complex nonlinear FOMs directly from the lifted data. While the existence of a finite-dimensional quadratic representation in a lifted setting is not universally guaranteed (see~\cite{KraPog_survey_quadratization_2025} for a survey of results), the work in~\cite{gu2011qlmor,bychkov2024exact} shows that almost all nonlinear systems in science and engineering applications can be lifted to quadratic form. While most lifting transformations in the literature have been derived by hand, algorithms with provable convergence guarantees exist~\cite{hemery2020complexity}, such as \textsf{QBee} \cite{bychkov2024exact} and \textsf{BIOCHAM}~\cite{biocham4}. Due to the problem-specific and non-unique nature of lifting transformations, these transformations are specifically derived for each problem. However, this current approach focuses on the nonlinear terms in the governing equations and does not account for the qualitative properties (e.g., energy conservation) of the nonlinear conservative FOM.  As a result, the quadratic ROM obtained by projecting the lifted quadratic FOM onto a reduced space is not guaranteed to conserve energy (see Section~\ref{sec:motivation}). 

The main goal of this work is to develop a structure-preserving lifting approach and subsequently derive computationally efficient quadratic ROMs that conserve the energy of the nonlinear PDE. The main contributions of this work are:
\begin{enumerate}
\item We present an energy-quadratization strategy for deriving lifting transformations that yields structure-preserving quadratic ROMs. In contrast to the standard lifting approach of defining auxiliary variables based on the nonlinear terms in the governing equations, the proposed approach defines auxiliary variables such that the system energy in the lifted setting is quadratic. We also present a theoretical result that shows that the proposed strategy combined with POD model reduction yields quadratic ROMs that are guaranteed to conserve the lifted FOM energy exactly.
\item We derive the structure-preserving lifting transformations and the corresponding quadratic ROMs for three nonlinear Hamiltonian PDEs, including a two-dimensional nonlinear wave equation with parametric dependence. The numerical results show that quadratic ROMs obtained via the proposed energy-quadratization strategy achieve accuracy and computational efficiency similar to the state-of-the-art spDEIM ROMs with significantly lower offline computational costs.
\item We apply the proposed energy-quadratization strategy to the Klein-Gordon-Zakharaov equations, a system of coupled PDEs with a nonlinear conservation law, to derive accurate and stable quadratic ROMs of a nonlinear conservative FOM with $960{,}000$ degrees of freedom. This numerical example demonstrates the applicability of the proposed structure-preserving lifting approach for a wider class of nonlinear conservative PDEs that do not have a canonical Hamiltonian formulation.
\end{enumerate}
The paper is structured as follows. Section~\ref{sec:background} reviews the basics of conservative PDEs and their symplectic model reduction. We also summarize the standard lifting approach for nonlinear model reduction of general nonlinear systems. Section~\ref{sec:method} introduces our structure-preserving lifting approach for deriving quadratic ROMs that are guaranteed to be energy-conserving at the lifted FOM level. Section~\ref{sec:numerical} demonstrates the proposed structure-preserving lifting approach on four nonlinear conservative PDEs with increasing complexity: one-dimensional nonlinear wave equation with exponential nonlinearity, two-dimensional sine-Gordon equation, two-dimensional Klein-Gordon equation with parametric dependence, and  two-dimensional Klein-Gordon-Zakharov equation. Finally, Section~\ref{sec:conclusion} provides concluding remarks and suggests future research directions.

\section{Background}
\label{sec:background}
In Section~\ref{sec:fom} we introduce the nonlinear conservative PDEs  considered in this work. In Section~\ref{sec:spmor} we review the structure-preserving model reduction of nonlinear conservative FOMs. {\Ra{In Section~\ref{sec:spdeim}  we provide an overview of the structure-preserving hyper-reduction strategy from~\cite{pagliantini2023gradient}, which we use later to compare our results with.}} In Section~\ref{sec:lifting} we discuss the construction of projection-based nonlinear ROMs via lifting transformations.
\subsection{Nonlinear conservative PDEs and their spatial discretization}
\label{sec:fom}
Let $\Omega \subseteq \real^d$ be the spatial domain and consider the nonlinear wave PDE
\begin{equation}
\frac{\partial^2 \phi (\x,t)}{\partial t^2}=\mathbf \Delta \phi (\x,t) - f_{\text{non}}(\phi (\x,t)), 
\label{eq:pde}
\end{equation}
where $\x=(x_1,x_2,\cdots,x_d) \in \Omega$ is the spatial variable, $t$ is time, $\mathbf \Delta$ is the Laplacian operator, $\phi (\x,t)$ is the scalar state field,  and the nonlinear component of the vector field, $f_{\text{non}}(\phi (\x,t)):=\nabla(g(\phi(\x,t)))$, is derived from a smooth nonlinear function $g(\phi(\x,t))$. A characteristic feature of nonlinear wave equations of the form~\eqref{eq:pde} is that the total energy
\begin{equation}
\mathcal E[\phi(\x,t)]:=\int_{\Omega} \left( \frac{1}{2}\left( \frac{\partial \phi(\x,t)}{\partial t} \right)^2 + \frac{1}{2}\left(\nabla \phi(\x,t) \right)^2  + g(\phi(\x,t)) \right) \dd \x,
\label{eq:epde}
\end{equation}
is conserved for $t\geq0$. The integrand in~\eqref{eq:epde} typically has the physical interpretation of energy density with the first term in the integrand representing the kinetic energy and the other two terms representing the nonlinear potential energy.

Since energy conservation is intricately related to accurate and stable numerical solution of nonlinear conservative PDEs, the field of structure-preserving numerical methods~\cite{sharma2020review} has developed a wide variety of spatial discretization schemes that conserve the energy. These approaches typically derive energy-conserving schemes in two steps. The first step is to rewrite~\eqref{eq:pde} into the first-order form by defining $q(\x,t):=\phi(\x,t)$ and $p(\x,t):=\frac{\partial}{\partial t}(\phi(\x,t))$, i.e.,
\begin{equation}
\label{eq:first_pde}
\frac{\partial q (\x,t)}{\partial t}=p(\x,t), \qquad \frac{\partial p (\x,t)}{\partial t}=\mathbf \Delta q (\x,t) - f_{\text{non}}(q (\x,t)). 
\end{equation}
The second step discretizes the first-order PDEs in~\eqref{eq:first_pde} to derive nonlinear conservative FOMs of the form
\begin{equation}
    \dot{\q}(t)=\p(t), \qquad
    \dot{\p}(t)=\D \q(t) - \fn(\q(t)),
\label{eq:cons_fom}
\end{equation}
where $\qdot(t) \in \real^{n}$ and $\pdot(t)\in \real^{n}$ denote the time derivatives of the space-discretized state vectors $\q(t)\in \real^{n}$ and $\p(t)\in \real^{n}$, respectively, $\D=\D^\top \in \real^{n \times n}$ is the symmetric {\Ra{discrete Jacobian obtained via symmetric finite differences}}, and $ \fn(\q(t)) \in \real^{n}$ is the nonlinear component of the space-discretized vector field. The FOM~\eqref{eq:cons_fom} conserves the space-discretized nonlinear FOM energy
\begin{equation}
E(\q,\p,t)=\frac{1}{2} \p(t)^\top \p(t) - \frac{1}{2}\q(t)^\top\D \q(t) + \sum_{i=1}^n \left( g(q_i(t)) \right).
\label{eq:ed_fom}
\end{equation}
From hereon, we simplify the notation by omitting explicit dependence on time: the FOM state vectors $\q(t)$ and $\p(t)$ at time $t$ are therefore denoted as $\q$ and $\p$, respectively.   
\subsection{Symplectic model reduction of nonlinear conservative PDEs}
\label{sec:spmor}
In this section, we review the derivation of structure-preserving ROMs for nonlinear conservative PDEs of the form~\eqref{eq:pde}.  
Symplectic model reduction methods leverage the canonical Hamiltonian structure of~\eqref{eq:pde} to write the nonlinear conservative FOM in the form
\begin{equation}
\ydot=\Jn\nabla H(\y),
\label{eq:fom}
\end{equation}
where $\y=[\q^\top,\p^\top]^\top \in \real^{2n}$ is the space-discretized FOM state vector, $\Jn\in \real^{2n\times 2n}$ is the canonical symplectic matrix, and the Hamiltonian function, $H(\y)$, is the space-discretized nonlinear FOM energy from~\eqref{eq:ed_fom}, i.e. $H(\y)=E(\q,\p)$. The main goal in symplectic model reduction is to derive nonlinear conservative ROMs that are Hamiltonian.

Symplectic model reduction techniques approximate the high-dimensional state as $\y \approx \V \yhat$ where $\yhat \in \real^{2r}$ is the reduced state vector and $\V \in \real^{2n \times 2r}$ is the symplectic basis matrix that satisfies $\V^\top\Jn\V=\Jr$. The symplectic Galerkin projection is defined as $\yhat=\V^+\y \in \real^{2r}$ where $\V^+$ is the symplectic inverse of the symplectic basis matrix $\V$, i.e., $\V^+:=\Jr^\top\V^\top\Jn \in \real^{2r \times 2n}$. The authors in~\cite{peng2016symplectic} present three proper symplectic decomposition (PSD) techniques for constructing the symplectic basis matrix $\V$ from the FOM snapshot data matrices $\Q:=[\q_1,\cdots,\q_K] \in \real^{n \times K}$ and $\Pp:=[\p_1,\cdots,\p_K] \in \real^{n \times K}$.

In this work, we focus on the cotangent lift algorithm which computes a PSD basis matrix $\V$ with block-diagonal structure that preserves the topology of the FOM state vector. This algorithm builds an orthogonal and symplectic basis matrix of the form $\V=\text{blkdiag}(\bPhi,\bPhi)\in \real^{2n\times 2r}$ where $\bPhi \in \real^{n\times r}$ is computed via singular value decomposition (SVD) of extended snapshot data matrix $\Y_e:=[\Q,\Pp] \in \real^{n \times 2K}$.  {\Ra{Assuming the residual $\mathbf r=\V \dot{\yhat}-\Jn\nabla H(\V\yhat)$ is orthogonal to the column space of the PSD basis matrix $\V$ for all $t\geq 0$,}} the nonlinear Hamiltonian ROM can be derived from the symplectic projection of the nonlinear FOM~\eqref{eq:fom} onto the symplectic subspace, i.e., ${\Ra{\V^\top\mathbf r =\bzero}}$. The governing equations for the resulting $2r-$dimensional Hamiltonian ROM are
\begin{equation}
\dot{\yhat}=\Jr\nabla \widehat{H}(\yhat),
\label{eq:ham_rom}
\end{equation}
where the Hamiltonian for the ROM is defined as
\begin{equation}
\widehat{H}(\yhat):=H(\V\yhat)=\frac{1}{2} \phat^\top \phat -\frac{1}{2}\qhat^\top\left(\V^\top\D\V\right) \qhat + {\Ra{\widehat{H}_{\text{non}}(\qhat)}}, 
\label{eq:ham_red}
\end{equation}
{\Ra{where $\widehat{H}_{\text{non}}(\qhat):=\sum_{i=1}^n \left( g(\bPhi_i\qhat) \right)$ is the nonlinear component of the reduced Hamiltonian with $\bPhi_i$ denoting}} the $i$th row of $\bPhi$. While the nonlinear Hamiltonian ROM~\eqref{eq:ham_rom} is guaranteed to preserve the Hamiltonian structure, the computational cost of evaluating its vector field still scales with the FOM dimension $n$.  
\begin{remark} 
{\Ra{Since this work is mainly focused on nonlinear conservative FOMs of the form~\eqref{eq:cons_fom}, we find that the using the cotangent lift basis provides accurate approximate solutions to the FOM. For problems where the cotangent lift basis is not useful or not available, one avenue to consider would be to use the variationally consistent model reduction approach from~\cite{gruber2025variationally} which is compatible with any linear orthonormal basis. However, the energy-conserving property of the lifted ROMs no longer holds with this approach. Extension of the proposed structure-preserving lifting approach for problems where the contangent lift basis is not useful is non-trivial and left for future work.}}
\end{remark}
\subsection{Structure-preserving hyper-reduction for Hamiltonian systems}
\label{sec:spdeim}
A gradient-preserving DEIM strategy for Hamiltonian systems has been presented recently in~\cite{pagliantini2023gradient} where the authors first map the high-dimensional nonlinear Hamiltonian gradient into the reduced space via symplectic projection and then approximate nonlinear components of the reduced Hamiltonian gradient via a suitable DEIM projection of the Jacobian of the nonlinear function $\f_{\text{non}}$ in~\eqref{eq:cons_fom}. 
{\Ra{As a first step, the nonlinear component of the reduced Hamiltonian in~\eqref{eq:ham_red} is written as $\widehat{H}_{\text{non}}(\qhat)=\textbf{1}^\top\widehat{\mathbf h}_{\text{non}}(\bPhi\qhat)$ where $\textbf{1}\in \real^{n}$ is a constant vector with $1$ in each entry and $\widehat{\mathbf h}_{\text{non}}(\bPhi\qhat):=[g(\bPhi_1\qhat), \cdots, g(\bPhi_n\qhat)]^\top$. Based on this decomposition, the gradient of the nonlinear component of the reduced Hamiltonian can be written as
\begin{equation*}
\nabla \widehat{H}_{\text{non}}(\qhat)=\bPhi^\top \D\widehat{\mathbf h}_{\text{non}}^\top(\bPhi\qhat)\textbf{1},
\end{equation*}
where $\D\widehat{\mathbf h}_{\text{non}}(\bPhi\qhat) \in \real^{n \times n}$ is the Jacobian of $\widehat{\mathbf h}_{\text{non}}(\qhat)$.
The key idea is to use discrete empirical interpolation to approximate the reduced Jacobian term $\D\widehat{\mathbf h}_{\text{non}}(\bPhi\qhat)\bPhi$ as follows
\begin{equation}
\D\widehat{\mathbf h}_{\text{non}}(\bPhi\qhat)\bPhi\approx \V_{\text{DEIM}}(\mathbb P^\top \V_{\text{DEIM}})^{-1} \mathbb P^\top \D\widehat{\mathbf h}_{\text{non}}(\bPhi\qhat), 
\label{eq:spdeim}
\end{equation}
where $\V_{\text{DEIM}} \in \real^{n \times m}$ is the DEIM basis and $\mathbb P \in \real^{d \times m}$ is the corresponding selection matrix, see~\cite{chaturantabut2010nonlinear} for more details. The DEIM basis $\V_{\text{DEIM}}$ in~\eqref{eq:spdeim} is obtained from reduced Jacobian snapshots via SVD. With the reduced Jacobian term approximated as in~\eqref{eq:spdeim}, the governing equations for the resulting $2r-$dimensional hyper-reduced Hamiltonian ROM are
\begin{equation}
\dot{\yhat}=\Jr\nabla \widehat{H}_{\text{spDEIM}}(\yhat),
\label{eq:ham_rom_spdeim}
\end{equation}
where the hyper-reduced Hamiltonian is defined as 
\begin{equation*}
\widehat{H}_{\text{spDEIM}}(\yhat):=\frac{1}{2} \phat^\top \phat -\frac{1}{2}\qhat^\top\left(\V^\top\D\V\right) \qhat + \textbf{1}^\top \V_{\text{DEIM}}(\mathbb P^\top \V_{\text{DEIM}})^{-1} \mathbb P^\top\widehat{\mathbf{h}}_{\text{non}}(\bPhi\qhat).
\end{equation*}
}}This gradient-preserving DEIM strategy ensures the conservation of the FOM Hamiltonian in an asymptotic sense at a computational cost independent of the FOM dimension $n$. However, the computation of the structure-preserving DEIM basis in this approach is computationally demanding in the offline stage, especially for nonlinear conservative FOMs of multi-dimensional PDEs with parametric dependence.
\subsection{Nonlinear model reduction via lifting transformations}
\label{sec:lifting}
We review the {\Rc{standard}} lifting approach from~\cite{kramer2019nonlinear} {\Rc{to derive}} quadratic ROMs of general nonlinear FOMs. The lifting approach avoids the hyper-reduction step by first transforming nonlinear FOMs into quadratic models and then projecting the lifted FOM operators onto a reduced space. In the following, the notation $\otimes$ denotes the Kronecker product of matrices or vectors. We begin by defining a polynomialization for a general nonlinear FOM. 
 \begin{definition}[\textbf{Polynomialization and Quadratization}~\cite{bychkov2024exact}] Consider a system of nonlinear ODEs
\begin{equation}
\dot{\y}=\f(\y),
\label{eq:gen_non}
\end{equation}
where $\y \in \real^n$ is the state vector and $\f(\y)=[f_{1}(\y), \cdots, f_{n}(\y)]^\top$ is an $n-$dimensional vector of real-valued functions. Then an $\ell-$dimensional vector of new variables 
\begin{equation}
\w=[w_1(\y),\cdots,w_{\ell}(\y)]^\top
\end{equation}
is said to be a polynomialization of~\eqref{eq:gen_non} if there exist polynomial vectors $\bar{\f}(\y,\w)$ and $\bar{\f}_{{\textup{aux}}}(\y,\w)$  of dimension $n$ and $\ell$, respectively, such that
\begin{equation*}
\dot{\y}=\bar{\f}(\y,\w) \quad and \quad \dot{\w}=\bar{\f}_{{\textup{aux}}}(\y,\w),
\end{equation*}
for each $\y$ solving ~\eqref{eq:gen_non}. Similarly, $\w$ is said to be a quadratization of~\eqref{eq:gen_non} if all the entries in both $\bar{\f}(\y,\w)$ and $\bar{\f}_{{\textup{aux}}}(\y,\w)$ are polynomial of total degree at most two.

\end{definition}

For nonlinear model reduction, the idea of lifting is to find a quadratization $\w \in \real^{k \cdot n}$ of 
the nonlinear FOM~\eqref{eq:gen_non} in terms of $k$ additional variables for some positive integer $k\in \mathbb{Z}^+$ and  then employ a transformation $\tau: \real^{n} \to \real^{\bar{n}}$ with $\bar{n}=(k+1)\cdot n$  that transforms the original nonlinear system with state vector $\y\in \real^n$ into an equivalent lifted system with augmented state vector $\ybar=[\y^\top,\w^\top]^\top \in \real^{\bar{n}}$ yet quadratic dynamics, i.e., 
\begin{equation} 
\ydot=\f(\y) \qquad \xrightarrow{\ybar=\tau(\y)} \qquad \dot{\ybar}=\bar{\A}\ybar + \bar{\B}(\ybar \otimes \ybar),
\label{eq:lifted_FOM}
\end{equation}
where $\ybar\in \real^{\nbar}$ is the lifted state vector, and $\bar{\A} \in \real^{\nbar \times\nbar}$ and $\bar{\B}\in \real^{\nbar \times\nbar^2}$ are the linear and quadratic FOM operators in the lifted setting, respectively. To reduce the lifted FOM~\eqref{eq:lifted_FOM} with quadratic dynamics, we approximate the lifted state as $\ybar \approx \bar{\V}\ybar_r$ with POD basis matrix $\bar{\V} \in \real^{\nbar \times \bar{r}}$ and then perform a standard Galerkin projection to derive the quadratic ROM
\begin{equation}
\dot{\ybar}_r=\bar{\A}_r\ybar_r + \bar{\B}_r(\ybar_r \otimes \ybar_r),
\label{eq:std_lifting}
\end{equation}
where the reduced operators $\bar{\A}_r:=\bar{\V}^\top\bar{\A}\bar{\V} \in \real^{\bar{r} \times\bar{r}}$ and $\bar{\B}_r:=\bar{\V}^\top\bar{\B}(\bar{\V}\otimes \bar{\V})\in \real^{\bar{r} \times \bar{r}^2}$ can be computed via projection of the lifted FOM operators onto the low-dimensional basis $\bar{\V}$.  

A key advantage of the lifting approach is that the reduced operators can be precomputed efficiently in the offline phase once the POD basis matrix $\bar{\V}$ is chosen. Thus, the POD-Galerkin ROM of the lifted FOM yields a quadratic ROM without any additional hyper-reduction. Although the lifting approach provides an efficient online-offline decomposition by eliminating non-polynomial nonlinearities, it does not take the additional qualitative features like conservation laws into account. As a result, the POD-Galerkin ROM~\eqref{eq:std_lifting} is not guaranteed to be energy-conserving for nonlinear FOMs of the form~\eqref{eq:cons_fom}. 

\begin{remark}{\Ra{
In contrast to the straightforward construction of the linear ROM operator $\bar{\A}_r$, the quadratic ROM operator $\bar{\B}_r$ needs to be computed carefully as forming $\bar{\V} \otimes \bar{\V} \in \real^{\nbar^2 \times \bar{r}^2}$ can be computationally prohibitive even for moderately-sized FOMs. We follow~\cite{benner2015two} to efficiently construct $\bar{\B}_r$ without explicitly forming $\bar{\V} \otimes \bar{\V}$. Using properties of tensor multiplication~\cite{kolda2009tensor}, a tensor representation of the quadratic ROM operator $\bar{\mathcal B} \in \real^{\bar r \times \bar r \times \bar r}$ can be constructed in three steps: i) compute $\bar{\mathcal B}_1 \in \real^{\bar{r} \times \nbar \times  \nbar}$ via $\bar{\B}_1^{(1)}=\bar{\V}^\top\bar{\B}^{(1)}$ where $\bar{\B}_1^{(1)}$ and $\bar{\B}^{(1)}$ are the mode-$1$ matricizations of $\bar{\mathcal B}_1$ and $\bar{\mathcal B}$, respectively, ii) compute $\bar{\mathcal B}_2 \in \real^{\bar{r} \times \bar{r} \times  \nbar}$ via $\bar{\B}_2^{(2)}=\bar{\V}^\top \bar{\B}_1^{(2)}$ where $\bar{\B}_2^{(2)}$ is the mode-$2$ matricization of $\bar{\mathcal B}_2$, and iii) compute $\bar{\mathcal B}_r \in \real^{\bar{r} \times \bar{r} \times \bar{r}}$ via $\bar{\B}_r^{(3)}=\bar{\V}^\top \bar{\B}_2^{(3)}$ where $\bar{\B}_r^{(3)}$ is the mode-$3$ matricization of $\bar{\mathcal B}_r$. After the final step, the tensor representation of the quadratic ROM operator $\bar{\mathcal B}_r \in \real^{\bar{r}\times \bar{r} \times \bar{r}}$ can be stored in the matrix form as $\bar{\B}_r\in \real^{\bar{r} \times \bar{r}^2}$ where $\bar{\B}_r$ is the mode-$1$ matricization of $\bar{\mathcal B}_r$.}}
\end{remark}
\section{Structure-preserving lifting via energy quadratization}
\label{sec:method}
In Section~\ref{sec:motivation} we first motivate the need for structure-preserving lifting by demonstrating how the standard lifting approach leads to quadratic ROMs that do not conserve the lifted FOM energy. In Section~\ref{sec:splifting} we present a novel energy-quadratization strategy for finding lifting transformations that lead to quadratic ROMs with quadratic energy in the lifted setting. We also show that under certain mild assumptions, the proposed quadratic ROMs conserve the lifted FOM energy exactly. In Section~\ref{sec:comp} we compare the accuracy and the energy error of quadratic ROMs obtained via the proposed structure-preserving lifting approach and the standard lifting approach.  In Section~\ref{sec:offline} we discuss the offline computational costs involved in deriving the lifted quadratic ROMs and compare it against the state-of-the-art structure-preserving DEIM method~\cite{pagliantini2023gradient}. 
\subsection{Motivational example}
\label{sec:motivation}
 A common strategy to find a lifting map is to introduce auxiliary variables for non-quadratic terms in the governing equations and then augment the FOM with evolution equations for the auxiliary variables. However, such an approach is not guaranteed to be structure-preserving for the nonlinear conservative FOMs considered herein. This is best understood with an example. Consider the one-dimensional sine-Gordon equation
\begin{equation}
\frac{\partial^2 \phi(x,t)}{\partial t^2}=\frac{\partial^2 \phi(x,t)}{\partial x^2}-\sin(\phi(x,t))
 \label{eq:sg_pde}
\end{equation}
that conserves the nonlinear energy
\begin{equation}  
\mathcal E[\phi(\x,t)]=\int  \left[ \frac{1}{2} \left(\frac{\partial \phi(x,t)}{\partial t}\right)^2 + \frac{1}{2} \left(\frac{\partial \phi(x,t)}{\partial x}\right)^2 +(1-\cos ( \phi(x,t) ) )\right] \ {\dd} x.
\label{eq:hamf_sg}
\end{equation}
By defining $q(x,t)=\phi(x,t)$ and $p(x,t)=\frac{\partial \phi(x,t)}{\partial t}$,  we rewrite the nonlinear conservative PDE~\eqref{eq:sg_pde} as
\begin{equation}
\frac{\partial q(x,t)}{\partial t} =p(x,t), \qquad \qquad
\frac{\partial p(x,t)}{\partial t} =\frac{\partial^2 q(x,t)}{\partial x^2} - \sin(q(x,t)).
\label{eq:sg_first}
\end{equation}
We discretize equation~\eqref{eq:sg_first} in space using $n$ equally spaced grid points to derive nonlinear conservative FOM 
\begin{equation} 
    \dot{\q}=\p, \quad
    \dot{\p}
    =\D\q - \sin(\q),
 \label{eq:sg_cons}
\end{equation}
which conserves the FOM energy
\begin{equation}
E( \q,\p)=\frac{1}{2}\p^\top\p - \frac{1}{2}\q^\top \D\q + \sum^n_{i=1} (1-\cos (q_i)),
\label{eq:sg_1d_fom}
\end{equation}
where $q_i:={\Rb{q(x_i,t)}}$ for $i=1,\ldots,n$. 

Given this nonlinear FOM with a sinusoidal nonlinearity, we introduce an auxiliary variable for the non-quadratic term, i.e., $\w_1=\sin(\q)$, which upon replacement, makes the original dynamics~\eqref{eq:sg_cons} quadratic. The evolution equation for the auxiliary variable is $ \dot{\w}_1=\cos(\q) \odot \p$, where the notation $\odot$ denotes the (Hadamard) component-wise product of two vectors. Since the evolution equation for the first auxiliary variable is non-polynomial in the new state $[\q^\top,\p^\top,\w_1^\top]^\top$, we introduce another auxiliary variable $\w_2=\cos(\q)$ that
transforms the nonlinear conservative FOM with $2n$ variables into a quadratic lifted FOM with $4n$ variables
\begin{align} 
    \dot{\q}&=\p, \nonumber \\
    \dot{\p}&=\D\q - \w_1, \label{eq:lifted_fom_motivational} \\
    \dot{\w}_1&=\w_2 \odot \p, \nonumber \\
    \dot{\w}_2&= - \w_1\odot \p. \nonumber
    \end{align}
Since no approximations have been made, this lifted FOM is equivalent to the nonlinear conservative FOM~\eqref{eq:sg_cons}. Therefore, it is straightforward to check that the lifted FOM conserves the total system energy in the lifted setting, i.e., $\frac{\text{d}}{{\dd}t}E_{\text{lift}}(\q,\p,\w_1,\w_2)=0$ where the lifted FOM energy is 
\begin{equation}
E_{ \text{lift}}(\q,\p,\w_1,\w_2)=\frac{1}{2}\p^\top\p - \frac{1}{2}\q^\top \D\q + \sum^n_{i=1} (1-w_{2,i} ). 
\label{eq:elift_standard}
\end{equation}
Thus, we transformed the $2n-$dimensional nonlinear conservative FOM into a $4n-$dimensional quadratic lifted FOM with an invariant of the motion. 

We note that the lifted FOM energy function in equation~\eqref{eq:elift_standard} possesses a linear term which is atypical in energy functions of conservative PDEs. In the subsequent analysis, we observe that this linear term contributes to a residual error term in the energy conservation equation \textit{at the reduced level}. To derive the quadratic ROM of the lifted FOM~\eqref{eq:lifted_fom_motivational}, we use a block-diagonal projection matrix $\bar{\V} \in \real^{4n \times 4r}$ that preserves the coupling structure
\begin{equation*}
\bar{\V} =\text{blkdiag}(\bPhi,\bPhi,{\Ra{\V_{1}}},{\Ra{\V_{2}}})\in \real^{4n\times 4r},
\end{equation*}
where $\bPhi$ is the PSD basis matrix computed using the cotangent lift algorithm, and ${\Ra{ \V_{1}}}$ and ${\Ra{ \V_{2}}}$ are the POD basis matrices that contain as columns POD basis vectors for $\w_1$ and $\w_2$, respectively. We approximate the lifted state $\ybar\approx \bar{\V} \ybar_r$ where $\ybar_r\in \real^{4r}$ is the reduced state with $r\ll n$. Substituting this approximation into the lifted FOM and using the standard POD Galerkin projection yields the quadratic ROM of the lifted system
\begin{align*} 
    \dot{\qhat}&=\phat, \\
    \dot{\phat}&=\Dhat\qhat - (\bPhi^\top{\Ra{ \V_{1}}})\what_1, \\
    \dot{\what}_1&={\Ra{ \V_{1}}}^\top \left( {\Ra{ \V_{2}}}\what_2 \odot \bPhi\phat\right), \\
        \dot{\what}_2&=-{\Ra{ \V_{2}}}^\top \left( {\Ra{ \V_{1}}}\what_1 \odot \bPhi\phat\right),
    \end{align*}
where $\Dhat:=(\bPhi^\top\D \bPhi) \in \real^{r \times r}$. Substituting the lifted FOM state approximations from the governing equations of the quadratic ROM into the time derivative of the lifted energy expression yields
\begin{align} 
    \frac{\dd}{\dd t}E_{\text{lift}}(\bPhi\qhat,\bPhi\phat,{\Ra{ \V_{1}}}\what_1,{\Ra{ \V_{2}}}\what_2)
    &= -  \what^\top_1\V_{1}^\top  \bPhi\phat  + \what^\top_1\V_{1}^\top \left({\Ra{ \V_{2}}}{\Ra{ \V_{2}}}^\top \right) \bPhi\phat \nonumber \\
    &=\what^\top_1{\Ra{ \V_{1}}}^\top \left({\Ra{ \V_{2}}}{\Ra{ \V_{2}}}^\top -\In \right) \bPhi\phat \nonumber \\
      & \neq 0,
       \label{eq:residual}
        \end{align}
 where  we use $\left({\Ra{ \V_{2}}}{\Ra{ \V_{2}}}^\top -\In \right)\neq 0$ for any $r < n$. This example illustrates that the standard lifting approach combined with POD model reduction is not guaranteed to yield energy-conserving ROMs.
\subsection{Structure-preserving lifting {\Rc{and model reduction}} via energy quadratization}
\label{sec:splifting}
Our goal is to transform the nonlinear conservative FOM~\eqref{eq:cons_fom} to an equivalent lifted quadratic system that after Galerkin projection leads to a structure-preserving quadratic ROM. Since lifting transformations for a general nonlinear system are not unique, there is an opportunity to identify a lifting transformation that leads to a quadratic lifted ROM that conserves the lifted FOM energy. We propose a structure-preserving lifting approach based on \textit{energy quadratization}, which is summarized in {\Rd{five}} steps:
\begin{enumerate}
\item Define an auxiliary variable that transforms the nonlinear FOM energy~\eqref{eq:ed_fom} into a quadratic lifted FOM energy in terms of the augmented state variables. {\Rd{Specifically, the first auxiliary variable $\w_1$ is defined by 
\begin{equation}
 \w_1^\top\w_1=\kappa^2 g(\q), 
\end{equation}
where $\kappa \in \real$ is a free scalar~parameter.}}
\item {\Rd{Compute the evolution equation for $\w_1$, and if the time evolution equations for $\{\q,\p,\w_1\}$ are not quadratic in terms of $\{\q,\p,\w_1\}$, then use knowledge about the nonlinear term $f_{\text{non}}(\q)$ in~\eqref{eq:cons_fom} to introduce the second auxiliary variable $\w_2=\frac{f_{\text{non}}(\q)}{\kapbar \w_1}$ with another free scalar parameter $\kapbar \in \real$.}}
\item {\Rd{Compute the evolution equation for the second auxiliary variable}}, and if it is nonlinear then introduce additional auxiliary variable(s) to ensure that dynamics for the lifted system (including evolution equations for all auxiliary variables) are quadratic.
\item Choose a basis that preserves the coupling structure of the lifted FOM in the reduced setting. 
\item {\Rc{Project the lifted FOM operators onto the reduced subspace to derive structure-preserving quadratic ROMs that conserve the lifted FOM energy exactly.}}
\end{enumerate}
We now present a theoretical result that shows that the proposed energy-quadratization strategy yields energy-conserving quadratic ROMs of nonlinear FOMs obtained via spatial discretization of conservative PDEs of the form~\eqref{eq:pde}. For this, we make the following two assumptions.

\begin{assumption}
\label{ass1}
The nonlinear potential energy component $g(q)$ in equation~\eqref{eq:ed_fom} is nonnegative.
\end{assumption}
\begin{assumption}
\label{ass2}
The proposed energy-quadratization strategy yields a quadratization of the nonlinear dynamics~\eqref{eq:cons_fom}{\Ra{ for nonlinear conservative PDEs with non-polynomial nonlinearities in $g(q)$}}.
\end{assumption}
There is over a century of theory on the existence of quadratizations for finite-dimensional systems, which we recently surveyed in~\cite{KraPog_survey_quadratization_2025}. This theory guarantees that any autonomous polynomial ODE can be quadratized with a finite number of auxiliary variables. However, no such theory exists for systems with non-polynomial nonlinearities. Therefore, we need
Assumption~\ref{ass2} which ensures the existence of a quadratization with a finite number of auxiliary variables. While we do not know of any theoretical result guaranteeing that Assumption 2 holds for arbitrary non-polynomial nonlinearities in $g(q)$, we have empirically observed that the proposed approach leads to a lifted quadratic FOM {\Rd{with a finite number of auxiliary variables for a variety of nonlinear conservative PDEs of the form~\eqref{eq:pde},  including those with non-polynomial nonlinearities such as exponential or sinusoidal terms.}}
\begin{theorem}
Consider a nonlinear conservative FOM~\eqref{eq:cons_fom} for which Assumptions~\ref{ass1}-\ref{ass2} hold. Then,  the proposed energy-quadratization strategy (Steps $1-3$ above) combined with POD model reduction {\Rd{(Steps $4-5$ above)}} yields quadratic ROMs that conserve the lifted FOM energy, i.e., $\frac{\dd}{\dd t}\left( E_{{\textup{lift}}} (\bar{\V}\bar{\y}_r) \right)=0$.
\label{theorem}
\end{theorem}
\begin{proof} 
Given a nonlinear conservative FOM of the form~\eqref{eq:cons_fom} with nonlinear potential energy component $g(q)$ satisfying the nonnegativity condition from Assumption~\ref{ass1}, we follow the proposed energy-quadratization strategy to {\Rd{introduce the first auxiliary variable $\w_1$ defined by $\w_1^\top\w_1=\kappa^2 g(\q)$, where $\kappa \in \real$ is a free scalar~parameter. This specific choice for the first auxiliary variable quadratizes the nonlinear component in the FOM energy, i.e., $\sum_{i=1}^n(g(q_i))=\frac{1}{\kappa^2}\w_1^\top\w_1$. If the first auxiliary variable does not quadratize the time evolution equations for $\{\q,\p,\w_1\}$ in terms of $\{\q,\p,\w_1\}$, we use knowledge about the nonlinear component of the vector field $f_{\text{non}}(\q)$ in~\eqref{eq:cons_fom} to introduce the second auxiliary variable $\w_2=\frac{f_{\text{non}}(\q)}{\kapbar \w_1}$ with another free scalar parameter $\kapbar \in \real$. This second auxiliary variable quadratizes the nonlinear component of the space-discretized vector field in terms of the first two auxiliary variables, i.e., $f_{\text{non}}(\q)=\kapbar \w_1\odot \w_2$.}} These first two auxiliary variables {\Rd{also}} ensure that the time evolution equations for $\{\q,\p,\w_1\}$ are quadratic in terms of $\{\q, \p, \w_1, \w_2\}$, independent of the form of the nonlinearity in $g(q)$. We note that Assumption~\ref{ass2} ensures that the proposed energy-quadratization strategy yields a finite-dimensional quadratic lifted FOM. Since the auxiliary variables are only a function of $q$, the dynamics for these auxiliary variables will have the form $\dot{\w}_i=(\cdots)\odot \p$ for $i=2, \cdots, k$. Combining this observation with Assumption 2, the resulting quadratic FOM can be written as
\begin{align}
\label{eq:gen_lift_fom}
    \dot{\q}&=\p, \nonumber\\
    \dot{\p}&=\D\q - {\Rd{\kapbar}} \w_1\odot \w_2,\nonumber \\
    \dot{\w}_1&=\frac{{\Rd{\kappa^2\kapbar}}}{2}\w_2 \odot \p,\nonumber \\
    \dot{\w}_2&=\left(\alpha_2\q +\sum_{i=1}^{k}\alpha_{2,i}\w_i \right) \odot \p, \\
    & \qquad \qquad \vdots  \nonumber\\
    \dot{\w}_{k}&= \left(\alpha_k\q +\sum_{i=1}^{k}\alpha_{k,i}\w_i \right) \odot \p,\nonumber
\end{align}
where $\alpha_{2},\ldots,\alpha_k$ and $\alpha_{i,1}, \ldots, \alpha_{i,k}$ for $i=2,\ldots,k$ are real-valued constant coefficients such that the constants in the set $\pmb \alpha_{i}:=\{\alpha_{i},\alpha_{i,1}, \ldots, \alpha_{i,k}\}$ can not be all zero for $i=2,\ldots,k$.
This lifted FOM possesses a quadratic invariant in the lifted variables, i.e., 
\begin{equation}
\label{eq:gen_lift_energy}
E_{\text{lift}}(\q,\p,\w_1, \cdots,\w_{k})=\frac{1}{2}\p^\top\p - \frac{1}{2}\q^\top \D\q + {\Rd{\frac{1}{\kappa^2}}}\w_1^\top\w_1. 
\end{equation}
{\Ra{To construct the basis matrix for the augmented state vector in the lifted quadratic FOM~\eqref{eq:gen_lift_fom} without ever solving the lifted system, we first build position and momentum snapshot data matrices $\Q$ and $\Pp$ by simulating the nonlinear conservative FOM~\eqref{eq:cons_fom}. We then use the lifting transformations to construct the lifted snapshot data matrix $\mathbf W_i$ for each lifted variable.}}
We use a block-diagonal basis matrix $\bar{\V} \in \real^{\nbar\times \bar{r}}$ that preserves the coupling structure
\begin{equation*}
\bar{\V} =\text{blkdiag}(\bPhi,\bPhi,{\Ra{ \V_{1}}}, \cdots,{\Ra{ \V_{k}}})\in \real^{\nbar\times \bar{r}},
\end{equation*}
where $\bPhi$ is the PSD basis matrix computed using the cotangent lift algorithm, and $ {\Ra{ \V_{i}}}$ is the POD basis matrix that contains as columns the POD basis vectors for $\w_i$ for $i=1, \ldots, k$. A standard POD Galerkin projection yields
\begin{align} 
\label{eq:gen_lift_rom}
    \dot{\qhat}&=\phat,\nonumber \\
    \dot{\phat}&=\Dhat\qhat - {\Rd{\kapbar}}\bPhi^\top\left({\Ra{ \V_{1}}}\what_1 \odot {\Ra{ \V_{2}}}\what_2\right),\nonumber \\
    \dot{\what}_1&=\frac{{\Rd{\kappa^2\kapbar}}}{2}{\Ra{ \V_{1}}}^\top \left( {\Ra{ \V_{2}}}\what_2 \odot \bPhi\phat\right), \nonumber\\
        \dot{\what}_2&={\Ra{ \V_{2}}}^\top \left( \left( \alpha_2\bPhi\qhat +\sum_{i=1}^{k}\alpha_{2,i}{\Ra{ \V_{i}}}\what_i\right) \odot \bPhi \phat \right), \\
        & \qquad \qquad \qquad \vdots \nonumber \\
         \dot{\what}_{k}&={\Ra{ \V_{k}}}^\top \left(  \left( \alpha_k\bPhi\qhat +\sum_{i=1}^{k}\alpha_{k,i}{\Ra{ \V_{i}}}\what_i\right)\odot \bPhi \phat \right)\nonumber.
    \end{align}
    where $\Dhat=(\bPhi^\top\D \bPhi)$. We compute the time-derivative of the lifted FOM energy
\begin{align*} 
    \frac{\dd}{\dd t}E_{\text{lift}}(\bPhi\qhat,\bPhi\phat,{\Ra{ \V_{1}}}\what_1, \cdots,{\Ra{ \V_{k}}}\what_{k})&=\phat^\top\phatdot - (\Dhat\qhat)^\top\qhatdot +{\Rd{\frac{2}{\kappa^2}}} \what_1^\top\dot{\what}_1 \\
    &=-{\Rd{\kapbar}}\phat^\top\bPhi^\top\left( {\Ra{ \V_{1}}}\what_1 \odot  {\Ra{ \V_{2}}}\what_2\right) +{\Rd{\kapbar}}   \what_1^\top{\Ra{ \V_{1}}}^\top \left( {\Ra{ \V_{2}}}\what_2 \odot \bPhi\phat\right)\\
    &= 0. 
        \end{align*}
Thus, the quadratic ROM~\eqref{eq:gen_lift_rom}  obtained via the proposed energy-quadratization strategy conserves the lifted FOM energy~\eqref{eq:gen_lift_energy}.
\end{proof}
{\Ra{While we need to form the lifted FOM operators in~\eqref{eq:gen_lift_fom} to derive the structure-preserving quadratic ROM~\eqref{eq:gen_lift_rom}, we do not numerically solve the lifted FOM to construct the lifted snapshot data matrix. Instead, we use the lifting transformation $\tau_i$ for the $i$th auxiliary variable $\w_i$ to compute the lifted state at time $t_j$ as $\w_{i,j}=\tau_{i}(\q_j)$ and construct the corresponding lifted snapshot data matrix $\mathbf{W}_i=[\w_{i,1},\cdots,\w_{i,K}]\in \real^{n \times K}$ for $i=1,\cdots,k$.}}
\begin{remark}
The idea of introducing auxiliary variables to quadratize the system energy in terms of the augmented state variables has found success in the field of structure-preserving time integrators. Building on the Lagrange multiplier idea in~\cite{badia2011finite,guillen2013linear}, the authors in~\cite{yang2016linear} proposed the so called invariant energy quadratization (IEQ) approach for deriving computationally efficient energy-conserving integrators for phase field modeling applications. The recent successes of energy-conserving integrators based on the IEQ strategy and its modifications~\cite{shen2018scalar,shen2018convergence} have provided an alternative to structure-preserving time integrators that leverage the Hamiltonian structure. In a similar vein, the energy quadratization strategy proposed in this work can be viewed as providing an alternative to the symplectic model reduction approach in the context of structure-preserving model reduction of nonlinear conservative PDEs.
\end{remark}
\subsection{Comparison of structure-preserving and non-structure-preserving lifting on the sine-Gordon equation}
\label{sec:comp}
\subsubsection{Analysis}
To illustrate the theoretical result presented in Section~\ref{sec:splifting}, we revisit the sine-Gordon equation example from Section~\ref{sec:motivation}. {\Ra{ Since the nonlinear term $g(q_i)=1-\cos(q_i)$ in the nonlinear FOM energy~\eqref{eq:sg_1d_fom} satisfies the nonnegativity condition, we {\Rd{solve $ \w_1^\top\w_1=\kappa^2 g(\q)$ with $g(q)=1-\cos(q)$ and choose $\kappa=1/\sqrt{2}$ to obtain the first auxiliary variable $\w_1=\frac{1}{\sqrt{2}}\sqrt{1-\cos(\q)}=\sin(\q/2)$}}. Consequently, the nonlinear potential energy term in $\q$ in~\eqref{eq:sg_1d_fom} can be replaced by a quadratic term in terms of $\w_1$, i.e., $\sum^n_{i=1} (1-\cos (q_i))=2\w_1^\top\w_1$.}} The evolution equation for the auxiliary variable is $\dot{\w}_1=\frac{1}{2}\cos(\q/2)\odot \p$, which is non-polynomial in the augmented state $[\q^\top,\p^\top,\w_1^\top]^\top$. We introduce another auxiliary variable $\w_2={\Rd{\frac{\sin(\q)}{2\sin(\q/2)}}}=\cos(\q/2)$ {\Rd{with $\kapbar=2$}} that transforms the nonlinear conservative FOM of dimension $2n$ into the $4n-$dimensional quadratic lifted FOM
\begin{align}
    \dot{\q}&=\p, \nonumber\\
    \dot{\p}&=\D\q - 2\w_1\odot \w_2, \\
    \dot{\w}_1&=\frac{1}{2}\w_2 \odot \p,\nonumber\\
    \dot{\w}_2&= -\frac{1}{2} \w_1\odot \p, \nonumber
\end{align}
with a quadratic invariant in the lifted variables
\begin{equation}
E_{\text{lift}}(\q,\p,\w_1,\w_2)=\frac{1}{2}\p^\top\p - \frac{1}{2}\q^\top \D\q + 2\w_1^\top\w_1.
\label{eq:elift_sg}
\end{equation} 
We use a block-diagonal basis matrix $\bar{\V}  \in \real^{4n \times 4r}$ that preserves the coupling structure
\begin{equation*}
\bar{\V} =\text{blkdiag}(\bPhi,\bPhi,{\Ra{ \V_{1}}},{\Ra{ \V_{2}}})\in \real^{4n\times 4r},
\end{equation*}
where $\bPhi$ is the PSD basis matrix computed using the cotangent lift algorithm, and $ {\Ra{ \V_{1}}}$ and ${\Ra{ \V_{2}}}$ are the POD basis matrices that contain as columns POD basis vectors for $\w_1$ and $\w_2$, respectively. A standard POD Galerkin projection yields
\begin{align*} 
    \dot{\qhat}&=\phat, \\
    \dot{\phat}&=\Dhat\qhat - 2\bPhi^\top\left({\Ra{ \V_{1}}}\what_1 \odot {\Ra{ \V_{2}}}\what_2\right), \\
    \dot{\what}_1&=\frac{1}{2}{\Ra{ \V_{1}}}^\top \left( {\Ra{ \V_{2}}}\what_2 \odot \bPhi\phat\right), \\
        \dot{\what}_2&=-\frac{1}{2}{\Ra{ \V_{2}}}^\top \left( {\Ra{ \V_{1}}}\what_1 \odot \bPhi \phat \right),
    \end{align*}
    where $\Dhat=\bPhi^\top\D\bPhi \in \real^{r \times r}$. We compute the time-derivative of the lifted FOM energy
\begin{align} 
    \frac{\dd}{\dd t}E_{\text{lift}}(\bPhi\qhat,\bPhi\phat,{\Ra{ \V_{1}}}\what_1,{\Ra{ \V_{2}}}\what_2)&=\phat^\top\phatdot - (\Dhat\qhat)^\top\qhatdot +4 \what_1^\top\dot{\what}_1 \nonumber \\
    &=-2\phat^\top\bPhi^\top\left( {\Ra{ \V_{1}}}\what_1 \odot  {\Ra{ \V_{2}}}\what_2\right) +2   \what_1^\top{\Ra{ \V_{1}}}^\top \left( {\Ra{ \V_{2}}}\what_2 \odot \bPhi\phat\right) \label{eq:fom_energy_zero} \\
    &= 0. \nonumber
        \end{align}
In contrast to the quadratic ROM derived via the standard lifting approach in Section~\ref{sec:motivation}, the energy-quadratization strategy leads to a quadratic lifted ROM that conserves the lifted FOM energy exactly.
\begin{figure}[tbp]
\small
\captionsetup[subfigure]{oneside,margin={1.8cm,0 cm}}
\begin{subfigure}{.45\textwidth}
       \setlength\fheight{6 cm}
        \setlength\fwidth{\textwidth}
%
%
\begin{tikzpicture}

\begin{axis}[%
width=0.951\fheight,
height=0.536\fheight,
at={(0\fheight,0\fheight)},
scale only axis,
xmin=2,
xmax=20,
xlabel style={font=\color{white!15!black}},
xlabel={Reduced dimension $2r$},
ymode=log,
ymin=0.0001,
ymax=10,
yminorticks=true,
ylabel style={font=\color{white!15!black}},
ylabel={Relative state error in $q$},
axis background/.style={fill=white},
xmajorgrids,
ymajorgrids,
yminorgrids,
legend style={at={(0.1,1.25)}, anchor=south west, legend cell align=left, align=left, draw=white!15!black}
]
\addplot [color=green, dashed, line width=2.0pt, mark size=3.0pt, mark=triangle, mark options={solid, green}]
  table[row sep=crcr]{%
2	1.40832315143151\\
4	0.97600434536982\\
6	0.263596757633243\\
8	0.017713790417063\\
10	0.00291171191843248\\
12	0.00161987744907376\\
14	0.00124060868233179\\
16	0.00129756228360735\\
18	0.000958489512221317\\
20	0.000956566086761161\\
};
\addlegendentry{structure-preserving lifting}

\addplot [color=blue, dotted, line width=2.0pt, mark size=3.0pt, mark=o, mark options={solid, blue}]
  table[row sep=crcr]{%
2	1.37951883829635\\
4	1.68639435134735\\
6	0.63371749779811\\
8	0.994044036820677\\
10	0.0646022344944805\\
12	0.00725660231724073\\
14	0.00174385118773281\\
16	0.00125466119894856\\
18	0.00108365823271292\\
20	0.000946906644684891\\
};
\addlegendentry{standard lifting}

\end{axis}
\end{tikzpicture}%
\caption{Relative state error}
\label{fig:comp_state}
    \end{subfigure}
    \hspace{0.2cm}
    \begin{subfigure}{.45\textwidth}
           \setlength\fheight{6 cm}
           \setlength\fwidth{\textwidth}
\input{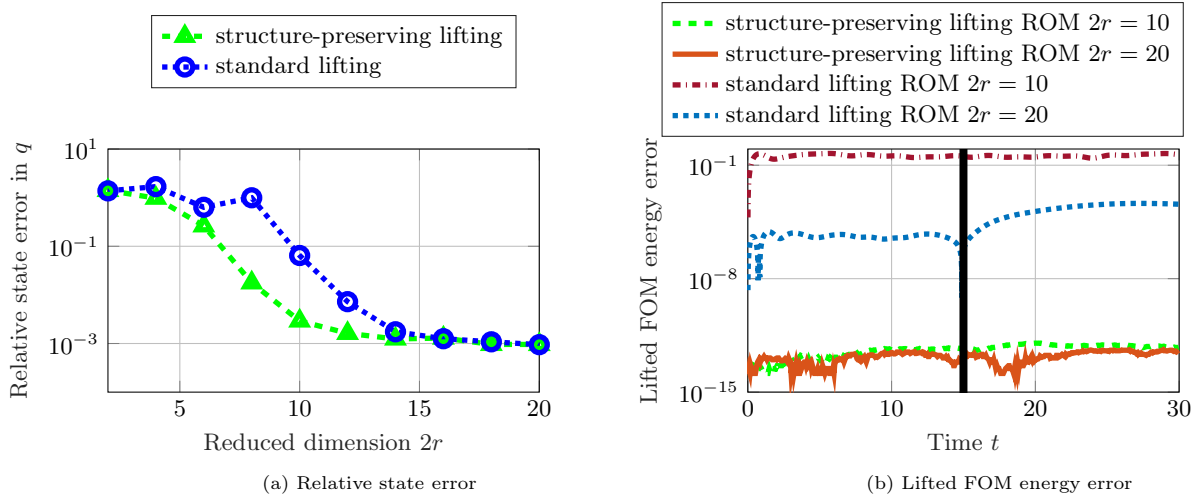}
\caption{Lifted FOM energy error}
\label{fig:comp_energy}
    \end{subfigure}
\caption{Comparative study for the one-dimensional sine-Gordon equation. Plot (a) shows that the proposed energy-quadratization strategy yields quadratic ROMs that achieve lower relative state error than the quadratic ROMs obtained using the standard lifting approach. The energy error comparison in plot (b) demonstrates that the quadratic ROMs derived using proposed structure-preserving lifting approach conserve the lifted FOM energy exactly.}
 \label{fig:comp}
\end{figure}
\subsubsection{Numerical simulations}
We demonstrate the advantages of using the proposed structure-preserving lifting approach over the standard lifting approach through a numerical study. We build a training dataset by simulating the nonlinear FOM~\eqref{eq:sg_cons} with $n=200$ from $t=0$ to $t=15$. We then compute a block-diagonal POD basis matrix for the lifted FOM state and then derive quadratic ROMs for both lifting approaches. To demonstrate the energy-conserving nature of the quadratic ROMs derived using the structure-preserving lifting approach, we numerically integrate both the quadratic ROMs using the implicit midpoint method which is a fully implicit energy-conserving integrator for quadratic vector fields.

We compare the relative state error and the lifted FOM energy error for the two lifting approaches in Figure~\ref{fig:comp}. The relative state error comparison in Figure~\ref{fig:comp_state} shows that the quadratic ROMs based on the structure-preserving lifting approach achieve higher accuracy than the standard lifting approach for $4\leq2r\leq14$. For ROMs of size $2r>14$, we observe that both lifting approaches yield similar relative state error which is a consequence of the fact that the residual energy error term in~\eqref{eq:residual} decreases as we increase the reduced dimension $r$. 

The energy-conserving nature of the energy-quadratization strategy is corroborated in Figure~\ref{fig:comp_energy} where quadratic ROMs of size $2r=10$ and $2r=20$ based on the structure-preserving lifting approach conserve the lifted FOM energy error close to machine precision. The quadratic ROMs based on the standard lifting approach, on the other hand, demonstrate substantially higher lifted FOM energy error despite using an energy-conserving numerical integrator. This provides numerical evidence of the fact that the standard lifting approach yields quadratic ROMs that do not conserve the lifted FOM energy. 
\subsection{Offline computational costs: Comparing structure-preserving lifting with spDEIM}
\label{sec:offline}
In this section we highlight the computational advantages of the proposed nonlinear model reduction approach based on structure-preserving lifting in the offline phase by comparing its offline computational costs with the spDEIM approach summarized in Section~\ref{sec:spmor}. For this comparison, we only focus on components that incur additional computational costs compared to the baseline PSD ROM (without spDEIM approximation).
\begin{enumerate}
\item \textit{spDEIM:} Compared to the baseline PSD ROM, the additional computational cost in the offline phase for this hyper-reduction method comes from the construction of the DEIM basis. {\Ra{Given $K$ snapshots from the nonlinear FOM simulation, }} the computation of the DEIM basis in the spDEIM approach requires singular vectors of a Jacobian snapshot data matrix of size $n \times rK$. The construction of this Jacobian snapshot data matrix and its SVD makes the offline phase computationally expensive, especially for nonlinear conservative FOMs with dimension $n >10^4$. Moreover, for parametric problems with $M$ training parameters, the size of the Jacobian snapshot data matrix becomes $n \times rMK$ which further increases the computational cost in the offline phase for spDEIM.  \\

\item \textit{Structure-preserving lifting:} The proposed structure-preserving lifting approach has two components that incur additional computational costs. The first component is building the POD basis vectors for the auxiliary variables that are used to derive the quadratic ROM operators. Consider a problem with FOM snapshot data matrices $\Q \in \real^{n \times K}$ and $\Pp \in \real^{n \times K}$ where $K$ is the total number of snapshots. We then apply the structure-preserving lifting map to each column of the FOM snapshot data matrix $\Q$ to obtain lifted snapshot data and build the snapshot data matrix for each of the $k$ auxiliary variables. Thus, building the block-diagonal POD basis matrix for the lifted state in the proposed approach requires building a lifted snapshot matrix $\textbf W_i \in \real^{n \times K}$ and then computing its SVD for $i=1, \cdots, k$.  The second component is the projection of the lifted FOM operators onto the reduced subspace. The first step in the three-step procedure summarized in Section~\ref{sec:lifting} involves the construction of a tensor with $\bar{r}\times \nbar^2$ entries which can be computationally prohibitive for high-dimensional FOMs. In this work, we exploit the highly sparse nature of the lifted FOM operators with only $\mathcal{O}(n)$ nonzero entries to compute $\bar{\V}^\top\B$ by multiplying rows of $\bar{\V}^\top$ only with columns that contain nonzero terms. This exploitation of the sparsity of the lifted FOM operator helps reduce the computational complexity from $\mathcal{O}(\bar{r}\times \nbar^2)$ to $\mathcal{O}(\bar{r}\times \nbar)$ which makes the computational cost of constructing the quadratic ROM operators in this second component negligible compared to the first component.
\end{enumerate}
In summary, the main additional computational cost in the offline phase for the structure-preserving lifting approach comes from building and computing the SVD of the lifted snapshot data matrices for $k$ auxiliary variables. Since most of the nonlinear conservative PDEs in science and engineering applications require only $k=1$ or $k=2$ auxiliary variables, the added computational costs in the offline stage for the structure-preserving lifting approach is substantially lower than the added computational costs for the spDEIM approach.

\section{Numerical results}
\label{sec:numerical}
In this section, we study the numerical performance of the proposed structure-preserving lifting approach for four nonlinear conservative PDEs with increasing level of complexity. {\Rc{Section~\ref{sec:practical}}} provides practical details about the numerical integrators used for simulating the ROMs. {\Rc{Section~\ref{sec:error} defines the error measures used in this section}}.  In Section~\ref{sec:exp}, we study the one-dimensional nonlinear wave equation with exponential nonlinearity. In Section~\ref{sec:sg_2d} we consider the two-dimensional sine-Gordon equation.  In Section~\ref{sec:kg_2d}, we demonstrate the proposed approach on the two-dimensional Klein-Gordon equation with parametric dependence. Finally, in Section~\ref{sec:kgz}, we consider a nonlinear conservative FOM with $960{,}000$ degrees of freedom to derive structure-preserving ROMs for the two-dimensional Klein-Gordon-Zakharov equations, a system of coupled PDEs with a nonlinear conservation law. This example from plasma physics demonstrates the wider applicability of the proposed approach for conservative PDEs that are not canonical Hamiltonian PDEs. 

All numerical experiments in Section~\ref{sec:exp}, Section~\ref{sec:sg_2d}, and  Section~\ref{sec:kg_2d} are implemented in MATLAB 2022a on a quad-core Intel i7 processor with 2.3 GHz and 32 GB memory. The numerical experiments for the two-dimensional Klein-Gordon-Zakharov equations in Section~\ref{sec:kgz} are implemented in MATLAB 2022b using compute nodes of the Triton Shared Computing Cluster~\cite{san2022triton} equipped with $8$ processing cores of Intel Xeon Platinum 64-core CPU at 2.9 GHz and 1 TB memory.
\subsection{Practical details about numerical time integrators}
\label{sec:practical}
For nonlinear ROMs obtained via PSD (with or without  spDEIM), we use the implicit midpoint rule for numerical time integration. The corresponding time-marching equations are
\begin{equation*}
    \frac{\yhat^{k+1}-\yhat^k}{\Delta t}=\widehat{\f}\left(\frac{\yhat^k + \yhat^{k+1}}{2}\right).
\end{equation*}
The implicit midpoint rule is a second-order symplectic integrator that exhibits bounded energy error for nonlinear Hamiltonian systems~\cite{hairer2006geometric}.  The magnitude of the maximum FOM energy error for ROM trajectories obtained via the implicit midpoint rule depends on the size of the fixed time step $\Delta t$. While it is possible to achieve exact FOM energy conservation using energy-conserving integrators like the average vector field method~\cite{celledoni2012preserving}, we prefer the implicit midpoint rule {\Rb{due to ease of implementation and marginally higher computational efficiency in the online stage.}} 

For quadratic ROMs obtained via the proposed structure-preserving lifting approach, we use Kahan's method for numerical time integration. Kahan's method is a second-order structure-preserving numerical integrator for quadratic vector fields, see~\cite{celledoni2012geometric} for more details about the geometric properties of Kahan's method. The corresponding time-marching equations are
\begin{equation}
    \frac{\yhat^{k+1}-\yhat^{k}}{\dt}=\widehat{\A} \left( \frac{\yhat^{k+1}+\yhat^{k}}{2}\right) + \widehat{\B} \left( \frac{\left( \yhat^{k+1} \otimes \yhat^{k} \right) + \left( \yhat^{k} \otimes \yhat^{k+1} \right) }{2}\right),
\end{equation} 
where $\Delta t$ is the fixed time step. In contrast to the implicit midpoint method used for integrating the quadratic ROMs in the comparative study in Section~\ref{sec:comp},  Kahan's method is not energy-conserving for quadratic vector fields, and as a result, the quadratic ROM trajectories obtained with Kahan's method do not conserve the lifted FOM energy exactly. Nevertheless, Kahan's method exploits the quadratic nature of dynamics to integrate the proposed quadratic ROMs in a computationally efficient manner while also providing lifted FOM state approximations with bounded lifted FOM energy error. 
\subsection{Reported error measures}
\label{sec:error}
The \emph{relative state error} in $q$ is computed in the entire training or testing intervals as  
\begin{equation}\label{eq:err_state}
 \text{Relative state error in } q =\frac{\lVert \Q- \bPhi\widehat{\Q} \rVert^2_F}{\lVert \Q \rVert^2_F},
\end{equation}
where $\widehat{\Q}=[\qhat_1, \cdots, \qhat_K] \in \real^{r \times K}$ is the ROM snapshot data obtained from the ROM simulations and $\bPhi\widehat{\Q}\in \real^{n \times K}$ is the approximation of the FOM snapshot data $\Q$.  We measure the common cost/accuracy tradeoff for ROMs using the \emph{efficacy} metric which is computed as
\begin{equation}\label{eq:eff}
{\text{Efficacy}}= \frac{1}{\text{relative state error in training data regime} \times \text{wall-clock time in seconds}} \ .
\end{equation}
In comparisons based on this metric, the model reduction approach with higher efficacy is considered advantageous. We note that trivial ROM solutions are excluded, and we only compute efficacy once a certain threshold of accuracy is achieved.  Finally, the \emph{FOM energy error} is computed as follows:
\begin{equation}\label{eq:err_fom}
  \text{FOM energy error}=  \left| E(\bPhi\qhat(t),\bPhi \phat(t)) -E(\bPhi\qhat(0),\bPhi \phat(0)) \right|,
\end{equation}
where $E(\bPhi\qhat(t),\bPhi \phat(t))$ is the FOM energy approximation obtained by evaluating the space-discretized nonlinear FOM energy $E_d$ for the FOM state approximation at time $t$. 
\subsection{Nonlinear wave equation with exponential nonlinearity}
\label{sec:exp}
The one-dimensional nonlinear wave equation with exponential nonlinearity considered here arises from the Johnson–Mehl–Avrami–Kolmogorov theory~\cite{avrami1939kinetics, johnson1939reaction, kolmogorov1937statistical} of nucleation and growth reactions for modeling the kinetics of phase change. Building on this theory from the 1930s, the author in~\cite{cahn1995time} proposed the time cone method for deriving an integral equation that characterized the nucleation and growth phenomenon of nuclei in a finite domain. The authors in~\cite{liu2014multiple} reduced the integral equation from the time cone method to a nonlinear hyperbolic equation where the exponential nonlinearity models the physical fact that the nucleation rate decreases with an increase in the number of nuclei, see~\cite{wang2015energy} for more details.
\subsubsection{PDE formulation and the corresponding nonlinear conservative FOM}
We consider the FOM setup from~\cite{li2020linearly}. Let $\Omega=(0,\pi) \subset \real$ be the spatial domain and consider the one-dimensional nonlinear wave equation
\begin{equation}
\frac{\partial^2 \phi(x,t)}{\partial t^2}=\frac{\partial^2 \phi(x,t)}{\partial x^2}+\exp(-\phi(x,t)),
\label{eq:exp_wave}
\end{equation}
with state $\phi(x,t)$ at spatial location $x \in \Omega$ and time $t \in (0,T]$. We consider homogenous Dirichelet boundary conditions
\begin{equation*}
\phi(0,t)=\phi(\pi,t)=0,
\end{equation*}
for $t \in (0,T]$. The initial conditions are
\begin{equation*}
\phi(x,0)=0.5x(\pi-x), \qquad \frac{\partial \phi}{\partial t} (x,0)=0, \qquad x \in [0,\pi].
\end{equation*}
We define $q(x,t):=\phi(x,t)$ and $p(x,t):=\frac{\partial \phi(x,t)}{\partial t}$ to
 rewrite the nonlinear wave equation with exponential nonlinearity~\eqref{eq:exp_wave} in first-order form  
\begin{equation}
\frac{\partial q (x,t)}{\partial t}=p(x,t), \qquad \frac{\partial p (x,t)}{\partial t}=\frac{\partial^2 q (x,t)}{\partial x^2} + \exp(-q (x,t)).
\end{equation}

We discretize the system of first-order PDEs using $n=200$ equally spaced grid points to derive the nonlinear conservative FOM
\begin{equation*} 
    \dot{\q}=\p, \qquad
    \dot{\p} =\D\q + \exp (-\q),
    \label{eq:exp_fom}
\end{equation*}
where $\D=\D^\top$ denotes the {\Ra{symmetric discrete Jacobian}} and the vector $ \exp (-\q) \in  \real^n$ contains as components the entry-wise exponential of the negative of the discretized state vector $\q$. The nonlinear conservative FOM conserves the space-discretized energy 
\begin{equation}
E(\q,\p)=\frac{1}{2} \p^\top \p -\frac{1}{2}\q^\top\D \q + \sum_{i=1}^n \left( \exp(-q_i) \right).
\label{eq:exp_ener_fom}
\end{equation}
\subsubsection{Structure-preserving lifting based on energy quadratization}
Based on {\Ra{the nonnegative nonlinear term $g(q)=\exp(-q)$}} in the nonlinear FOM energy expression~\eqref{eq:exp_ener_fom}, we {\Rd{solve $ \w_1^\top\w_1=\kappa^2 g(\q)$ with $g(q)=\exp(-q)$ and choose $\kappa=1$ to obtain the first auxiliary variable $\w_1=\exp(-\q/2)$, which}} quadratizes the nonlinear FOM energy in the lifted variables. {\Rd{Since the time evolution equations for $\{\q,\p,\w_1\}$ are quadratic in terms of $\{\q,\p,\w_1 \}$ in this example, we do not need to introduce additional auxiliary variables.}} The resulting lifted FOM  
\begin{align*} 
    \dot{\q}&=\p, \\
    \dot{\p}&=\D \q + {\Rd{\w_1}} \odot {\Rd{\w_1}}, \\
    {\Rd{\dot{\w}_1}}&=-\frac{1}{2}{\Rd{\w_1}} \odot \p ,
  \end{align*}
conserves the lifted FOM energy $E_{\text{lift}}(\q,\p,{\Rd{\w_1}})=\frac{1}{2}\p^\top\p - \frac{1}{2}\q^\top \D\q +{\Rd{\w_1}}^\top{\Rd{\w_1}} $. {\Ra{Thus, Assumptions 1 \& 2 of Theorem~\ref{theorem} are satisfied.}} For Galerkin projection, we consider a block-diagonal basis matrix of the form
\begin{equation*}
\bar{\V} =\text{blkdiag}(\bPhi,\bPhi,{\Rd{\V_1}})\in \real^{3n\times 3r},
\end{equation*}
where the PSD basis matrix $\bPhi \in \real^{n\times r}$ for $\q$ and $\p$ is computed using the cotangent lift algorithm and the POD basis matrix ${\Rd{\V_1}} \in \real^{n \times r}$ for the auxiliary variable ${\Rd{\w_1}}$ is computed via SVD of the lifted snapshot data matrix ${\Rd{\textbf W_1}} \in \real^{n \times r}$. The resulting POD-Galerkin ROM of the lifted FOM is
\begin{align*} 
    \dot{\qhat}&=\phat, \\
    \dot{\phat}&=\Dhat\qhat +\bPhi^\top\left(  {\Rd{\V_1\what_1}}  \odot   {\Rd{\V_1\what_1}} \right) , \\
    {\Rd{\dot{\what}_1}}&=-\frac{1}{2}{\Rd{\V_1}}^\top \left(  {\Rd{\V_1\what_1}}  \odot \bPhi\phat\right),
    \end{align*}
    where $\Dhat:=\bPhi^\top\D\bPhi \in \real^{r \times r}$. We substitute the lifted FOM state approximation into the lifted FOM energy expression and compute its time derivative
    \begin{align*} 
    \frac{\dd}{{\dd}t}E_{\text{lift}}(\bPhi\qhat,\bPhi\phat,{\Rd{\V_1\what_1}})&=\phat^\top(\Dhat\qhat+\bPhi^\top\left( {\Rd{\V_1\what_1}} \odot  {\Rd{\V_1\what_1}}\right) ) - \qhat^\top \Dhat\phat 
    +  2{\Rd{\what_1}}^\top\underbrace{\left( {\Rd{\V_1}}^\top {\Rd{\V_1}}\right)}_{\textbf{I}_r}{\Rd{\dot{\what}_1}} =0.
        \end{align*}        
Thus, the proposed energy-quadratization lifting strategy combined with POD-Galerkin yields a $3r-$dimensional quadratic ROM that is guaranteed to conserve the lifted FOM energy, {\Ra{see also Theorem~\ref{theorem}.}}
\subsubsection{Numerical results}
A key motivation for structure-preserving model reduction for nonlinear conservative PDEs is to derive stable ROMs that can provide accurate predictions outside the training dataset. In this study, the training dataset is built by integrating the FOM until time $t=10$ using the symplectic midpoint rule with $\dt=0.005$. For test data, we consider FOM snapshots from $t=10$ to $t=100$ to study the time extrapolation ability of quadratic ROMs derived via the proposed structure-preserving lifting approach. {\Rb{We simulate the ROMs until $t=100$, which is 900\% outside the training time interval to assess the proposed method's ability to capture the quasi-periodic behavior. }}

In Figure~\ref{fig:exp_1d_state}, we compare the relative state error~\eqref{eq:err_state} {\Ra{in $q$ and $p$}} for the training and test data regimes for different values of the reduced dimension. For comparison against the quadratic ROMs obtained via the proposed structure-preserving lifting approach, we consider structure-preserving PSD ROMs (without hyper-reduction) and spDEIM ROMs with $r_{\text{spDEIM}}=r$ and {\Rd{$r_{\text{spDEIM}}=2r$}}.  The comparison plots {\Ra{for $q$}} in Figure~\ref{fig:exp_state_train} and Figure~\ref{fig:exp_state_test} show that all three approaches yields ROMs with similar accuracy in both training and test data regimes. {\Rd{The relative state error comparison for $p$ in Figure~\ref{fig:exp_state_train} and Figure~\ref{fig:exp_state_test} shows that all ROMs provide inaccurate predictions, primarily due to the cotangent lift basis failing to provide accurate state approximations for $p$ in both training and test data regimes.  The quadratic ROMs obtained via the proposed structure-preserving lifting approach exhibit higher state error than the structure-preserving PSD ROMs (without hyper-reduction) and spDEIM ROMs.}}

We compare the computational efficiency of the structure-preserving lifting approach against the spDEIM approach with $r_{\text{spDEIM}}=r$ {\Rd{and $r_{\text{spDEIM}}=2r$}} through efficacy plots in Figure~\ref{fig:exp_1d_eff}. The $3r-$dimensional quadratic ROMs obtained via structure-preserving lifting achieve higher efficacy than the $2r-$dimensional nonlinear Hamiltonian ROMs obtained via spDEIM. Thus, compared to spDEIM, the proposed approach achieves similar state error performance at a lower computational cost in the online stage for this example. The FOM energy error plots in Figure~\ref{fig:exp_1d_ener} show that the all three approaches achieve bounded FOM energy error, {\Rd{with the PSD ROM achieving the lowest energy error. The quadratic ROM obtained via structure-preserving lifting achieves substantially lower energy error than the spDEIM ROMs with $r_{\text{spDEIM}}=r$ and $r_{\text{spDEIM}}=2r$ in both training and test data regimes.}} This ability to provide accurate and stable predictions 900\% outside the training data regime highlights a core advantage of the structure-preserving quadratic ROMs obtained via the structure-preserving lifting approach.
\begin{figure}[tbp]
\small
\captionsetup[subfigure]{oneside,margin={-2cm,0 cm}}
\begin{subfigure}{.23\textwidth}
       \setlength\fheight{3 cm}
        \setlength\fwidth{\textwidth}
\raisebox{8mm}{
%
%
\definecolor{mycolor1}{rgb}{1.00000,0.00000,1.00000}%
\begin{tikzpicture}

\begin{axis}[%
width=0.683\fheight,
height=0.65\fheight,
at={(0\fheight,0\fheight)},
scale only axis,
xmin=4,
xmax=20,
xlabel style={font=\color{white!15!black}},
xlabel={Reduced dimension $2r$},
ymode=log,
ymin=0.001,
ymax=0.1,
yminorticks=true,
ylabel style={font=\color{white!15!black}},
ylabel={Relative state error in $q$},
axis background/.style={fill=white},
xmajorgrids,
ymajorgrids,
legend style={draw=none, legend columns=-1},
legend style={at={(0.0,1.5)}, anchor=south west, legend cell align=left, align=left, draw=white!15!black},
legend style={font=\small}
]
\addplot [color=blue, dotted, line width=2.0pt, mark size=4.0pt, mark=x, mark options={solid, blue}]
  table[row sep=crcr]{%
4	0.0112030400233228\\
8	0.00986478175382966\\
12	0.00986269463646506\\
16	0.00985834428953639\\
20	0.00985187716174048\\
};
\addlegendentry{PSD}

\addplot [color=mycolor1, dashed, line width=2.0pt, mark size=4.0pt, mark=o, mark options={solid, mycolor1}]
  table[row sep=crcr]{%
4	0.0158688811002271\\
8	0.0103110498622526\\
12	0.0103805905294615\\
16	0.0137201066457001\\
20	0.0126414396328184\\
};
\addlegendentry{spDEIM with $r_{\text{spDEIM}}=r$}

\addplot [color=red, dashdotted, line width=2.0pt, mark size=4.0pt, mark=+, mark options={solid, red}]
  table[row sep=crcr]{%
4	0.0111940475922443\\
8	0.0101854870624077\\
12	0.0101919207763581\\
16	0.0100109318489529\\
20	0.00991712364133515\\
};
\addlegendentry{spDEIM with $r_{\text{spDEIM}}=2r$}

\addplot [color=green, line width=2.0pt, mark size=4.0pt, mark=triangle, mark options={solid, green}]
  table[row sep=crcr]{%
4	0.0107892836569169\\
8	0.0096727992375151\\
12	0.00985699578235654\\
16	0.00985664208428246\\
20	0.00985132149135764\\
};
\addlegendentry{structure-preserving lifting}

\end{axis}

\begin{axis}[%
width=1.227\fheight,
height=0.723\fheight,
at={(-0.16\fheight,-0.08\fheight)},
scale only axis,
xmin=0,
xmax=1,
ymin=0,
ymax=1,
axis line style={draw=none},
ticks=none,
axis x line*=bottom,
axis y line*=left
]
\end{axis}
\end{tikzpicture}%
}
\label{fig:exp_state_train}
    \end{subfigure}
     \begin{subfigure}{.23\textwidth}
       \setlength\fheight{3 cm}
        \setlength\fwidth{\textwidth}
%
%
\definecolor{mycolor1}{rgb}{1.00000,0.00000,1.00000}%

\begin{tikzpicture}

\begin{axis}[%
width=0.683\fheight,
height=0.65\fheight,
at={(0\fheight,0\fheight)},
scale only axis,
xmin=4,
xmax=20,
xlabel style={font=\color{white!15!black}},
xlabel={Reduced dimension $2r$},
ymode=log,
ymin=0.001,
ymax=10,
yminorticks=true,
ylabel style={font=\color{white!15!black}},
ylabel={Relative state error in $p$},
axis background/.style={fill=white},
xmajorgrids,
ymajorgrids,
]
\addplot [color=blue, dotted, line width=2.0pt, mark size=4.0pt, mark=x, mark options={solid, blue}, forget plot]
  table[row sep=crcr]{%
4	0.224754670895917\\
8	0.224193481511836\\
12	0.221568867386758\\
16	0.218620085095729\\
20	0.215991924581325\\
};
\addplot [color=mycolor1, dashed, line width=2.0pt, mark size=4.0pt, mark=o, mark options={solid, mycolor1}]
  table[row sep=crcr]{%
4	0.229847404371219\\
8	0.224846469128811\\
12	0.222566976524038\\
16	0.225476342666547\\
20	0.220901431088653\\
};
\addplot [color=red, dashdotted, dashed, line width=2.0pt, mark size=4.0pt, mark=+, mark options={solid, red}]
  table[row sep=crcr]{%
4	0.224756719040161\\
8	0.224931416991001\\
12	0.222514157334165\\
16	0.21910924665814\\
20	0.216181801908159\\
};
\addplot [color=green, line width=2.0pt, mark size=4.0pt, mark=triangle, mark options={solid, green}]
  table[row sep=crcr]{%
4	1.39659438456626\\
8	1.39732989008522\\
12	1.39788944666986\\
16	1.39828572120109\\
20	1.39873306960895\\
};
\end{axis}
\end{tikzpicture}%
\caption{Training data regime $[0,10]$}
\label{fig:exp_state_train}
    \end{subfigure}
    \hspace{0.25cm}
    \begin{subfigure}{.23\textwidth}
           \setlength\fheight{3 cm}
           \setlength\fwidth{\textwidth}
%
%
\definecolor{mycolor1}{rgb}{1.00000,0.00000,1.00000}%
\begin{tikzpicture}

\begin{axis}[%
width=0.683\fheight,
height=0.65\fheight,
at={(0\fheight,0\fheight)},
scale only axis,
xmin=4,
xmax=20,
xlabel style={font=\color{white!15!black}},
xlabel={Reduced dimension $2r$},
ymode=log,
ymin=0.001,
ymax=0.1,
yminorticks=true,
ylabel style={font=\color{white!15!black}},
ylabel={Relative state error in $q$},
axis background/.style={fill=white},
xmajorgrids,
ymajorgrids,
legend style={legend cell align=left, align=left, draw=white!15!black}
]
\addplot [color=blue, dotted, line width=2.0pt, mark size=4.0pt, mark=x, mark options={solid, blue}]
  table[row sep=crcr]{%
4	0.0264871172694559\\
8	0.0168955997680731\\
12	0.0168960916279813\\
16	0.0168899761529459\\
20	0.016872623930039\\
};

\addplot [color=mycolor1, dashed, line width=2.0pt, mark size=4.0pt, mark=o, mark options={solid, mycolor1}]
  table[row sep=crcr]{%
4	0.0523061292764936\\
8	0.0259244207792918\\
12	0.0254789749901673\\
16	0.0334265285267235\\
20	0.033783250551605\\
};

\addplot [color=red, dashed, line width=2.0pt, mark size=4.0pt, mark=+, mark options={solid, red}]
  table[row sep=crcr]{%
4	0.0259230575597494\\
8	0.022754887563149\\
12	0.0202861870643176\\
16	0.0176802158638452\\
20	0.0176955980236424\\
};
\addplot [color=green, line width=2.0pt, mark size=4.0pt, mark=triangle, mark options={solid, green}]
  table[row sep=crcr]{%
4	0.0182335027992255\\
8	0.0167515636859981\\
12	0.0169141259261324\\
16	0.0169089558737439\\
20	0.0168991242122208\\
};

\end{axis}
\end{tikzpicture}%
\label{fig:exp_state_test}
    \end{subfigure}
       \hspace{0.25cm}
    \begin{subfigure}{.23\textwidth}
           \setlength\fheight{3 cm}
           \setlength\fwidth{\textwidth}
%
%
\definecolor{mycolor1}{rgb}{1.00000,0.00000,1.00000}%
\begin{tikzpicture}

\begin{axis}[%
width=0.683\fheight,
height=0.65\fheight,
at={(0\fheight,0\fheight)},
scale only axis,
xmin=4,
xmax=20,
xlabel style={font=\color{white!15!black}},
xlabel={Reduced dimension $2r$},
ymode=log,
ymin=0.001,
ymax=10,
yminorticks=true,
ylabel style={font=\color{white!15!black}},
ylabel={Relative state error in $p$},
axis background/.style={fill=white},
xmajorgrids,
ymajorgrids,
]
\addplot [color=blue, dotted, line width=2.0pt, mark size=4.0pt, mark=x, mark options={solid, blue}, forget plot]
  table[row sep=crcr]{%
4	0.257328500869418\\
8	0.237330503877936\\
12	0.239933272651322\\
16	0.242701837462397\\
20	0.246335550453495\\
};
\addplot [color=mycolor1, dashed, line width=2.0pt, mark size=4.0pt, mark=o, mark options={solid, mycolor1}]
  table[row sep=crcr]{%
4	0.26731535382371\\
8	0.245264591898598\\
12	0.251958791501594\\
16	0.26991405729821\\
20	0.273406817025201\\
};
\addplot [color=red, dashdotted, line width=2.0pt, mark size=4.0pt, mark=+, mark options={solid, red}]
  table[row sep=crcr]{%
4	0.257080810797993\\
8	0.243817634102084\\
12	0.245290175335424\\
16	0.243364024999107\\
20	0.24585318625416\\
};
\addplot [color=green, line width=2.0pt, mark size=4.0pt, mark=triangle, mark options={solid, green}]
  table[row sep=crcr]{%
4	1.39691431028287\\
8	1.39735963462445\\
12	1.39794050511616\\
16	1.39845388502968\\
20	1.39909917356058\\
};
\end{axis}
\end{tikzpicture}%
\caption{Test data regime $[10,100]$}
\label{fig:exp_state_test}
    \end{subfigure}
\caption{Nonlinear wave equation with exponential nonlinearity. {\Ra{The relative state error comparison for $q$ in plots (a) and (b) show that}} quadratic ROMs obtained via the proposed structure-preserving lifting approach achieve similar state error to nonlinear ROMs obtained via PSD and spDEIM{\Rd{ with $r_{\text{spDEIM}}=2r$}} in both training and test data regimes. {\Rd{The relative state error comparison for $p$ in plots (a) and (b) indicates that, while all ROMs perform poorly with relative error in $p$ above $10^{-1}$, the nonlinear ROMs obtained via PSD and spDEIM achieve higher accuracy than the quadratic ROMs in both training and test regimes.}} } 
 \label{fig:exp_1d_state}
\end{figure}
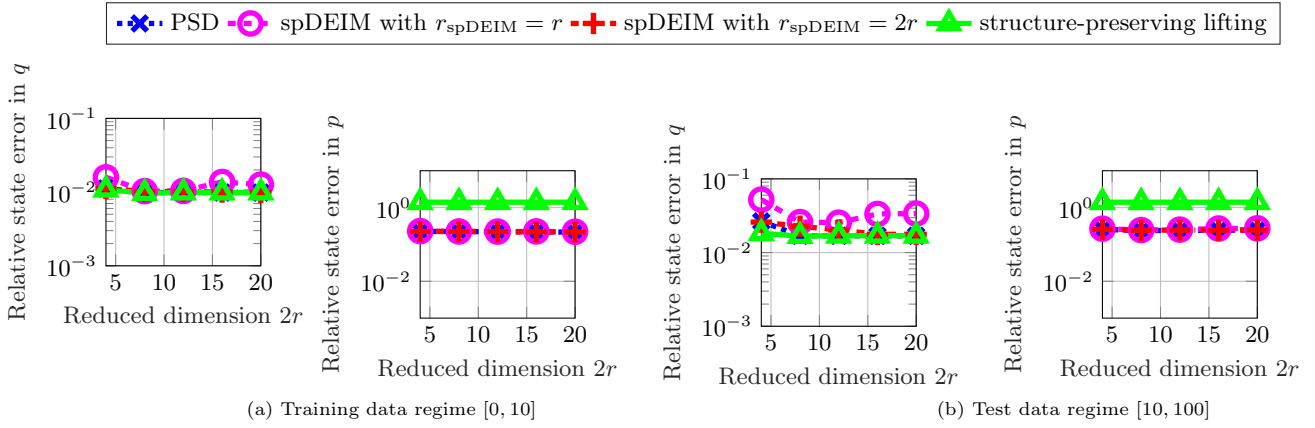
\begin{figure}[tbp]
\small
\captionsetup[subfigure]{oneside,margin={1.8cm,0 cm}}
\begin{subfigure}{.45\textwidth}
       \setlength\fheight{6 cm}
        \setlength\fwidth{\textwidth}
%
%
\definecolor{mycolor1}{rgb}{1.00000,0.00000,1.00000}%
\begin{tikzpicture}

\begin{axis}[%
width=0.951\fheight,
height=0.59\fheight,
at={(0\fheight,0\fheight)},
scale only axis,
xmin=4,
xmax=20,
xlabel style={font=\color{white!15!black}},
xlabel={Reduced dimension $2r$},
ymode=log,
ymin=100,
ymax=10000,
yminorticks=true,
ylabel style={font=\color{white!15!black}},
ylabel={Efficacy},
axis background/.style={fill=white},
xmajorgrids,
ymajorgrids,
legend style={at={(0.1,1.2)}, anchor=south west, legend cell align=left, align=left, draw=white!15!black},
legend style={font=\small}
]
\addplot [color=mycolor1, dashed, line width=2.0pt, mark size=4.0pt, mark=o, mark options={solid, mycolor1}]
  table[row sep=crcr]{%
4	573.365426056062\\
8	735.536527307993\\
12	631.814711423858\\
16	444.563938047148\\
20	398.158572098041\\
};
\addlegendentry{spDEIM with $r_{\text{spDEIM}}=r$}

\addplot [color=red, dashed, line width=2.0pt, mark size=4.0pt, mark=+, mark options={solid, red}]
  table[row sep=crcr]{%
4	810.987241224549\\
8	797.282973242682\\
12	682.079040663352\\
16	596.866735194317\\
20	473.469708690564\\
};
\addlegendentry{spDEIM with $r_{\text{spDEIM}}=2r$}

\addplot [color=green, line width=2.0pt, mark size=4.0pt, mark=triangle, mark options={solid, green}]
  table[row sep=crcr]{%
4	3485.23226968573\\
8	3320.31682291028\\
12	1835.18308349633\\
16	1262.32973935501\\
20	770.817897681221\\
};
\addlegendentry{structure-preserving lifting}

\end{axis}

\begin{axis}[%
width=1.227\fheight,
height=0.723\fheight,
at={(-0.16\fheight,-0.08\fheight)},
scale only axis,
xmin=0,
xmax=1,
ymin=0,
ymax=1,
axis line style={draw=none},
ticks=none,
axis x line*=bottom,
axis y line*=left
]
\end{axis}
\end{tikzpicture}%
\caption{Efficacy}
\label{fig:exp_1d_eff}
    \end{subfigure}
    \hspace{0.4cm}
    \begin{subfigure}{.45\textwidth}
           \setlength\fheight{6 cm}
           \setlength\fwidth{\textwidth}
\input{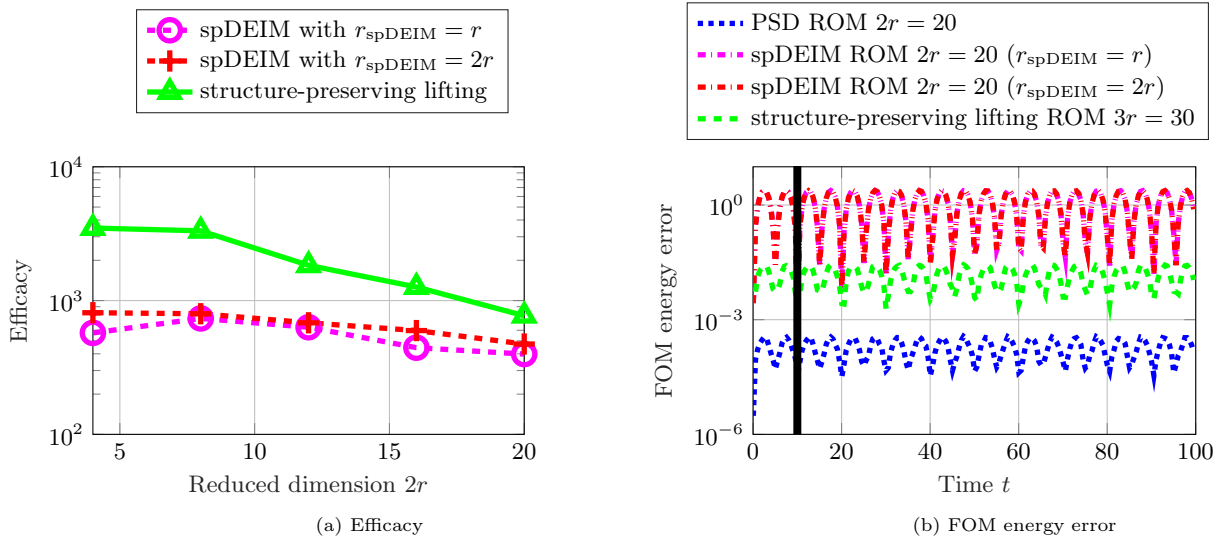}
\caption{FOM energy error}
\label{fig:exp_1d_ener}
    \end{subfigure}
\caption{Nonlinear wave equation with exponential nonlinearity. The efficacy comparison in plot (a) shows that the proposed structure-preserving lifting approach achieves similar accuracy at a substantially lower computational cost in the online stage than the spDEIM approach. Plot (b) shows that all ROMs demonstrate bounded energy error,{\Rd{ with the PSD ROM achieving the lowest energy error.}} The solid black line in plot (b) indicates the end of the training data regime.}
 \label{fig:exp_1d}
\end{figure}
\subsection{Two-dimensional sine-Gordon wave equation}
\label{sec:sg_2d}
We now consider the two-dimensional analogue of the sine-Gordon equation from Section~\ref{sec:fom}. The sine-Gordon equation~\cite{rubinstein1970sine,cuevas2014sine} is a universal nonlinear wave model combining the wave dispersion and the periodic nonlinearity. This equation is used for modeling nonlinear phenomena in a variety of applications, including solid state physics~\cite{josephson1962possible}, relativistic field theory~\cite{samuel1978grand}, nonlinear optics~\cite{mccall1969self}, and hydrodynamics~\cite{coullet1986resonance}. The field variable in the sine-Gordon equation can be interpreted as modeling the phase in the respective physical setting.
\subsubsection{PDE formulation and the corresponding nonlinear conservative FOM}
We consider the FOM setup from~\cite{bo2022arbitrary}. Let $\Omega=(-7,7) \times (-7,7) \subset \real^2 $ be the spatial domain and consider the two-dimensional sine-Gordon equation
\begin{equation}\label{eq:sg_2d_pde}
\frac{\partial^2 \phi}{\partial t^2}(x,y,t)=\frac{\partial^2 \phi}{\partial x^2} (x,y,t)+ \frac{\partial^2 \phi}{\partial y^2}(x,y,t) - \sin(\phi(x,y,t)),
\end{equation}
for $t \in (0,T]$. The boundary conditions are periodic and the corresponding initial conditions are
\begin{equation*}
\phi(x,y,0)=4\tan^{-1}\left(\exp\left((3-\sqrt{x^2+y^2}\right)\right), \qquad  \frac{\partial \phi}{\partial t}(x,y,0)=0.
\end{equation*}
We define $q(x,y,t):=\phi(x,y,t)$ and $p(x,y,t):=\partial \phi(x,y,t)/\partial t$ and rewrite the second-order conservative PDE~\eqref{eq:sg_2d_pde} in the first-order form
\begin{equation}
\frac{\partial q(x,t)}{\partial t} =p(x,t), \qquad \qquad
\frac{\partial p(x,t)}{\partial t} =\frac{\partial^2 q(x,y,t)}{\partial x^2}+ \frac{\partial^2 q(x,y,t)}{\partial y^2} - \sin(q(x,y,t)).
\end{equation}

We discretize the two-dimensional spatial domain $\Omega$ with $n_x=n_y=100$ equally spaced grid points in both spatial directions leading to a nonlinear FOM of dimension $2n$ with $n=n_xn_y=10,000$. The corresponding nonlinear conservative FOM is
\begin{equation*} 
    \dot{\q}=\p, \qquad
    \dot{\p} =\D\q -\sin( \q),
    \label{eq:kg_fom}
\end{equation*}
where $\D=\D^\top$ denotes the {\Ra{symmetric discrete Jacobian}} in the two-dimensional setting and the vector $ \sin( \q)\in \real^{n}$ contains as components the entry-wise sine function of the FOM state vector $\bq$. The nonlinear FOM conserves the space-discretized energy
\begin{equation*}
E(\q,\p)=\frac{1}{2} \p^\top \p -\frac{1}{2}\q^\top\D \q + \sum_{i=1}^n (1-\cos(q_i)).
\end{equation*}
\subsubsection{Structure-preserving lifting based on energy quadratization}
Since the two-dimensional sine-Gordon equation~\eqref{eq:sg_2d_pde} has the same form of nonlinearity as its one-dimensional counterpart in equation~\eqref{eq:sg_pde}, the energy-quadratization strategy combined with POD model reduction yields an energy-preserving quadratic ROM that has the same structure as the one-dimensional sine-Gordon example discussed in Section~\ref{sec:comp}.  {\Ra{This example also satisfies Assumptions 1 and 2, and hence Theorem~\ref{theorem} applies.}}
\subsubsection{Numerical results}
We build the training dataset by integrating the nonlinear FOM for the two-dimensional sine-Gordon equation from $t=0$ to $t=10$. For test data, we consider FOM snapshots from $t=10$ to $t=12.5$, {\Rb{which is 25\% outside the training time interval to demonstrate the proposed method's time extrapolation capability for transport-dominated problems}}. Due to the challenging nature of this two-dimensional example, spDEIM ROMs with $r_{\text{spDEIM}}=r$ are unable to provide accurate predictions. Therefore, we also include plots for spDEIM ROMs with $r_{\text{spDEIM}}=2r$. 

We compare the relative state error of the structure-preserving lifting approach against PSD and spDEIM approaches for the training data regime and the test data regime in Figure~\ref{fig:sg_state_train} and Figure~\ref{fig:sg_state_test}, respectively. We observe that the quadratic ROMs obtained via the structure-preserving lifting approach achieve accuracy similar to the PSD ROMs in both training and test data regimes. While the spDEIM ROMs with $r_{\text{spDEIM}}=2r$ achieve higher accuracy than the spDEIM ROMs with $r_{\text{spDEIM}}=r$, the proposed structure-preserving lifting approach achieves higher accuracy than the spDEIM ROMs with $r_{\text{spDEIM}}=2r$ for $2r\geq 20$. 

In Figure~\ref{fig:sg_2d_eff}, we compare the efficacy of quadratic ROMs obtained via structure-preserving lifting and spDEIM ROM with $r_{\text{spDEIM}}=r$ and $r_{\text{spDEIM}}=2r$. The $4r-$dimensional quadratic ROMs obtained via structure-preserving lifting and the $2r-$dimensional nonlinear spDEIM ROMs with $r_{\text{spDEIM}}=2r$ achieve similar efficacy whereas the $2r-$dimensional nonlinear ROMs obtained with $r_{\text{spDEIM}}=r$ yield substantially lower efficacy primarily due to their poor relative state error performance. Thus, the proposed approach yields quadratic ROMs that achieve accuracy similar to nonlinear PSD ROMs while achieving computational efficiency comparable to nonlinear spDEIM ROMs.

The FOM energy error plots in Figure~\ref{fig:sg_2d_ener} compare energy error performance of the structure-preserving lifting approach against PSD and spDEIM. {\Ra{The quadratic ROM obtained via structure-preserving lifting achieves substantially lower energy error than the spDEIM ROMs in both training and test data regimes. The PSD ROM, on the other hand, yields the lowest energy error among all three approaches.}}
\begin{figure}[tbp]
\small
\captionsetup[subfigure]{oneside,margin={1.8cm,0 cm}}
\begin{subfigure}{.45\textwidth}
       \setlength\fheight{6 cm}
        \setlength\fwidth{\textwidth}
%
%
\definecolor{mycolor1}{rgb}{1.00000,0.00000,1.00000}%
\begin{tikzpicture}

\begin{axis}[%
width=0.951\fheight,
height=0.59\fheight,
at={(0\fheight,0\fheight)},
scale only axis,
xmin=0,
xmax=40,
xlabel style={font=\color{white!15!black}},
xlabel={Reduced dimension $2r$},
ymode=log,
ymin=0.001,
ymax=1,
yminorticks=true,
ylabel style={font=\color{white!15!black}},
ylabel={Relative state error in $q$},
axis background/.style={fill=white},
xmajorgrids,
ymajorgrids,
legend style={draw=none, legend columns=-1},
legend style={at={(-0.1,1.1)}, anchor=south west, legend cell align=left, align=left, draw=white!15!black},
legend style={font=\small}
]
\addplot [color=blue, dotted, line width=2.0pt, mark size=4.0pt, mark=x, mark options={solid, blue}]
  table[row sep=crcr]{%
4	0.158554136590828\\
8	0.0745593043337017\\
12	0.0255604183072662\\
16	0.0105140103648031\\
20	0.00350022677982608\\
24	0.00368553805692792\\
28	0.00149616240165273\\
32	0.00133812841155964\\
36	0.0011969274878512\\
40	0.00113979563762971\\
};
\addlegendentry{PSD}

\addplot [color=mycolor1, dashed, line width=2.0pt, mark size=4.0pt, mark=o, mark options={solid, mycolor1}]
  table[row sep=crcr]{%
4	0.35034064414865\\
8	0.393670285881099\\
12	0.319929632768091\\
16	0.327174539772891\\
20	0.292456642360924\\
24	0.265333215645605\\
28	0.254594347897181\\
32	0.256654134804446\\
36	0.340432814253274\\
40	0.294361182169194\\
};
\addlegendentry{spDEIM with $r_{\text{spDEIM}}=r$}

\addplot [color=red, dashdotted, line width=2.0pt, mark size=4.0pt, mark=+, mark options={solid, red}]
  table[row sep=crcr]{%
4	0.149974003188555\\
8	0.0732131477871174\\
12	0.0308966438346059\\
16	0.0144796931850272\\
20	0.0421429465949991\\
24	0.0091278175462787\\
28	0.0145548529154858\\
32	0.032782131467471\\
36	0.0403346100555187\\
40	0.0150797521897998\\
};
\addlegendentry{spDEIM with $r_{\text{spDEIM}}=2r$}

\addplot [color=green, line width=2.0pt, mark size=4.0pt, mark=triangle, mark options={solid, green}]
  table[row sep=crcr]{%
4	0.24266766129569\\
8	0.0691722107715322\\
12	0.0350473268638889\\
16	0.0106455848990292\\
20	0.00352828840194326\\
24	0.00390853353636164\\
28	0.00152206863655869\\
32	0.00136041476180307\\
36	0.0012034997637606\\
40	0.0011620193887574\\
};
\addlegendentry{structure-preserving lifting}

\end{axis}
\end{tikzpicture}%
\caption{Training data regime $[0,10]$}
\label{fig:sg_state_train}
    \end{subfigure}
    \hspace{0.4cm}
    \begin{subfigure}{.45\textwidth}
           \setlength\fheight{6 cm}
           \setlength\fwidth{\textwidth}
\raisebox{-53mm}{
%
%
\definecolor{mycolor1}{rgb}{1.00000,0.00000,1.00000}%
\begin{tikzpicture}

\begin{axis}[%
width=0.951\fheight,
height=0.59\fheight,
at={(0\fheight,0\fheight)},
scale only axis,
xmin=0,
xmax=40,
xlabel style={font=\color{white!15!black}},
xlabel={Reduced dimension $2r$},
ymode=log,
ymin=0.01,
ymax=1,
yminorticks=true,
ylabel style={font=\color{white!15!black}},
ylabel={Relative state error in $q$},
axis background/.style={fill=white},
xmajorgrids,
ymajorgrids,
legend style={at={(0.3,0.268)}, anchor=south west, legend cell align=left, align=left, draw=white!15!black}
]
\addplot [color=blue, dotted, line width=2.0pt, mark size=4.0pt, mark=x, mark options={solid, blue}]
  table[row sep=crcr]{%
4	0.427475012086782\\
8	0.215380006122958\\
12	0.173182225294446\\
16	0.106715689835152\\
20	0.0717537674474621\\
24	0.0692958511936203\\
28	0.0586768412538284\\
32	0.0589658939274266\\
36	0.0563376560743208\\
40	0.0556042747913625\\
};

\addplot [color=mycolor1, dashed, line width=2.0pt, mark size=4.0pt, mark=o, mark options={solid, mycolor1}]
  table[row sep=crcr]{%
4	0.435407468436249\\
8	0.333725141169838\\
12	0.54415388459157\\
16	0.482357140422123\\
20	0.30064352118013\\
24	0.280100617369238\\
28	0.237171818116093\\
32	0.255409264977332\\
36	0.539059113394953\\
40	0.290371012727452\\
};

\addplot [color=red, dashdotted, line width=2.0pt, mark size=4.0pt, mark=+, mark options={solid, red}]
  table[row sep=crcr]{%
4	0.455295384320912\\
8	0.218197990851136\\
12	0.176027639761994\\
16	0.105480286180431\\
20	0.0891433135483439\\
24	0.0795784754117206\\
28	0.0755683784458242\\
32	0.0878236216463809\\
36	0.086679062018885\\
40	0.0721181172396473\\
};

\addplot [color=green, line width=2.0pt, mark size=4.0pt, mark=triangle, mark options={solid, green}]
  table[row sep=crcr]{%
4	0.378018218184115\\
8	0.234285714590927\\
12	0.171080932168652\\
16	0.105929805607251\\
20	0.0712833541784429\\
24	0.068711608526565\\
28	0.0582902504637456\\
32	0.0589014487242964\\
36	0.0560710647817749\\
40	0.0554768114995593\\
};

\end{axis}

\begin{axis}[%
width=1.227\fheight,
height=0.723\fheight,
at={(-0.16\fheight,-0.08\fheight)},
scale only axis,
xmin=0,
xmax=1,
ymin=0,
ymax=1,
axis line style={draw=none},
ticks=none,
axis x line*=bottom,
axis y line*=left
]
\end{axis}
\end{tikzpicture}%
}
\caption{Test data regime $[10,12.5]$}
\label{fig:sg_state_test}
    \end{subfigure}
\caption{Two-dimensional sine-Gordon equation. The relative state error comparison in plot (a) shows that proposed structure-preserving lifting approach achieves higher accuracy than the spDEIM approach in the training regime. Plot (b) shows that both approaches yield similar accuracy in the test data regime with the structure-preserving lifting approach performing marginally better. }
 \label{fig:sg_2d_state}
\end{figure}
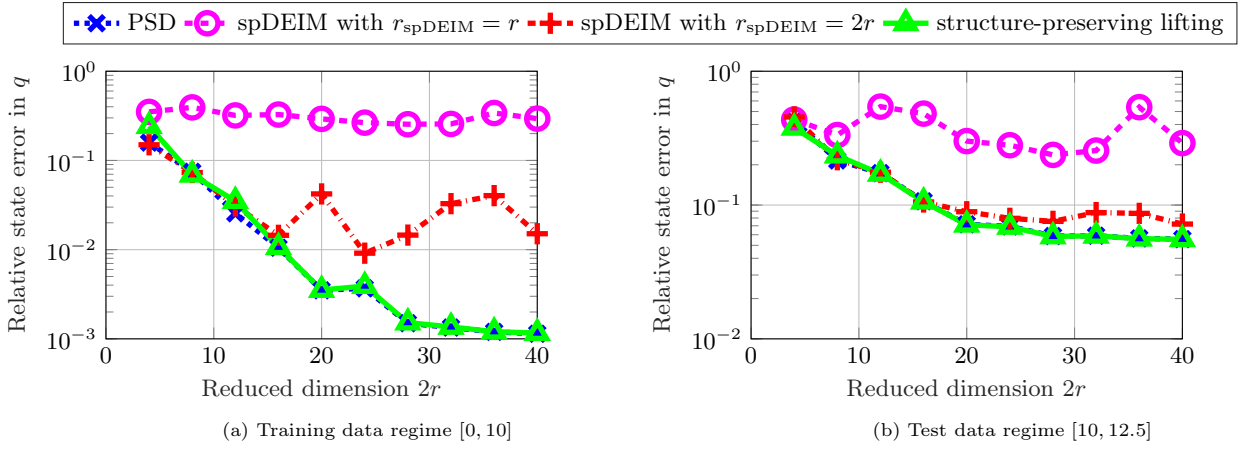
\begin{figure}[tbp]
\small
\captionsetup[subfigure]{oneside,margin={1.8cm,0 cm}}
\begin{subfigure}{.45\textwidth}
       \setlength\fheight{6 cm}
        \setlength\fwidth{\textwidth}
%
%
\definecolor{mycolor1}{rgb}{1.00000,0.00000,1.00000}%
\begin{tikzpicture}

\begin{axis}[%
width=0.951\fheight,
height=0.59\fheight,
at={(0\fheight,0\fheight)},
scale only axis,
xmin=0,
xmax=40,
xlabel style={font=\color{white!15!black}},
xlabel={Reduced dimension $2r$},
ymode=log,
ymin=10,
ymax=1839.38392423517,
yminorticks=true,
ylabel style={font=\color{white!15!black}},
ylabel={Efficacy},
axis background/.style={fill=white},
xmajorgrids,
ymajorgrids,
legend style={at={(0.1,1.25)}, anchor=south west, legend cell align=left, align=left, draw=white!15!black},
legend style={font=\small}
]
\addplot [color=mycolor1, dashed, line width=2.0pt, mark size=4.0pt, mark=o, mark options={solid, mycolor1}]
  table[row sep=crcr]{%
4	57.2847678996387\\
8	50.7598043879949\\
12	56.6680856463636\\
16	38.669227546655\\
20	45.7493415036782\\
24	34.30784892147\\
28	36.3016688128291\\
32	28.9560012543227\\
36	20.3832844660372\\
40	24.5077703807247\\
};
\addlegendentry{spDEIM with $r_{\text{spDEIM}}=r$}

\addplot [color=red, dashdotted, line width=2.0pt, mark size=4.0pt, mark=+, mark options={solid, red}]
  table[row sep=crcr]{%
4	129.540833324326\\
8	284.640739133744\\
12	568.695273072775\\
16	1017.43190431502\\
20	282.127722878766\\
24	1086.23676395647\\
28	629.967504589345\\
32	256.168833912976\\
36	187.447419776401\\
40	475.697171648081\\
};
\addlegendentry{spDEIM with $r_{\text{spDEIM}}=2r$}

\addplot [color=green, line width=2.0pt, mark size=4.0pt, mark=triangle, mark options={solid, green}]
  table[row sep=crcr]{%
4	119.811484979457\\
8	398.479468570567\\
12	520.256196644394\\
16	1042.3009671115\\
20	1839.38392423517\\
24	425.551727037633\\
28	718.083679580088\\
32	207.794146487056\\
36	379.401466891611\\
40	299.690778573534\\
};
\addlegendentry{structure-preserving lifting}

\end{axis}
\end{tikzpicture}%
\caption{Efficacy}
\label{fig:sg_2d_eff}
    \end{subfigure}
    \hspace{0.4cm}
    \begin{subfigure}{.45\textwidth}
           \setlength\fheight{6 cm}
           \setlength\fwidth{\textwidth}
\input{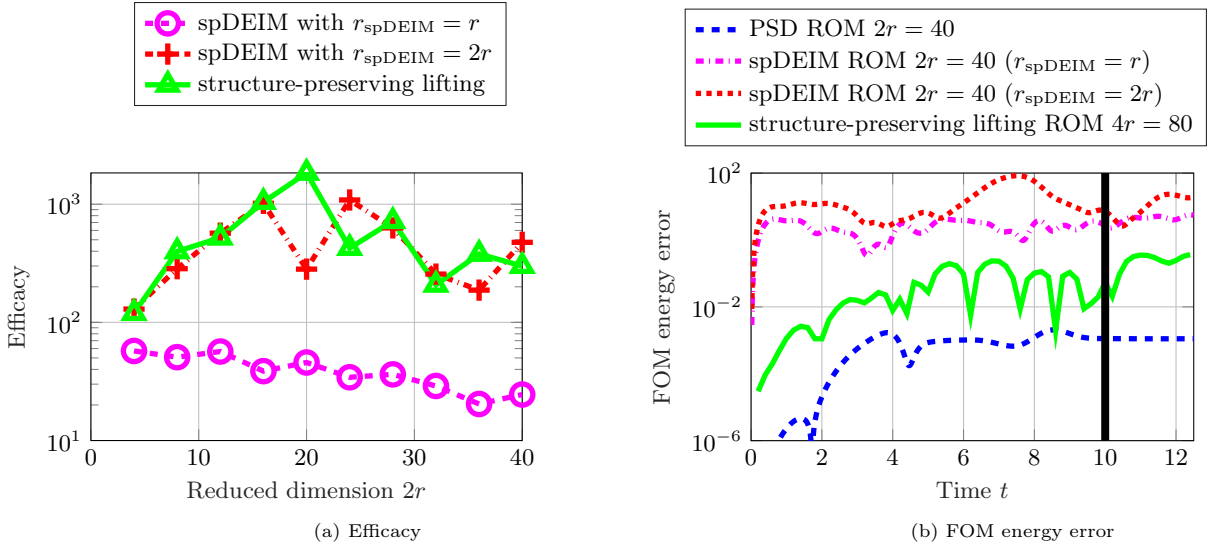}
\caption{FOM energy error}
\label{fig:sg_2d_ener}
    \end{subfigure}
\caption{Two-dimensional sine-Gordon wave equation. Plot (a) shows that both {\Rc{spDEIM with $r_{\text{spDEIM}}=2r$}} and structure-preserving lifting approaches provide similar efficacy {\Rc{whereas spDEIM ROMs with $r_{\text{spDEIM}}=r$ yield substantially lower efficacy}}. The energy error comparison in plot (b) shows that the {\Ra{structure-preserving lifting approach achieves lower energy error than both spDEIM ROMs whereas the PSD ROM achieves the lowest energy error.}} The solid black line in plot (b) indicates the end of the training time interval.}
 \label{fig:sg_2d}
\end{figure}
\subsection{Two-dimensional parametrized Klein-Gordon equation}
\label{sec:kg_2d}
The nonlinear Klein-Gordon equation with parametric nonlinearity~\cite{scott1969nonlinear} is one of the simplest nonlinear relativistic equations in mathematical physics. This nonlinear wave equation was originally studied in the context of the general theory of relativity where it was considered a candidate for relativistic generalization of the Schrödinger equation, see~\cite{kragh1984equation} for more details about its origin. Since its early development almost a century ago, the nonlinear Klein-Gordon equation has become a prototype for modeling nonlinear phenomena in various fields, including metamaterials~\cite{giri2011klein}, electromagnetism~\cite{carvalho2016klein}, and fluid dynamics~\cite{mauser2020rotating}. 
\subsubsection{Parametric PDE formulation and the corresponding nonlinear conservative FOM with parametric dependence}
The nonlinear parametric FOM considered in this section is similar to the two-dimensional nonlinear wave equation example in~\cite{jiang2020linearly}. Let $\Omega=(-10,10) \times (-10,10) \subset \real^2$ be the spatial domain and consider the parametrized two-dimensional nonlinear wave equation 
\begin{equation}\label{eq:non_pde}
\frac{\partial^2 \phi}{\partial t^2}(x,y,t;\mu)=\frac{\partial^2 \phi}{\partial x^2} (x,y,t;\mu)+ \frac{\partial^2 \phi}{\partial y^2}(x,y,t;\mu) - \mu\phi(x,y,t;\mu)^3, 
\end{equation}
with time $t \in (0,T]$ and the scalar parameter $\mu \in \mathcal P=[0.1,1.4]$. We consider periodic boundary conditions
\begin{equation}
\phi(-10,y,t;\mu)=\phi(10,y,t;\mu), \qquad \phi(x,-10,t;\mu)=\phi(x,10,t;\mu).
\end{equation}
The initial conditions are
\begin{equation*}
\phi(x,y,0)=2\sech(\cosh(x^2+y^2)), \qquad  \frac{\partial \phi}{\partial t}(x,y,0)=0.
\end{equation*}
We define $q(x,y,t;\mu):=\phi(x,y,t;\mu)$ and $p(x,y,t;\mu):=\partial \phi(x,y,t;\mu)/\partial t$ to rewrite~\eqref{eq:non_pde} as a system of first-order nonlinear PDEs
\begin{align*}
\frac{\partial }{\partial t}q(x,y,t;\mu)&=p(x,y,t;\mu),\\
\frac{\partial }{\partial t}p(x,y,t;\mu)&=\frac{\partial ^2}{\partial x^2}q(x,y,t;\mu)+\frac{\partial ^2}{\partial y^2}q(x,y,t;\mu)-\mu q(x,y,t;\mu)^3.
\end{align*}

We discretize the two-dimensional spatial domain $\Omega$ with $n_x=n_y=100$ equally spaced grid points in both spatial directions leading to a nonlinear FOM of dimension $2n$ with $n=n_xn_y=10,000$. The corresponding parametrized nonlinear Hamiltonian FOM is
\begin{equation*} 
    \dot{\q}=\p(\mu), \qquad
    \dot{\p} =\D\q -\mu \q^3,
    \label{eq:kg_fom}
\end{equation*}
where $\D=\D^\top$ denotes the {\Ra{symmetric discrete Jacobian}} in the two-dimensional setting and the vector $ \bq^3\in \real^{n}$ contains as components the entry-wise cubic exponential of the FOM state vector $\bq$. The parametric nonlinear FOM conserves the space-discretized energy
\begin{equation}
E(\q,\p;\mu)=\frac{1}{2} \p^\top \p -\frac{1}{2}\q^\top\D \q + \frac{\mu}{4}\sum_{i=1}^n (q_i)^4.
\label{eq:kg_non_fom}
\end{equation}
\subsubsection{Structure-preserving lifting based on energy quadratization}
Since the nonlinear term {\Rd{$g(q)=q^4/4$}} in the nonlinear FOM energy expression~\eqref{eq:kg_non_fom} is nonnegative, we follow the energy-quadratization strategy and {\Rd{solve $ \w_1^\top\w_1=\kappa^2 g(\q)$ with $g(q)=q^4/4$ and choose $\kappa=2$ to obtain the first auxiliary variable $\w_1=\q^2$}}. {\Rd{Since the time evolution equations for $\{\q,\p,\w_1\}$ are quadratic in terms of $\{\q,\p,\w_1 \}$, we do not require additional auxiliary variables in this parametric example.}} The proposed lifting transformation
transforms the nonlinear conservative FOM into a quadratic lifted FOM
\begin{align*} 
    \dot{\q}&=\p, \\
    \dot{\p}&=\D\q - \mu \left({\Rd{\w_1}} \odot \q\right), \\
    {\Rd{\dot{\w}_1}}&=2\q \odot \p,    \end{align*}
with quadratic FOM energy in the lifted variables, i.e., $E_{\text{lift}}(\q,\p,{\Rd{\w_1}})=\frac{1}{2}\p^\top\p - \frac{1}{2}\q^\top \D\q + \frac{\mu}{4} {\Rd{\w_1}}^\top{\Rd{\w_1}} $. 
Similarly to the exponential nonlinear example from Section~\ref{sec:exp}, we consider a block-diagonal POD basis matrix of the form
\begin{equation*}
\bar{\V} =\text{blkdiag}(\bPhi,\bPhi,{\Rd{\V_1}})\in \real^{3n\times 3r},
\end{equation*}
where $\bPhi \in \real^{n\times r}$ is the PSD basis matrix for $\q$ and $\p$, and ${\Rd{\V_1}}\in \real^{n\times r}$ is the POD basis computed via SVD of the lifted snapshot data ${\Rd{\mathbf W_1}}$. Projecting the governing equations for the lifted parametric FOM onto the basis matrix $\bar{\V} $ yields a quadratic ROM with parametric dependence
\begin{align*} 
    \dot{\qhat}&=\phat, \\
    \dot{\phat}&=\Dhat\qhat -\mu\bPhi^\top\left( {\Rd{\V_1}}{\Rd{\what_1}} \odot  \bPhi\qhat\right) , \\
    {\Rd{\dot{\what}_1}}&=2{\Rd{\V_1}}^\top \left( \bPhi\qhat \odot \bPhi\phat\right),
    \end{align*}
    where $\Dhat=\bPhi^\top\D\bPhi$. Substituting the lifted FOM approximation based on the basis matrix $\bar{\V} $ into the lifted FOM energy expression and computing its time derivative yields
\begin{align*} 
   \frac{\dd}{{\dd}t}E_{\text{lift}}(\bPhi\qhat,\bPhi\phat,{\Rd{\V_1}}{\Rd{\what_1}};\mu)&=\phat^\top(\Dhat\qhat-\mu\bPhi^\top\left( {\Rd{\V_1}}{\Rd{\what_1}} \odot  \bPhi\qhat\right) ) - \qhat^\top \Dhat\phat + \frac{\mu}{2}  {\Rd{\what_1}}^\top\underbrace{\left( {\Rd{\V_1}}^\top {\Rd{\V_1}}\right)}_{\textbf{I}_r}{\Rd{\dot{\what}_1}} =0
        \end{align*}
 Thus, the structure-preserving lifting approach yields a structure-preserving quadratic ROM with parametric dependence that conserves the lifted FOM energy. {\Ra{This conservative nonlinear FOM with cubic nonlinearity satisfies Assumption 1, and hence Theorem~\ref{theorem} applies.}}
\subsubsection{Numerical results}
Let $\mu_1, \cdots, \mu_{10}$ be $M_{\text{train}}=10$ training parameters equidistantly distributed (including endpoints) in $[0.1,1]\subset \mathcal P$. For this parametric nonlinear PDE example, we build a training dataset by integrating the nonlinear FOM for each training parameter using the implicit midpoint method from time $t=0$ to $t=8$ with a fixed time step of $\Delta t=0.1$. For test data, we consider $M_{\text{test}}=4$ test parameters $\mu_{\text{test},1}=1.1$, $\mu_{\text{test},2}=1.2$, $\mu_{\text{test},3}=1.3$, and $\mu_{\text{test},4}=1.4$ to evaluate how the structure-preserving lifting approach performs in a parameter extrapolation study. {\Rb{The test parameters in this study are chosen such that the test dataset features nonlinear phenomena that are qualitatively different from the training dataset without making the projection error in the test dataset substantially higher.}} Similarly to the two-dimensional sine-Gordon equation example in Section~\ref{sec:sg_2d}, the spDEIM approach fails to yield accurate ROMs with $r_{\text{spDEIM}}=r$ so we show spDEIM ROMs with $r_{\text{spDEIM}}=2r$ and $r_{\text{spDEIM}}=4r$.

In Figure~\ref{fig:kg_2d_state}, we compare the accuracy of the proposed approach against PSD and spDEIM for both training and test parameters. We evaluate the accuracy of ROMs for this parametric problem using the \textit{average relative state error} metric which is computed as 
\begin{equation*}
\text{Average relative state error in } q=\frac{1}{M_{\text{train/test}}}\sum_{j=1}^{M_{\text{train/test}}}   \frac{\lVert \Q_j- \bPhi\widehat{\Q}_j \rVert^2_F}{\lVert \Q_j \rVert^2_F},
\end{equation*}
where $\Q_j \in \real^{n \times K}$ is the FOM snapshot data for the parameter value $\mu_j$ and $\bPhi\widehat{\Q}_j\in \real^{n \times K}$ is its approximation.
Due to the parametric nature of the training dataset, all four approaches achieve relative state error below $10^{-1}$  {\Ra{only for ROMs of sizes greater than $2r=40$}}, and therefore, we only focus on ROMs with $2r \geq 40$. For training parameters, we observe that all three approaches yield similar accuracy in Figure~\ref{fig:kg_train} with the spDEIM approach with $r_{\text{spDEIM}}=4r$ performing marginally better than quadratic ROMs obtained via the structure-preserving lifting approach. The relative state error comparison for test parameters in Figure~\ref{fig:kg_test} shows a similar trend with the quadratic ROMs achieving marginally lower accuracy than the PSD ROMs for $2r>50$. 

We compare the efficacy of the proposed approach against {\Ra{PSD}} and spDEIM in Figure~\ref{fig:kg_eff}. {\Ra{The plots in Figure~\ref{fig:kg_eff} show that the PSD ROMs yield substantially lower efficacy than the structure-preserving lifting ROMs.}} Unlike the efficacy comparisons for the other non-parametric examples, we observe that the spDEIM ROMs yield higher efficacy than the structure-preserving lifting approach. Due to the parametric nature of this problem,  achieving relative state error below $10^{-2}$ requires ROMs of size $2r\geq44$ and for such ROM sizes the computational cost of simulating $3r-$dimensional quadratic ROM is higher than simulating a $2r-$dimensional nonlinear ROM. However, this higher efficacy of spDEIM for such parametric problems comes with a significant higher computational expense than the structure-preserving lifting approach in the offline phase as the spDEIM approach needs to build and compute SVD of a Jacobian snapshot matrix of dimension $n \times rMK$.

{\Ra{The energy error comparison plots in Figure~\ref{fig:kg_energy} demonstrate that all three approaches achieve energy error below $10^{-1}$.  The bounded energy error behavior for the test parameter $\gamma_{\text{test}}=1.4$ shows the energy-conserving nature of the structure-preserving lifting ROM.}}
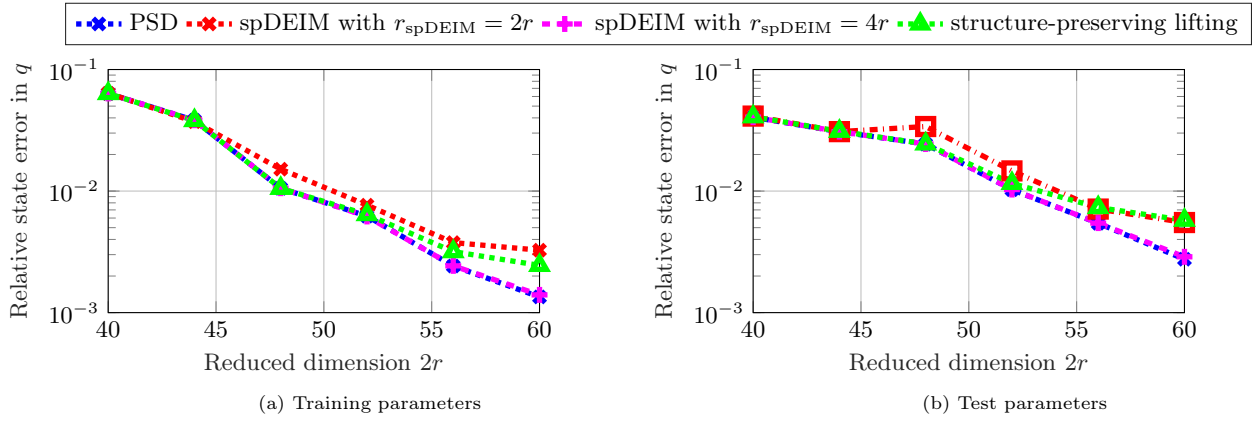
\begin{figure}[tbp]
\small
\captionsetup[subfigure]{oneside,margin={1.8cm,0 cm}}
\begin{subfigure}{.45\textwidth}
       \setlength\fheight{6 cm}
        \setlength\fwidth{\textwidth}
%
%
\definecolor{mycolor1}{rgb}{1.00000,0.00000,1.00000}%
\begin{tikzpicture}

\begin{axis}[%
width=0.951\fheight,
height=0.536\fheight,
at={(0\fheight,0\fheight)},
scale only axis,
xmin=40,
xmax=60,
xlabel style={font=\color{white!15!black}},
xlabel={Reduced dimension $2r$},
ymode=log,
ymin=1e-3,
ymax=0.1,
yminorticks=true,
ylabel style={font=\color{white!15!black}},
ylabel={Relative state error in $q$ },
axis background/.style={fill=white},
xmajorgrids,
ymajorgrids,
legend style={draw=none, legend columns=-1},
legend style={at={(-0.1,1.1)}, anchor=south west, legend cell align=left, align=left, draw=white!15!black}
]
\addplot [color=blue, dotted, line width=2.0pt, mark size=3.0pt, mark=x, mark options={solid, blue}]
  table[row sep=crcr]{%
40	0.0634908741709219\\
44	0.0382162350438569\\
48	0.0105366147387258\\
52	0.00622555964490029\\
56	0.00241892163609106\\
60	0.00134606401788301\\
};
\addlegendentry{PSD}

\addplot [color=red, dotted, line width=2.0pt, mark size=3.0pt, mark=x, mark options={solid, red}]
  table[row sep=crcr]{%
40	0.0638017723095194\\
44	0.0372261537829698\\
48	0.0151374013454008\\
52	0.00765547017701693\\
56	0.00376730466316413\\
60	0.00326442419985602\\
};
\addlegendentry{spDEIM with $r_{\text{spDEIM}}=2r$}

\addplot [color=mycolor1, dashed, line width=2.0pt, mark size=3.0pt, mark=+, mark options={solid, mycolor1}]
  table[row sep=crcr]{%
40	0.0635859569306786\\
44	0.0382803964381847\\
48	0.0105598115141744\\
52	0.00618197152811477\\
56	0.00242484302965896\\
60	0.001399499340521\\
};
\addlegendentry{spDEIM with $r_{\text{spDEIM}}=4r$}

\addplot [color=green, dotted, line width=2.0pt, mark size=3.0pt, mark=triangle, mark options={solid, green}]
  table[row sep=crcr]{%
38	0.0902973909704916\\
40	0.0634715714451674\\
44	0.0380528043642091\\
48	0.0105374504838166\\
52	0.00640898905372596\\
56	0.00317539032137059\\
60	0.00244768636183089\\
};
\addlegendentry{structure-preserving lifting}


\end{axis}

\begin{axis}[%
width=1.227\fheight,
height=0.658\fheight,
at={(-0.16\fheight,-0.072\fheight)},
scale only axis,
xmin=0,
xmax=1,
ymin=0,
ymax=1,
axis line style={draw=none},
ticks=none,
axis x line*=bottom,
axis y line*=left
]
\end{axis}
\end{tikzpicture}%
\caption{Training parameters}
\label{fig:kg_train}
    \end{subfigure}
    \hspace{0.4cm}
    \begin{subfigure}{.45\textwidth}
           \setlength\fheight{6 cm}
           \setlength\fwidth{\textwidth}
\raisebox{-50mm}{
%
%
\definecolor{mycolor1}{rgb}{1.00000,0.00000,1.00000}%
\begin{tikzpicture}

\begin{axis}[%
width=0.951\fheight,
height=0.536\fheight,
at={(0\fheight,0\fheight)},
scale only axis,
xmin=40,
xmax=60,
xlabel style={font=\color{white!15!black}},
xlabel={Reduced dimension $2r$},
ymode=log,
ymin=1e-3,
ymax=0.1,
yminorticks=true,
ylabel style={font=\color{white!15!black}},
ylabel={Relative state error in $q$},
axis background/.style={fill=white},
xmajorgrids,
ymajorgrids,
legend style={legend cell align=left, align=left, draw=white!15!black}
]
\addplot [color=blue, dotted, line width=2.0pt, mark size=3.0pt, mark=x, mark options={solid, blue}]
  table[row sep=crcr]{%
40	0.0404816012770555\\
44	0.0307712204225275\\
48	0.0243596243216795\\
52	0.0102678340922843\\
56	0.00543567837655684\\
60	0.00274958684160939\\
};

\addplot [color=red, dashdotted, line width=2.0pt, mark size=3pt, mark=square, mark options={solid, red}]
  table[row sep=crcr]{%
40	0.0413890795738617\\
44	0.030832962814923\\
48	0.0339633822566597\\
52	0.0146504179311609\\
56	0.00711219437017432\\
60	0.00551011326049215\\
};

\addplot [color=mycolor1, dashed, line width=2.0pt, mark size=3.0pt, mark=+, mark options={solid, mycolor1}]
  table[row sep=crcr]{%
40	0.0405679956072913\\
44	0.0308379719234648\\
48	0.0243691501513758\\
52	0.0102900162276769\\
56	0.00545071049784627\\
60	0.0028788817383856\\
};

\addplot [color=green, dotted, line width=2.0pt, mark size=3.0pt, mark=triangle, mark options={solid, green}]
  table[row sep=crcr]{%
38.6016221434843	0.0543326086910153\\
40	0.0408835135523471\\
44	0.0310556829149094\\
48	0.0244319104686917\\
52	0.0115710069760893\\
56	0.00732555725226403\\
60	0.00575574947723428\\
};


\end{axis}

\begin{axis}[%
width=1.227\fheight,
height=0.658\fheight,
at={(-0.16\fheight,-0.072\fheight)},
scale only axis,
xmin=0,
xmax=1,
ymin=0,
ymax=1,
axis line style={draw=none},
ticks=none,
axis x line*=bottom,
axis y line*=left
]
\end{axis}
\end{tikzpicture}%
}
\caption{Test parameters}
\label{fig:kg_test}
    \end{subfigure}
\caption{Two-dimensional Klein-Gordon equation with parametric dependence. Plots (a) and (b) show that quadratic ROMs obtained via structure-preserving lifting, PSD ROMs, and spDEIM ROMs achieve similar accuracy for both training and test parameters.}
 \label{fig:kg_2d_state}
\end{figure}
\begin{figure}[tbp]
\small
\captionsetup[subfigure]{oneside,margin={1.8cm,0 cm}}
\begin{subfigure}{.45\textwidth}
       \setlength\fheight{6 cm}
        \setlength\fwidth{\textwidth}
%
%
\begin{tikzpicture}

\begin{axis}[%
width=0.951\fheight,
height=0.536\fheight,
at={(0\fheight,0\fheight)},
scale only axis,
xmin=40,
xmax=60,
xlabel style={font=\color{white!15!black}},
xlabel={Reduced dimension $2r$},
ymode=log,
ymin=10,
ymax=10000,
yminorticks=true,
ylabel style={font=\color{white!15!black}},
ylabel={Efficacy},
axis background/.style={fill=white},
xmajorgrids,
ymajorgrids,
legend style={at={(0.138,1.25)}, anchor=south west, legend cell align=left, align=left, draw=white!15!black}
]

\addplot [color=blue, dotted, line width=2.0pt, mark size=3.0pt, mark=x, mark options={solid, blue}]
  table[row sep=crcr]{%
40	40.5268\\
44	43.7074\\
48	159.6407\\
52	157.5488\\
56	311.7804\\
60	505.6590\\
};
\addlegendentry{PSD}

\addplot [color=red, dashed, line width=2.0pt, mark size=3.0pt, mark=+, mark options={solid, red}]
  table[row sep=crcr]{%
40	513.167220201587\\
44	958.866878244416\\
48	2075.32854634491\\
52	3175.6389166633\\
56	6088.37939201631\\
60	6270.56986387954\\
};
\addlegendentry{spDEIM with $r_{\text{spDEIM}}=2r$}

\addplot [color=green, dotted, line width=2.0pt, mark size=3.0pt, mark=triangle, mark options={solid, green}]
  table[row sep=crcr]{%
40	416.475455921293\\
44	378.95458562223\\
48	1028.36918884866\\
52	1230.61442522161\\
56	2093.07973439771\\
60	2050.55547977916\\
};
\addlegendentry{structure-preserving lifting}

\end{axis}

\begin{axis}[%
width=1.227\fheight,
height=0.658\fheight,
at={(-0.16\fheight,-0.072\fheight)},
scale only axis,
xmin=0,
xmax=1,
ymin=0,
ymax=1,
axis line style={draw=none},
ticks=none,
axis x line*=bottom,
axis y line*=left
]
\end{axis}
\end{tikzpicture}%
\caption{Efficacy}
\label{fig:kg_eff}
    \end{subfigure}
    \hspace{0.4cm}
    \begin{subfigure}{.45\textwidth}
           \setlength\fheight{6 cm}
           \setlength\fwidth{\textwidth}
\raisebox{-65mm}{
%
%
\definecolor{mycolor1}{rgb}{0.00000,0.44700,0.74100}%
\definecolor{mycolor2}{rgb}{0.85000,0.32500,0.09800}%
\definecolor{mycolor3}{rgb}{0.92900,0.69400,0.12500}%
\begin{tikzpicture}

\begin{axis}[%
width=0.951\fheight,
height=0.536\fheight,
at={(0\fheight,0\fheight)},
scale only axis,
xmin=0,
xmax=8,
xlabel style={font=\color{white!15!black}},
xlabel={Time $t$},
ymode=log,
ymin=0.0001,
ymax=0.1,
yminorticks=true,
ylabel style={font=\color{white!15!black}},
ylabel={FOM energy error},
axis background/.style={fill=white},
xmajorgrids,
ymajorgrids,
legend style={at={(-0.15,1.25)}, anchor=south west, legend cell align=left, align=left, draw=white!15!black}
]

\addplot [color=blue, dashed, line width=2.0pt]
  table[row sep=crcr]{%
0.199999999999999	2.66803539511784e-05\\
0.300000000000001	0.00062788861717877\\
0.4	0.00260977822625477\\
0.5	0.00601258479599323\\
0.6	0.0100763912549985\\
0.699999999999999	0.0138493796390878\\
0.800000000000001	0.016626708472545\\
0.9	0.0179631798041055\\
1	0.017609145353648\\
1.1	0.0156910165081353\\
1.2	0.013026162738314\\
1.3	0.0109496633291076\\
1.4	0.0101576879254299\\
1.5	0.00998219027937126\\
1.6	0.0098572945220883\\
1.7	0.0103180204835451\\
1.8	0.0114082422779484\\
1.9	0.0123061780364389\\
2	0.0127415554800939\\
2.2	0.0132698838563556\\
2.3	0.0134790831106545\\
2.5	0.0137973204538696\\
2.7	0.0140464613037818\\
3	0.014338488213225\\
3.4	0.0146431910447814\\
3.9	0.0149482514923364\\
4.5	0.0152491468388132\\
5.2	0.0155407528832382\\
6	0.0158177483443456\\
7	0.0161041931443788\\
8	0.0163419697706672\\
};
\addlegendentry{PSD ROM $2r=60$}
\addplot [color=red, dashdotted, line width=2.0pt]
  table[row sep=crcr]{%
0.199999999999999	6.76432430782369e-05\\
0.300000000000001	0.000790115217762376\\
0.4	0.00283797247925008\\
0.5	0.00650260113213336\\
0.6	0.01121365724979\\
0.699999999999999	0.0148598727480976\\
0.800000000000001	0.0162635159027423\\
0.9	0.0169122043160405\\
1	0.0172287033338273\\
1.1	0.0153561280127394\\
1.2	0.0123167958063925\\
1.3	0.0112846257833223\\
1.4	0.0117876319079369\\
1.5	0.0114354168488728\\
1.6	0.0101688981211021\\
1.7	0.00997626598819579\\
1.9	0.0133693916809852\\
2	0.0145129794245429\\
2.1	0.01473325591449\\
2.4	0.0132269867099717\\
2.5	0.0120850451530191\\
2.6	0.0112286891921542\\
2.7	0.0118794019788754\\
2.8	0.0131079150576977\\
2.9	0.0137660326902733\\
3.1	0.0151005665222681\\
3.2	0.0150696574652375\\
3.3	0.014745771792634\\
3.4	0.0145487359327853\\
3.5	0.0140453698434988\\
3.6	0.0137811451415337\\
3.7	0.0146524752044433\\
3.8	0.0157424373702202\\
3.9	0.01550555222723\\
4	0.0142205838829057\\
4.1	0.0135226865630023\\
4.2	0.0140999126802671\\
4.3	0.0151569339050252\\
4.4	0.0156988638448183\\
4.5	0.0154275696436434\\
4.6	0.0149662887132195\\
4.7	0.0152405798301038\\
4.8	0.0160109639322183\\
4.9	0.0159012917099574\\
5	0.0146715675559813\\
5.1	0.0141195062460974\\
5.2	0.0153906722054467\\
5.3	0.0169158539747877\\
5.4	0.0167591470477646\\
5.5	0.015337457480814\\
5.6	0.0143147217094486\\
5.7	0.0144414031558665\\
5.8	0.0154143557519365\\
5.9	0.0165171952860646\\
6	0.0167898945580703\\
6.1	0.0155888041218191\\
6.2	0.0134593880253726\\
6.3	0.0115401968025731\\
6.4	0.0100935641213403\\
6.5	0.00886994792099812\\
6.6	0.00840596318504026\\
6.7	0.00929815211028767\\
6.8	0.0107070047598063\\
6.9	0.0111780620233367\\
7	0.0105458612861025\\
7.1	0.00976504589499313\\
7.2	0.00936907771624193\\
7.3	0.00932487014724755\\
7.4	0.00984890919455097\\
7.5	0.0109872842697109\\
7.6	0.0116945512894824\\
7.7	0.0105570808335506\\
7.8	0.00742491219103272\\
7.9	0.00353163826738182\\
8	0.000156501693604696\\
};
\addlegendentry{spDEIM ROM $2r=60$ ($r_{\text{spDEIM}}=2r$)}

\addplot [color=green, solid, line width=2.0pt]
  table[row sep=crcr]{%
0	0\\
0.1	0.000111870866179515\\
0.2	0.00197158020992788\\
0.3	0.00798882948934875\\
0.4	0.0182495656653714\\
0.5	0.0304291648059814\\
0.6	0.0417278031982823\\
0.7	0.0500603545174863\\
0.8	0.0541045937436456\\
0.9	0.053111246079352\\
1	0.0474126808297706\\
1.1	0.0393993697241297\\
1.2	0.0330470494605788\\
1.3	0.0305495813350774\\
1.4	0.0300021969249531\\
1.5	0.0296633824164655\\
1.6	0.0310694677940069\\
1.7	0.0343173014062836\\
1.8	0.0369830021462167\\
1.9	0.0382899143994894\\
2	0.0391177842999605\\
2.1	0.0398700163597812\\
2.2	0.0404856470433685\\
2.3	0.0409823979771295\\
2.4	0.0414205295540455\\
2.5	0.0418151833980278\\
2.6	0.0421682678726313\\
2.7	0.0424832224596003\\
2.8	0.0427664295041359\\
2.9	0.0430309740404869\\
3	0.0432764694693549\\
3.1	0.0435008453861872\\
3.2	0.0437191563974966\\
3.3	0.0439362212723177\\
3.4	0.0441375170520905\\
3.5	0.0443211860342296\\
3.6	0.0444995384706328\\
3.7	0.0446764534641784\\
3.8	0.0448477341690716\\
3.9	0.0450091815680673\\
4	0.0451625849918025\\
4.1	0.0453127138801142\\
4.2	0.0454580693852836\\
4.3	0.045599520032481\\
4.4	0.0457429698805356\\
4.5	0.0458841538607362\\
4.6	0.0460153921994424\\
4.7	0.0461390883754393\\
4.8	0.0462585548689617\\
4.9	0.0463755533058748\\
5	0.0464955778621186\\
5.1	0.0466174134839991\\
5.2	0.0467323938652851\\
5.3	0.0468404408821152\\
5.4	0.046947552362343\\
5.5	0.047053462852258\\
5.6	0.0471560531257899\\
5.7	0.0472565545283828\\
5.8	0.04735492207569\\
5.9	0.0474497628287134\\
6	0.0475412150519514\\
6.1	0.0476305951540667\\
6.2	0.0477215611310169\\
6.3	0.0478157784309212\\
6.4	0.047906308237026\\
6.5	0.0479882940698781\\
6.6	0.0480688311873212\\
6.7	0.04815202114263\\
6.8	0.0482319712081949\\
6.9	0.0483091510602401\\
7	0.0483889526537325\\
7.1	0.0484677967955963\\
7.2	0.0485426443944243\\
7.3	0.0486168616440421\\
7.4	0.048688900394568\\
7.5	0.0487589657642434\\
7.6	0.0488333544885506\\
7.7	0.0489041185146118\\
7.8	0.0489618667397804\\
7.9	0.0490280854506875\\
};
\addlegendentry{structure-preserving lifting ROM $3r=90$}

\end{axis}

\begin{axis}[%
width=1.227\fheight,
height=0.658\fheight,
at={(-0.16\fheight,-0.072\fheight)},
scale only axis,
xmin=0,
xmax=1,
ymin=0,
ymax=1,
axis line style={draw=none},
ticks=none,
axis x line*=bottom,
axis y line*=left
]
\end{axis}
\end{tikzpicture}%
}
\caption{FOM energy error for $\gamma_{\text{test}}=1.4$}
\label{fig:kg_energy}
    \end{subfigure}
\caption{Two-dimensional Klein-Gordon equation with parametric dependence. Plot (a) shows that the spDEIM approach yields ROMs with higher efficacy than the structure-preserving lifting approach {\Ra{whereas the PSD approach yields ROMs with the lowest efficacy}}. {\Ra{The energy error comparison in plot (b) shows that all three ROMs yield energy error below $10^{-1}$ with the PSD ROM and the spDEIM ROM performing marginally better than the structure-preserving lifting ROM.}}}
 \label{fig:kg_2d}
\end{figure}
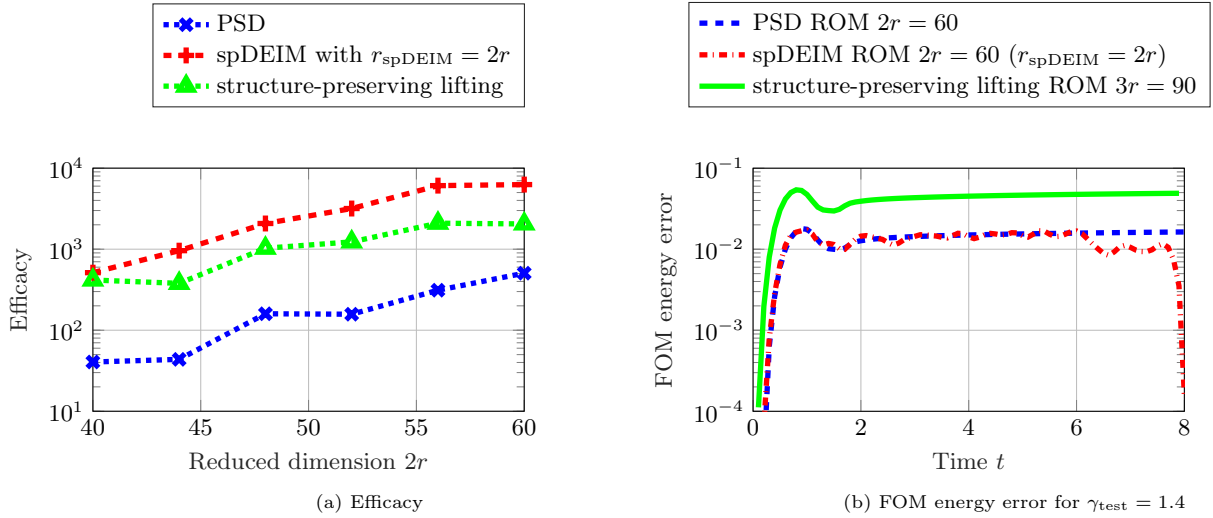
\subsection{Two-dimensional Klein-Gordon-Zakharov equations}
\label{sec:kgz}
The Klein-Gordon-Zakharov (KGZ) equations~\cite{boling1995global} are a system of nonlinear dispersive PDEs used to describe the mutual interaction between the Langmuir waves and ion acoustic waves in a plasma. These coupled equations play a crucial role in the study of the dynamics of strong Langmuir turbulence in plasma physics~\cite{berge1996perturbative}. The KGZ equations can be derived from the two-fluid Euler-Maxwell equations for the electrons, ions, and electric field, by first neglecting the magnetic field and also assuming that ions move much slower than electrons. These equations reduce to the Zakharov equations in the high-frequency limit and converge to the Klein-Gordon equation in the subsonic limit. 
\subsubsection{System of coupled PDEs and the corresponding nonlinear conservative FOM}
We consider the FOM setup from~\cite{guo2023energy}. Let $\Omega=(-20,20) \times (-20,20) \subset \real^2$ be the spatial domain and consider the coupled nonlinear equations
\begin{align*}
\frac{\partial ^2\psi}{\partial t^2}(x,y,t)&=\frac{\partial ^2\psi}{\partial x^2}(x,y,t)+\frac{\partial ^2\psi}{\partial y^2}(x,y,t)-\psi(x,y,t) -\psi(x,y,t)\phi(x,y,t)-|\psi(x,y,t)|^2\psi(x,y,t), \\
 \frac{\partial ^2\phi }{\partial t^2}(x,y,t)&=\frac{\partial ^2\phi}{\partial x^2}(x,y,t)+\frac{\partial ^2\phi}{\partial y^2}(x,y,t) + \left(\frac{\partial ^2}{\partial x^2}+  \frac{\partial ^2}{\partial y^2}\right) \left( |\psi (x,y,t)|^2\right),
\end{align*}
where $t \in (0,T]$ is time, $\psi(x,y,t)$ is a complex-valued scalar field that describes the fast time scale component of the electric field raised by electrons, and $\phi(x,y,t)$ is a real-valued scalar field that describes the deviation of the ion density from its equilibrium. The boundary conditions are periodic and the initial conditions are
\begin{align*}
\psi(x,y,0)&=\sech(-(x-2)^2-y^2) + \sech(-x^2-(y-2)^2), \qquad \frac{\partial \psi}{\partial t}(x,y,0)=0,\\
 \phi (x,y,0)&=\sech(-(x-2)^2-y^2) + \sech(-x^2-(y-2)^2), \qquad \frac{\partial \phi}{\partial t}(x,y,0)=0.
\end{align*}
The coupled system of PDEs conserves the total energy
\begin{equation*} 
 \mathcal E[\psi(x,y,t),\phi(x,y,t)]:=\int_{\Omega} \left( \bigg |\frac{\partial \psi}{\partial t}\bigg|^2 + |\nabla \psi|^2  +|\psi|^2  + \phi |\psi|^2
 +\frac{1}{2}|\nabla \varphi|^2 + \frac{1}{2}\phi^2 + \frac{1}{2} |\psi|^4  \right) {\dd} x {\dd} y,
\end{equation*}
where $\varphi:=\varphi (x,y,t)$ is defined via $\Delta \varphi (x,y,t)=\frac{\partial \phi (x,y,t)}{\partial t}$  with $\lim_{|\sqrt{x^2+y^2}|\to \infty}\varphi(x,y,t)=0$. To rewrite the KGZ equations in first-order form, we first write the complex-valued function $\psi$ in terms of its real and imaginary parts as $\psi=q_1+iq_2$ and then define $p_1=\partial q_1/\partial t$ and $p_2=\partial q_2/\partial t$. The resulting system of six coupled PDEs is
\begin{align*}
\frac{\partial q_1}{\partial t}&=p_1,  \qquad \ \frac{\partial p_1} {\partial t}= \mathbf \Delta q_1-q_1 -q_1\phi-(q_1^2+q_2^2)q_1,\\
\frac{\partial q_2}{\partial t}&=p_2, \qquad \ \frac{\partial p_2}{\partial t}=\mathbf \Delta q_2-q_2 -q_2\phi-(q_1^2+q_2^2)q_2, \\
\frac{\partial \phi}{\partial t}&=\mathbf \Delta \varphi, \qquad \frac{\partial \varphi}{\partial t}=\phi+  (q_1^2+q_2^2),
\end{align*}
where we use the notation $\mathbf \Delta$ to denote the Laplacian operator in $\real^2$.

We discretize the two-dimensional spatial domain $\Omega$ with $n_x=n_y=400$ equally spaced grid points in both spatial directions to derive the nonlinear FOM of dimension $6n=6n_xn_y=960{,}000$
\begin{align*}
\dot{\q_1}&=\p_1,\\
\dot{\q_2}&=\p_2,\\
\dot{\p_1}&=\D\q_1-\q_1-\bphi\odot\q_1 - (\q_1^2+\q_2^2)\odot \q_1,\\
\dot{\p_2}&=\D\q_2-\q_2-\bphi\odot\q_2 - (\q_1^2+\q_2^2)\odot\q_2,\\
\dot{\bvarphi}&=\bphi + (\q_1^2+\q_2^2),\\
\dot{\bphi}&=\D\bvarphi,
\end{align*}
with the space-discretized energy 
\begin{equation*} 
 E(\q_1,\q_2,\p_1,\p_2,\bvarphi,\bphi)=\p_1^\top\p_1 +\p_2^\top\p_2 + \q_1^\top\q_1 +\q_2^\top\q_2 - \q_1^\top\D\q_1 -\q_2^\top\D\q_2+ \bphi^\top (\q_1^2+\q_2^2)
 -\frac{1}{2}\bvarphi^\top\D\bvarphi + \frac{1}{2}\bphi^\top\bphi + {\Ra{\sum_{i=1}^n\frac{1}{2} (q_{1,i}^2+q_{2,i}^2)^2}}.
 \end{equation*} 
 Unlike the previous three numerical examples, the nonlinear conservative FOM for the KGZ equations is not of the form~\eqref{eq:cons_fom} as it does not have a canonical Hamiltonian formulation.
\subsubsection{Structure-preserving lifting based on energy quadratization}
We follow the proposed energy-quadratization strategy and introduce only one auxiliary variable $\w=\q_1^2+\q_2^2$ to lift the nonlinear FOM to a quadratic lifted FOM 
\begin{align*}
\dot{\q_1}&=\p_1,\\
\dot{\q_2}&=\p_2,\\
\dot{\p_1}&=\D\q_1-\q_1-\bphi\odot\q_1 - \w \odot \q_1,\\
\dot{\p_2}&=\D\q_2-\q_2-\bphi \odot\q_2 -\w \odot \q_2,\\
\dot{\bvarphi}&=\bphi + \w,\\
\dot{\bphi}&=\D\bvarphi,\\
\dot{\w}   &=2 \q_1\odot \p_1 + 2\q_2 \odot \p_2,
\end{align*}
with a quadratic FOM energy in the lifted state variables
\begin{equation*} 
 E_{\text{lift}}(\q_1,\q_2,\p_1,\p_2,\bvarphi,\bphi,\w)=\p_1^\top\p_1 +\p_2^\top\p_2 + \q_1^\top\q_1 +\q_2^\top\q_2 - \q_1^\top\D\q_1 -\q_2^\top\D\q_2+ \bphi^\top \w
 -\frac{1}{2}\bvarphi^\top\D\bvarphi + \frac{1}{2}\bphi^\top\bphi 
 + \frac{1}{2} \w^\top\w. 
\end{equation*}
To approximate the lifted FOM state, we consider a block-diagonal projection matrix $\bar{\V} $ of the form
\begin{equation*}
\bar{\V} =\text{blkdiag}(\bPhi,\bPhi,\bPhi,\bPhi,{\Ra{\V,\V,\V_1}}) \in \real^{7n \times 7r},
\end{equation*}
where $\bPhi \in \real^{n\times r}$ is computed from the FOM snapshots of $\q_1,\q_2, \p_1,$ and $\p_2$, ${\Ra{\V}}\in \real^{n \times r}$ is computed from the FOM snapshots of $\bvarphi$ and $\bphi$, and ${\Ra{\V_{1}}} \in \real^{n \times r}$ is computed from the lifted $\w$ snapshots. Substituting the lifted FOM state followed by Galerkin projection leads to  
\begin{align*}
\dot{\qhat}_1&=\phat_1,\\
\dot{\qhat}_2&=\phat_2,\\
\dot{\phat}_1&=\Dhat_1\qhat_1-\qhat_1-\bPhi^\top(\V\bphihat\odot \bPhi \qhat_1) - \bPhi^\top( \V_1\what\odot \bPhi \qhat_1),\\
\dot{\phat}_2&=\Dhat_1\qhat_2-\qhat_2-\bPhi^\top(\V\bphihat\odot \bPhi \qhat_2) - \bPhi^\top( \V_1\what\odot \bPhi \qhat_2),\\
\dot{\bvarphihat}&=\bphihat + \V^\top\V_1\what,\\
\dot{\bphihat}&=\Dhat_2\bvarphihat,\\
\dot{\what}   &=2\V_{1}^\top\left( \bPhi \qhat_1\odot \bPhi \phat_1 +  \bPhi\qhat_2 \odot \bPhi \phat_2\right),
\end{align*}
where $\Dhat_1=\bPhi^\top\D\bPhi$  and $\Dhat_2=\V^\top\D\V$. We substitute the lifted FOM state approximation into the lifted FOM energy expression and then compute its time derivative 
\begin{align*} 
    \frac{\dd}{{\dd}t}E_{\text{lift}}(\bPhi\qhat_1,\bPhi\qhat_2,\bPhi\phat_1,\bPhi\phat_2,\V\bvarphihat,\V\bphihat,\V_{1}\what)
    &=2\bPhi^\top \left( \V^\top\V_{1}\V_{1}^\top -\V^\top \right) \left( \bPhi \qhat_1\odot \bPhi \phat_1 +  \bPhi\qhat_2 \odot \bPhi \phat_2\right).
\end{align*}

The residual term in the above equation vanishes only for the special case of $\V=\V_{1}$. To ensure the quadratic ROM conserves the lifted FOM energy for all $t \in (0,T]$, we compute a joint basis matrix $\V_{\text{joint}}$ via SVD of the concatenated matrix consisting of snapshots of $\bphi$, $\bvarphi$, and $\w$. Thus, the structure-preserving lifting transformation $\w=(\q_1^2+\q_2^2)$ combined with a block-diagonal basis matrix of the form
\begin{equation*}
\bar{\V} =\text{blkdiag}(\bPhi,\bPhi,\bPhi,\bPhi,\V_{\text{joint}},\V_{\text{joint}},\V_{\text{joint}}) \in \real^{7n \times 7r},
\end{equation*}
yields a $7r-$dimensional energy-conserving quadratic ROM for the KGZ equations, i.e.,
\begin{align*} 
    \frac{\dd}{{\dd}t}E_{\text{lift}}(\bPhi\qhat_1,\bPhi\qhat_2,\bPhi\phat_1,\bPhi\phat_2,\V_{\text{joint}}\bvarphihat,\V_{\text{joint}}\bphihat,\V_{\text{joint}}\what)
    &=0.
\end{align*}
\subsubsection{Numerical results}
Due to the high-dimensional nature of the nonlinear FOM in this example, the discrete equations for time-marching are solved using Picard iterations as this approach does not require the computation of a Jacobian matrix at every time step. Numerical time
integration of the nonlinear FOM with $960{,}000$ degrees of freedom from $t=0$ to $t=5$ using $\dt=0.01$ requires approximately $89$~\si{\min} (MATLAB wall clock time) on the Triton Shared Computing Cluster~\cite{san2022triton} equipped with $8$ processing cores of Intel Xeon Platinum 64-core CPU at 2.9 GHz and 1 TB memory.

In this numerical example, we build a training dataset consisting of FOM snapshots  from $t=0$ to $t=4$. 
We consider a test dataset consisting of FOM snapshots from $t=4$ to $t=5$, {\Rb{which is 25\% outside the training time interval}} to study the time extrapolation capability of the structure-preserving lifting approach. {\Rb{Due to the transport-dominated nature of this problem, the test dataset in this problem features propagating wave fronts that travel outside the region covered in the training dataset}}. Due to the non-canonical nature of KGZ equations, the spDEIM approach from~\cite{pagliantini2023gradient} is not applicable to this example as the construction of the DEIM basis in spDEIM assumes the FOM is derived from a canonical Hamiltonian PDE. Therefore, we only consider structure-preserving lifting ROMs and PSD ROMs for this example.  

In Figure~\ref{fig:kgz_state_q}, we compare the relative state error of the structure-preserving lifting approach against PSD for the complex-valued scalar field $\psi(x,y,t)$. For training data, the relative state error plots in Figure~\ref{fig:state_train_q} show that the quadratic ROMs obtained via structure-preserving lifting perform similar to the PSD ROMs up to $6r=150$. For $6r>150$, we observe that the state error levels off in the training data regime for the structure-preserving lifting approach. In Figure~\ref{fig:state_test_q}, both structure-preserving lifting ROMs and PSD ROMs exhibit similar accuracy in the test data regime. Figure~\ref{fig:kgz_state_phi} shows relative state error for the real-valued scalar field $\phi (x,y,t)$ with increasing ROM order. The comparison in Figure~\ref{fig:state_train_phi}  shows that structure-preserving lifting ROMs and PSD ROMs achieve similar accuracy in the training data regime up to $6r=270$. For $6r>270$, we observe that the state error levels off for the structure-preserving lifting approach. In Figure~\ref{fig:state_test_phi}, we observe that both approaches fail to achieve relative state error below $10^{-1}$ in the test data regime, even for ROM dimension $6r=360$. This lack of predictive capability is primarily due to the transport nature of this problem where the linear basis for both approaches fails to provide accurate state approximations in the test data regime, see~\cite{peherstorfer2022breaking} for more details about limitations of using ROMs based on linear subspaces for transport-dominated problems. 

In Figure~\ref{fig:proj_q} and Figure~\ref{fig:proj_phi}, we compare the time evolution of the projection error in $\psi$ and $\phi$ for both structure-preserving lifting and PSD ROMs. The projection error plots in Figure~\ref{fig:proj_q} show that the basis matrix for $\psi$ provides accurate approximations with relative error below $10^{-13}$ for both approaches in the training data regime. For test data, the basis matrix yields substantially higher projection error in the test data regime with relative error higher than $10^{-2}$ at $t=5$. In Figure~\ref{fig:proj_phi}, we observe that both approaches provide accurate approximations for $\phi$ with relative error below $10^{-6}$ in the training data regime. However, the projection error for both approaches grows significantly in the test data regime with relative error higher than $10^{-1}$ at $t=5$. The substantial growth in the projection error for both $\psi$ and $\phi$ in the test data regime corroborates our explanation for both approaches failing to provide accurate predictions for this transport-dominated problem. The energy error comparison in Figure~\ref{fig:energy} shows that both the quadratic ROM based on the structure-preserving lifting approach and the PSD ROM achieve FOM energy error below $10^{-2}$ with bounded energy error in the test data regime. 

Finally, we compare the FOM solution and the approximate ROM solution for $\psi(x,y,t)$ and $\phi(x,y,t)$ in Figure~\ref{fig:kgz_psi} and Figure~\ref{fig:kgz_phi}, respectively. Figure~\ref{fig:kgz_psi} shows that the structure-preserving quadratic ROM of dimension $7r=420$ accurately captures the time-evolution of the complex-valued scalar field $\psi(x,y,t)$ and yields accurate approximate solution even at $t=4.5$, which is 12.5\% outside the training time interval. Figure~\ref{fig:kgz_phi} shows that the proposed approach provides accurate approximation of $\phi(x,y,t)$ solution over the two-dimensional computational domain at $t=1.5$, $t=3$, and $t=4.5$. Compared to the approximate FOM run time of $89$~\si{\min}, numerical time integration of the structure-preserving quadratic ROM of size $7r=420$ requires approximately $78$~\si{\s} (MATLAB wall clock time averaged over 20 runs), which is a factor
of $68\times$ speedup.
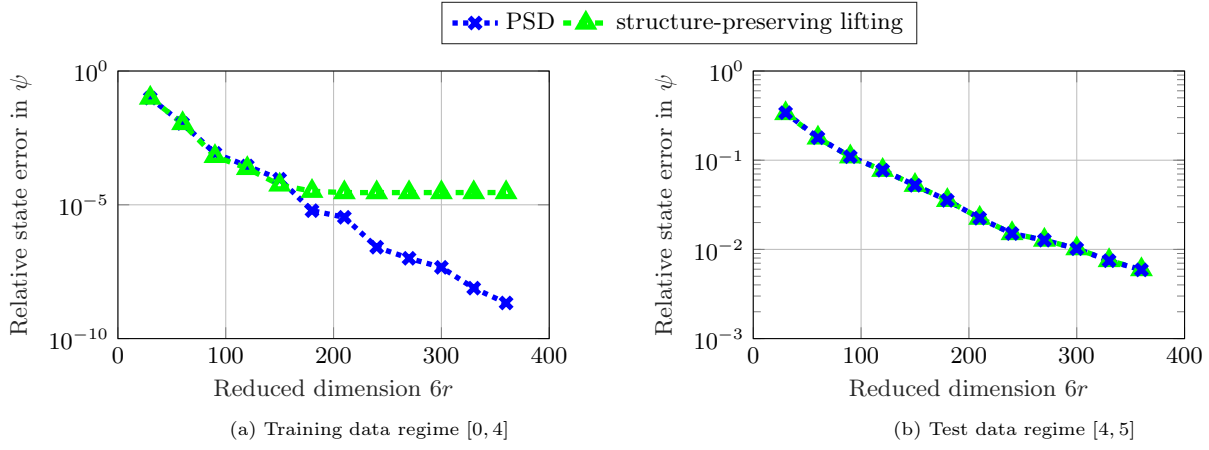
\begin{figure}[tbp]
\small
\captionsetup[subfigure]{oneside,margin={1.8cm,0 cm}}
\begin{subfigure}{.45\textwidth}
       \setlength\fheight{6 cm}
        \setlength\fwidth{\textwidth}
%
%
\begin{tikzpicture}

\begin{axis}[%
width=0.951\fheight,
height=0.59\fheight,
at={(0\fheight,0\fheight)},
scale only axis,
xmin=0,
xmax=400,
xlabel style={font=\color{white!15!black}},
xlabel={Reduced dimension $6r$},
ymode=log,
ymin=1e-10,
ymax=1,
yminorticks=true,
ylabel style={font=\color{white!15!black}},
ylabel={Relative state error in $\psi$},
axis background/.style={fill=white},
xmajorgrids,
ymajorgrids,
yminorgrids,
legend style={draw=none, legend columns=-1},
legend style={at={(0.75,1.1)}, anchor=south west, legend cell align=left, align=left, draw=white!15!black}
]

\addplot [color=blue, dotted, line width=2.0pt, mark size=3.0pt, mark=x, mark options={solid, blue}]
  table[row sep=crcr]{%
30	0.0973508267835936\\
60	0.0109763883744027\\
90	0.000839164320884976\\
120	0.000281886485202453\\
150	9.06210785814612e-05\\
180	6.04361626661903e-06\\
210	3.41262315717982e-06\\
240	2.60371019957538e-07\\
270	1.0016310461032e-07\\
300	4.60219717143753e-08\\
330	7.68066215324989e-09\\
360	2.14539667450598e-09\\
};
\addlegendentry{PSD}
\addplot [color=green, dashed, line width=2.0pt, mark size=3.0pt, mark=triangle, mark options={solid, green}]
  table[row sep=crcr]{%
30	0.0976326762239408\\
60	0.01094500500749\\
90	0.000645931701140754\\
120	0.000229076240930473\\
150	5.53977391438002e-05\\
180	3.21451520520133e-05\\
210	2.85916742514566e-05\\
240	2.85418440332051e-05\\
270	2.86271905511806e-05\\
300	2.86575670392268e-05\\
330	2.86960302488893e-05\\
360	2.87321804460762e-05\\
};
\addlegendentry{structure-preserving lifting}

\end{axis}

\begin{axis}[%
width=1.227\fheight,
height=0.723\fheight,
at={(-0.16\fheight,-0.08\fheight)},
scale only axis,
xmin=0,
xmax=1,
ymin=0,
ymax=1,
axis line style={draw=none},
ticks=none,
axis x line*=bottom,
axis y line*=left
]
\end{axis}
\end{tikzpicture}%
\caption{Training data regime $[0,4]$}
\label{fig:state_train_q}
    \end{subfigure}
    \hspace{0.4cm}
    \begin{subfigure}{.45\textwidth}
           \setlength\fheight{6 cm}
           \setlength\fwidth{\textwidth}
\raisebox{-53mm}{
%
%
\begin{tikzpicture}

\begin{axis}[%
width=0.951\fheight,
height=0.59\fheight,
at={(0\fheight,0\fheight)},
scale only axis,
xmin=0,
xmax=400,
xlabel style={font=\color{white!15!black}},
xlabel={Reduced dimension $6r$},
ymode=log,
ymin=0.001,
ymax=1,
yminorticks=true,
ylabel style={font=\color{white!15!black}},
ylabel={Relative state error in $\psi$},
axis background/.style={fill=white},
xmajorgrids,
ymajorgrids,
legend style={legend cell align=left, align=left, draw=white!15!black}
]
\addplot [color=green, dashed, line width=2.0pt, mark size=3.0pt, mark=triangle, mark options={solid, green}]
  table[row sep=crcr]{
30	0.333969451510617\\
60	0.177859571523418\\
90	0.109319283302542\\
120	0.0772185752599888\\
150	0.0524104004671025\\
180	0.0354941635565243\\
210	0.0224121180876265\\
240	0.0150795348591646\\
270	0.0127200761046453\\
300	0.0102009285357374\\
330	0.00751008584317494\\
360	0.00594239982295015\\
};

\addplot [color=blue, dotted, line width=2.0pt, mark size=3.0pt, mark=x, mark options={solid, blue}]
  table[row sep=crcr]{%
30	0.340734488159584\\
60	0.178539508377701\\
90	0.109444440895849\\
120	0.0772700106284573\\
150	0.0523488045421839\\
180	0.0354656291491545\\
210	0.0223663239839059\\
240	0.0150478101482762\\
270	0.0126963286447042\\
300	0.0101800475943521\\
330	0.00746928703912376\\
360	0.00589702966904651\\
};

\end{axis}
\end{tikzpicture}%
}
\caption{Test data regime $[4,5]$}
\label{fig:state_test_q}
    \end{subfigure}\caption{Klein-Gordon-Zakharov equations. Plot (a) shows that both structure-preserving lifting ROMs and PSD ROMs achieve similar accuracy for $\psi$ in the training data regime up to $6r=150$. For $6r>150$, the relative state error for the structure-preserving lifting approach does not decrease as favorably with an increase in the reduced dimension. Plot (b) shows that both approaches achieve similar state error performance in the test data regime.}
 \label{fig:kgz_state_q}
\end{figure}
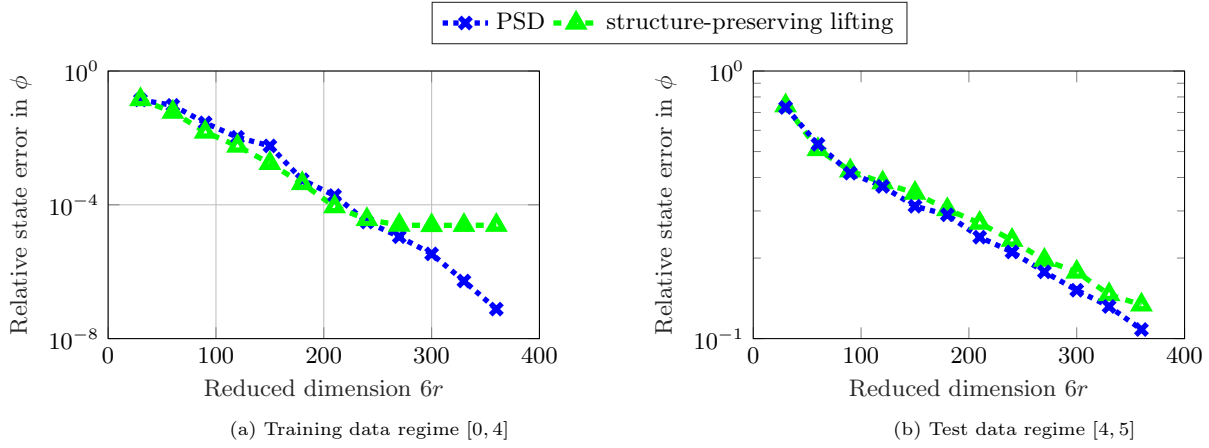
\begin{figure}[tbp]
\small
\captionsetup[subfigure]{oneside,margin={1.8cm,0 cm}}
    \begin{subfigure}{.45\textwidth}
       \setlength\fheight{6 cm}
        \setlength\fwidth{\textwidth}
%
%
\begin{tikzpicture}

\begin{axis}[%
width=0.951\fheight,
height=0.59\fheight,
at={(0\fheight,0\fheight)},
scale only axis,
xmin=0,
xmax=400,
xlabel style={font=\color{white!15!black}},
xlabel={Reduced dimension $6r$},
ymode=log,
ymin=1e-08,
ymax=1,
yminorticks=true,
ylabel style={font=\color{white!15!black}},
ylabel={Relative state error in $\phi$},
axis background/.style={fill=white},
xmajorgrids,
ymajorgrids,
legend style={draw=none, legend columns=-1},
legend style={at={(0.75,1.1)}, anchor=south west, legend cell align=left, align=left, draw=white!15!black}
]

\addplot [color=blue, dotted, line width=2.0pt, mark size=3.0pt, mark=x, mark options={solid, blue}]
  table[row sep=crcr]{%
30	0.136779287290138\\
60	0.0939888556598516\\
90	0.0281575689243546\\
120	0.0104343745556966\\
150	0.0057638389665394\\
180	0.000559757316623709\\
210	0.000189102677138314\\
240	3.06276358743014e-05\\
270	1.1164857558811e-05\\
300	3.4068564518606e-06\\
330	5.27527279016873e-07\\
360	7.62553868409222e-08\\
};
\addlegendentry{PSD}

\addplot [color=green, dashed, line width=2.0pt, mark size=3.0pt, mark=triangle, mark options={solid, green}]
  table[row sep=crcr]{%
30	0.140215223699327\\
60	0.058967236471198\\
90	0.0150084692853293\\
120	0.00571783789007595\\
150	0.00173415616775348\\
180	0.000434134224001606\\
210	8.82048988394856e-05\\
240	3.64495488827428e-05\\
270	2.47109098074387e-05\\
300	2.41109643266656e-05\\
330	2.44298496340645e-05\\
360	2.45485191145124e-05\\
};
\addlegendentry{structure-preserving lifting}

\end{axis}

\begin{axis}[%
width=1.227\fheight,
height=0.723\fheight,
at={(-0.16\fheight,-0.08\fheight)},
scale only axis,
xmin=0,
xmax=1,
ymin=0,
ymax=1,
axis line style={draw=none},
ticks=none,
axis x line*=bottom,
axis y line*=left
]
\end{axis}
\end{tikzpicture}%
\caption{Training data regime $[0,4]$}
\label{fig:state_train_phi}
    \end{subfigure}
    \hspace{0.4cm}
    \begin{subfigure}{.45\textwidth}
           \setlength\fheight{6 cm}
           \setlength\fwidth{\textwidth}
\raisebox{-53mm}{
%
%
\begin{tikzpicture}

\begin{axis}[%
width=0.951\fheight,
height=0.59\fheight,
at={(0\fheight,0\fheight)},
scale only axis,
xmin=0,
xmax=400,
xlabel style={font=\color{white!15!black}},
xlabel={Reduced dimension $6r$},
ymode=log,
ymin=0.1,
ymax=1,
yminorticks=true,
ylabel style={font=\color{white!15!black}},
ylabel={Relative state error in $\phi$},
axis background/.style={fill=white},
legend style={legend cell align=left, align=left, draw=white!15!black}
]
\addplot [color=green, dashed, line width=2.0pt, mark size=3.0pt, mark=triangle, mark options={solid, green}]
  table[row sep=crcr]{%
30	0.74202777726921\\
60	0.509024955419476\\
90	0.423784664574767\\
120	0.382584946028884\\
150	0.349676817326932\\
180	0.303699862434508\\
210	0.270222962817022\\
240	0.23323100422575\\
270	0.196463153841104\\
300	0.177387052334283\\
330	0.14613533244765\\
360	0.133708784520085\\
};

\addplot [color=blue, dotted, line width=2.0pt, mark size=3.0pt, mark=x, mark options={solid, blue}]
  table[row sep=crcr]{%
30	0.730774703288248\\
60	0.532159228566006\\
90	0.415193732527051\\
120	0.37052567025014\\
150	0.312370566022836\\
180	0.290218067826112\\
210	0.239388628135366\\
240	0.210738831830813\\
270	0.177604078650503\\
300	0.151637726198157\\
330	0.131814986023554\\
360	0.10814316983368\\
};

\end{axis}

\begin{axis}[%
width=1.227\fheight,
height=0.723\fheight,
at={(-0.16\fheight,-0.08\fheight)},
scale only axis,
xmin=0,
xmax=1,
ymin=0,
ymax=1,
axis line style={draw=none},
ticks=none,
axis x line*=bottom,
axis y line*=left
]
\end{axis}
\end{tikzpicture}%
}
\caption{Test data regime $[4,5]$}
\label{fig:state_test_phi}
    \end{subfigure}
\caption{Klein-Gordon-Zakharov equations. Plot (a) shows that both structure-preserving lifting ROMs and PSD ROMs achieve similar accuracy for $\phi$ in the training data regime up to $6r=270$. The relative state error for the structure-preserving lifting approach levels off for $6r>270$. Plot (b) shows that both approaches fail to achieve relative state error below $10^{-1}$ in the test data regime.}
 \label{fig:kgz_state_phi}
\end{figure}
\begin{figure}[tbp]
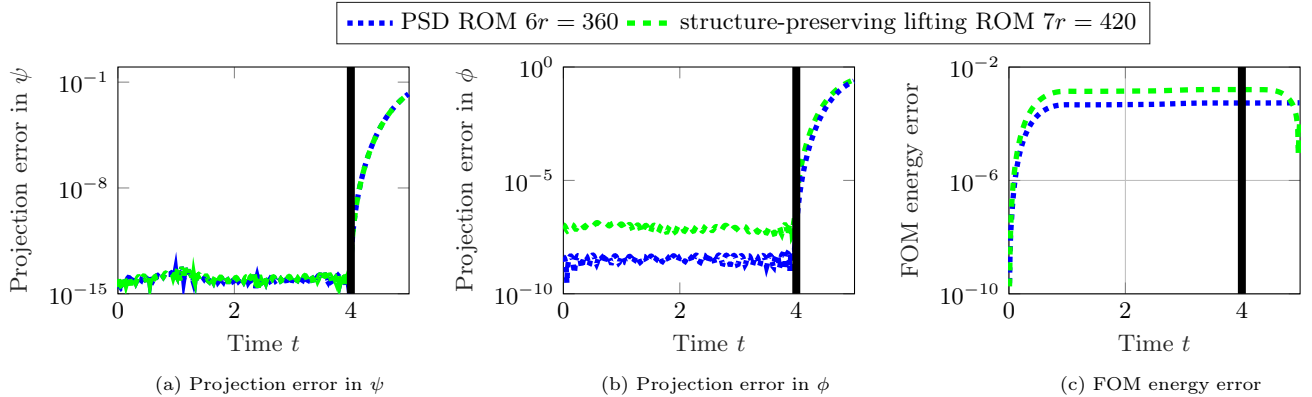

\small
\captionsetup[subfigure]{oneside,margin={1.8cm,0 cm}}
    \begin{subfigure}{.3\textwidth}
       \setlength\fheight{4 cm}
        \setlength\fwidth{\textwidth}
\input{figures/kgz/finest/proj_q.tex}
\caption{Projection error in $\psi$}
\label{fig:proj_q}
    \end{subfigure}
    \hspace{0.4cm}
    \begin{subfigure}{.3\textwidth}
           \setlength\fheight{4 cm}
           \setlength\fwidth{\textwidth}
\raisebox{-48mm}{\input{figures/kgz/finest/proj_phi.tex}}
\caption{Projection error in $\phi$}
\label{fig:proj_phi}
    \end{subfigure}
        \hspace{0.4cm}
    \begin{subfigure}{.3\textwidth}
           \setlength\fheight{4 cm}
           \setlength\fwidth{\textwidth}
\raisebox{-48mm}{
%
%
\definecolor{mycolor1}{rgb}{0.00000,0.44700,0.74100}%
\definecolor{mycolor2}{rgb}{0.85000,0.32500,0.09800}%
\begin{tikzpicture}

\begin{axis}[%
width=0.963\fheight,
height=0.75\fheight,
at={(0\fheight,0\fheight)},
scale only axis,
xmin=0,
xmax=5,
xlabel style={font=\color{white!15!black}},
xlabel={Time $t$},
ymode=log,
ymin=1e-10,
ymax=0.01,
yminorticks=true,
ylabel style={font=\color{white!15!black}},
ylabel={FOM energy error},
axis background/.style={fill=white},
xmajorgrids,
ymajorgrids,
]
\addplot [color=blue, dotted, line width=2.0pt]
  table[row sep=crcr]{%
0.00999999999999979	2.36668711295351e-10\\
0.0199999999999996	3.18671027343953e-09\\
0.0299999999999994	1.36322660182486e-08\\
0.0500000000000007	6.21517347099142e-08\\
0.0600000000000005	1.21819693958969e-07\\
0.0700000000000003	2.19414569073706e-07\\
0.0800000000000001	3.67752445527002e-07\\
0.0999999999999996	8.74894280968872e-07\\
0.119999999999999	1.77063202499994e-06\\
0.140000000000001	3.19456353736313e-06\\
0.16	5.29074059159029e-06\\
0.18	8.19997146663808e-06\\
0.199999999999999	1.20526587897984e-05\\
0.220000000000001	1.69625694252318e-05\\
0.24	2.30218307069663e-05\\
0.26	3.02973302450481e-05\\
0.279999999999999	3.88285788130815e-05\\
0.300000000000001	4.86269888824608e-05\\
0.32	5.96764316287589e-05\\
0.34	7.19348701704803e-05\\
0.359999999999999	8.53368271828003e-05\\
0.380000000000001	9.97964301677713e-05\\
0.4	0.000115210784601913\\
0.42	0.000131463450575211\\
0.44	0.000148427835101756\\
0.460000000000001	0.000165970356329126\\
0.48	0.000183953282576113\\
0.5	0.000202237193816473\\
0.529999999999999	0.000229923794336173\\
0.56	0.00025751626991223\\
0.59	0.000284643397358196\\
0.619999999999999	0.000310926982165256\\
0.65	0.000335845284148491\\
0.68	0.000359071418652092\\
0.710000000000001	0.000380325039691343\\
0.74	0.000399379418704485\\
0.77	0.000416069277175666\\
0.800000000000001	0.000430298629366917\\
0.83	0.000442047445417302\\
0.859999999999999	0.000451375667998945\\
0.890000000000001	0.00045842308299143\\
0.92	0.000463403815210767\\
0.960000000000001	0.00046731555662518\\
1	0.00046873285379661\\
1.06	0.000467967088015939\\
1.22	0.000464254641633488\\
1.32	0.000465068186672397\\
1.5	0.000466657899351048\\
1.72	0.000468570668836036\\
2.49	0.000483198483343586\\
2.6	0.000489402149714806\\
2.74	0.000500160478927683\\
3.02	0.00052319914205782\\
3.15	0.000531110702290789\\
3.28	0.000536433660276999\\
3.42	0.000539553237554174\\
3.6	0.000540698933637031\\
3.93	0.000539399302761013\\
4.28	0.00053908157949536\\
4.57	0.000541401795217097\\
4.99	0.000548051783284791\\
};

\addplot [color=green, dashed, line width=2.0pt]
  table[row sep=crcr]{%
0.00999999999999979	1.83495103556197e-10\\
0.0199999999999996	4.03078956878745e-09\\
0.0299999999999994	2.13819384953239e-08\\
0.0399999999999991	6.846988981124e-08\\
0.0500000000000007	1.67673301803006e-07\\
0.0600000000000005	3.47177297044254e-07\\
0.0700000000000003	6.40548441879218e-07\\
0.0800000000000001	1.08623236428684e-06\\
0.0899999999999999	1.72698335518362e-06\\
0.109999999999999	3.78243476916395e-06\\
0.130000000000001	7.21020576293086e-06\\
0.15	1.24386632205642e-05\\
0.17	1.98978684193207e-05\\
0.19	2.9997123378962e-05\\
0.210000000000001	4.3104752671752e-05\\
0.23	5.95311657980347e-05\\
0.25	7.95159084736951e-05\\
0.27	0.000103219057377828\\
0.289999999999999	0.000130716969638342\\
0.31	0.000162002102565566\\
0.33	0.000196986385167293\\
0.35	0.000235507465167757\\
0.369999999999999	0.000277337074840034\\
0.390000000000001	0.000322190748565845\\
0.41	0.000369738175563726\\
0.43	0.000419613565022701\\
0.449999999999999	0.000471425525765881\\
0.470000000000001	0.000524766098615145\\
0.49	0.000579218718257834\\
0.52	0.000662068457290844\\
0.550000000000001	0.000745088892776948\\
0.58	0.000826940847114203\\
0.609999999999999	0.00090636165380147\\
0.640000000000001	0.00098218473079669\\
0.67	0.00105335753677537\\
0.699999999999999	0.00111896063326185\\
0.73	0.00117822928726128\\
0.76	0.00123057748033716\\
0.790000000000001	0.00127562256860984\\
0.82	0.0013132074102441\\
0.85	0.00134341579101147\\
0.880000000000001	0.00136657666701467\\
0.91	0.00138325327744951\\
0.94	0.00139421462178688\\
0.98	0.00140155928368131\\
1.03	0.00140250648042638\\
1.13	0.00139385557544756\\
1.22	0.00138936142372131\\
1.33	0.00139183599291755\\
1.54	0.00139660742943306\\
1.72	0.00140190408264061\\
2.5	0.00144577140546062\\
2.61	0.00146398413594397\\
2.75	0.0014956907009082\\
3.02	0.00156061798697465\\
3.15	0.00158366396193287\\
3.28	0.00159902833212527\\
3.42	0.00160784483705811\\
3.6	0.00161071516431093\\
3.94	0.00160577379142978\\
4.37	0.0015957558463333\\
4.42	0.00158397227615751\\
4.46	0.00156670327446136\\
4.49	0.00154725191192938\\
4.52	0.00152073061021383\\
4.55	0.00148581195408497\\
4.58	0.00144130869352011\\
4.6	0.00140584169334033\\
4.62	0.00136550309382528\\
4.64	0.00132018343029813\\
4.66	0.00126986033163575\\
4.68	0.00121459944301933\\
4.7	0.00115454993154344\\
4.72	0.00108993449945956\\
4.74	0.00102103430690022\\
4.76	0.000948169698008313\\
4.78	0.000871678087878534\\
4.8	0.000791890751397659\\
4.82	0.000709110515281281\\
4.84	0.000623592452843695\\
4.86	0.000535529598828361\\
4.87	0.000490585223051313\\
4.88	0.000445045431556536\\
4.89	0.000398914880784105\\
4.9	0.000352194429406154\\
4.91	0.000304881641473003\\
4.92	0.000256971426374548\\
4.93	0.000208456808541086\\
4.94	0.000159329818006881\\
4.95	0.00010958248942643\\
4.96	5.92079553780422e-05\\
4.97	8.20161724732315e-06\\
4.98	4.34376250223066e-05\\
4.99	9.57061043254725e-05\\
};

\addplot [color=black, line width=3.0pt, forget plot]
  table[row sep=crcr]{%
4	1e-10\\
4	0.01\\
};
\end{axis}

\begin{axis}[%
width=1.322\fheight,
height=0.991\fheight,
at={(-0.239\fheight,-0.167\fheight)},
scale only axis,
xmin=0,
xmax=1,
ymin=0,
ymax=1,
axis line style={draw=none},
ticks=none,
axis x line*=bottom,
axis y line*=left
]
\end{axis}
\end{tikzpicture}%
}
\caption{FOM energy error}
\label{fig:energy}
    \end{subfigure}
\caption{Klein-Gordon-Zakharov equations. The projection error comparisons in plots (a) and (b) shows that the basis matrix for both approaches provides accurate approximations in the training data regime. However, the projection error for both $\psi$ and $\phi$ grows substantially in the test data regime for both approaches. Despite the projection error growth in the test data regime, the energy error comparison in plot (c) shows that both structure-preserving lifting and PSD approaches yield ROMs with bounded energy error. }
 \label{fig:kgz_proj}
\end{figure}
\begin{figure}[tbp]
\centering

\newcolumntype{C}[1]{>{\centering\arraybackslash}m{#1}}

\setlength{\tabcolsep}{8pt}
\renewcommand{\arraystretch}{1.0}

\begin{tabular}{C{0.20\textwidth} C{0.23\textwidth} C{0.23\textwidth} C{0.23\textwidth}}

\makecell{\small Conservative\\ \small \Rc{Nonlinear} FOM} &

\begin{minipage}[b]{\linewidth}
  \centering
  \includegraphics[width=\linewidth]{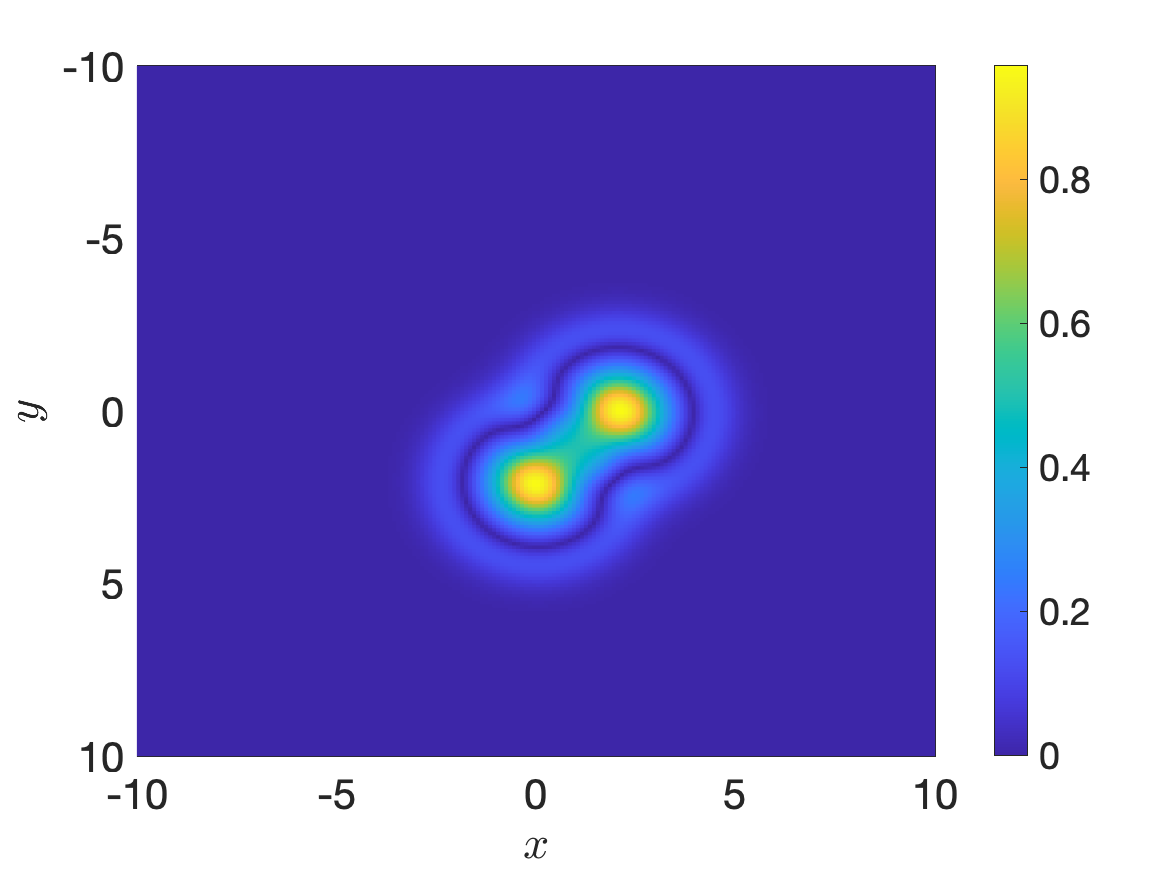}
\end{minipage} &

\begin{minipage}[b]{\linewidth}
  \centering
  \includegraphics[width=\linewidth]{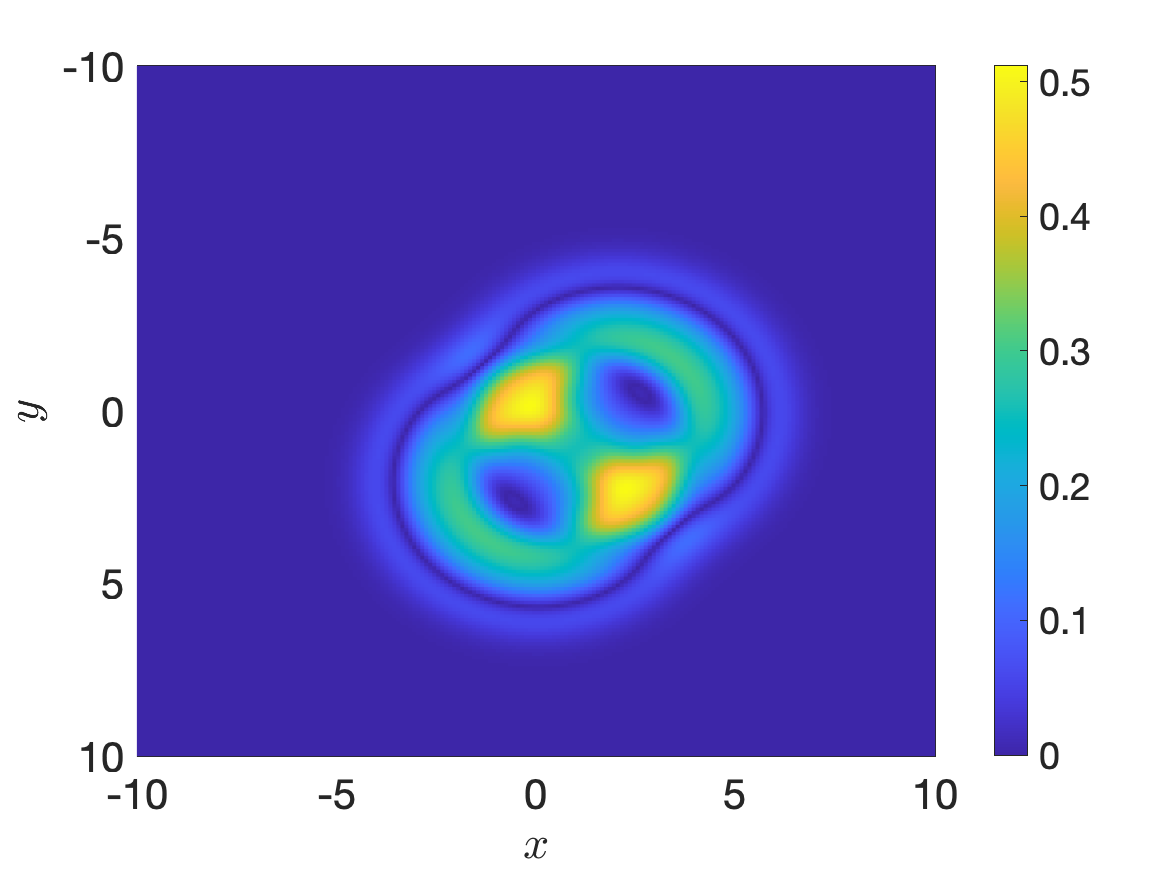}
\end{minipage} &

\begin{minipage}[b]{\linewidth}
  \centering
  \includegraphics[width=\linewidth]{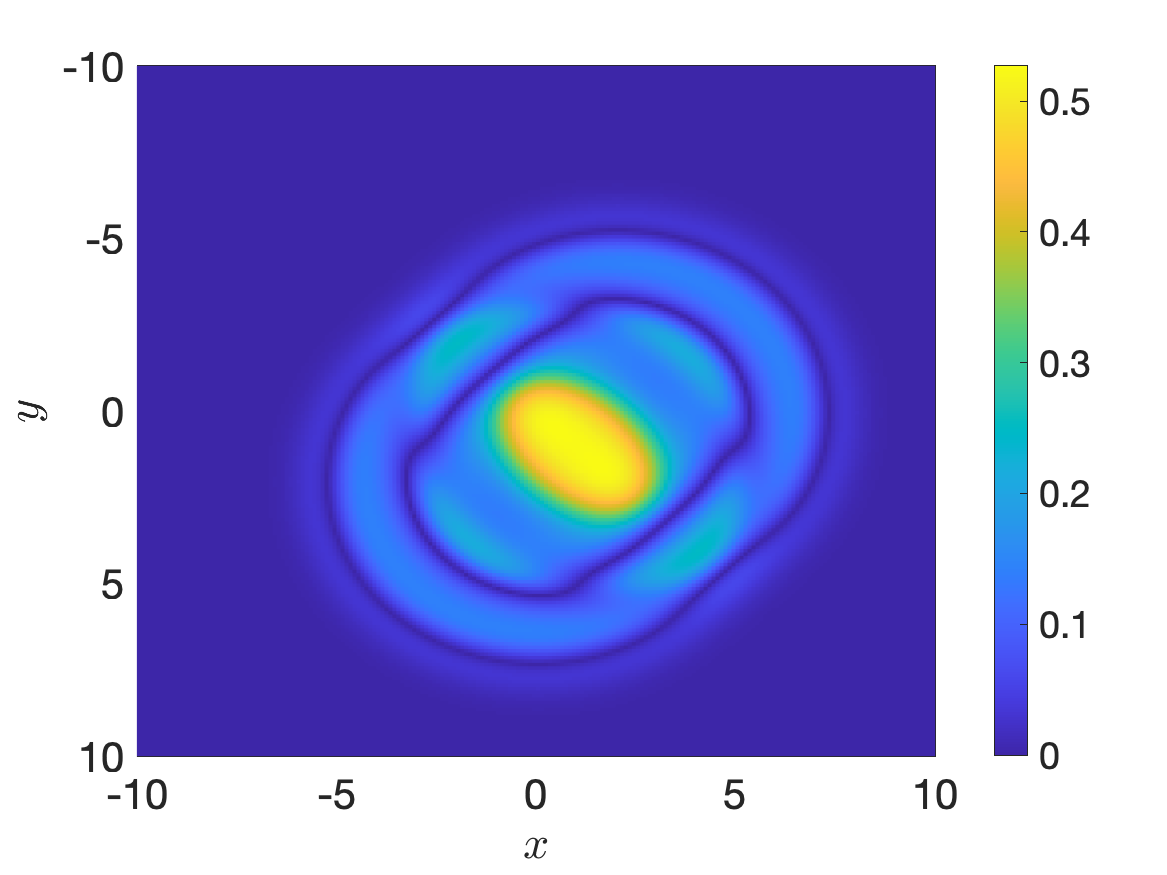}
\end{minipage}
\\[0.6em] 

\makecell{\small \Rc{Structure-preserving}\\ \small Quadratic ROM} &

\begin{minipage}[b]{\linewidth}
  \centering
  \includegraphics[width=\linewidth]{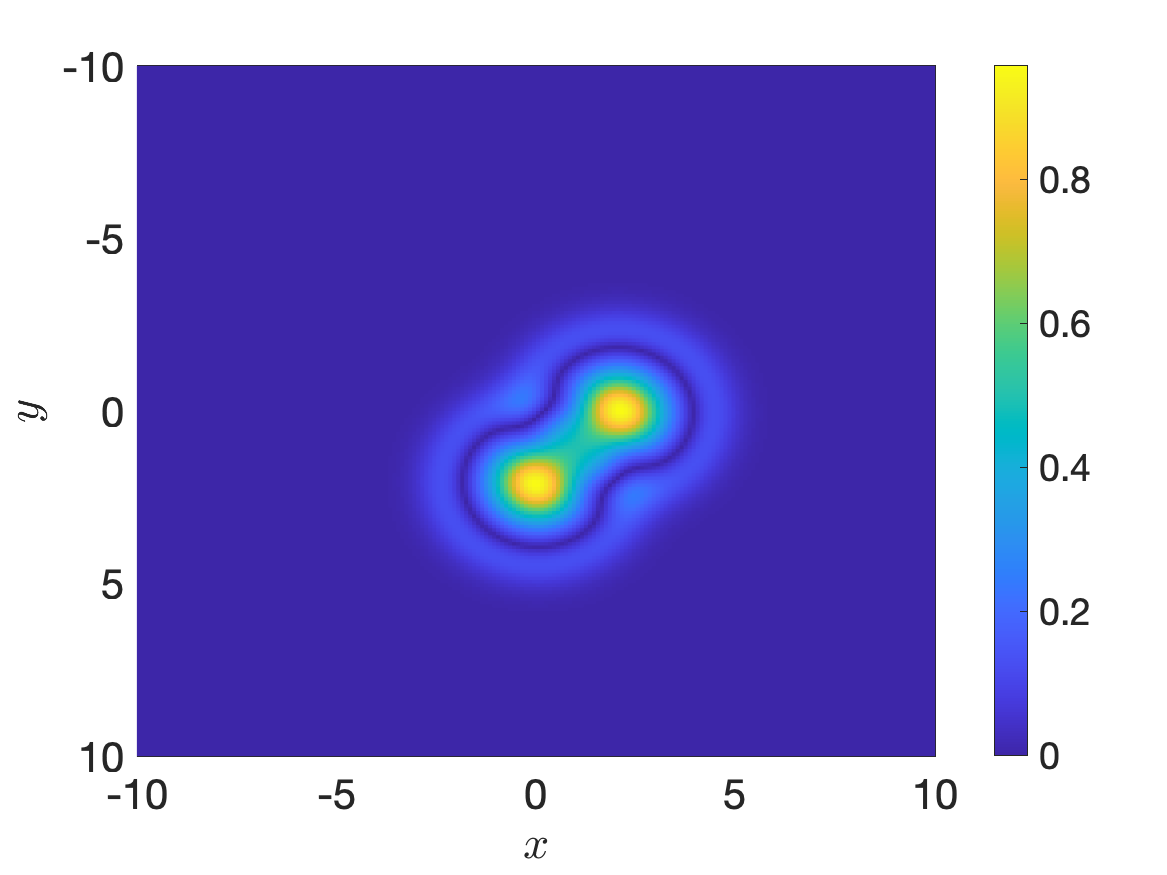}
  \subcaption{$t=1.5$}
\end{minipage} &

\begin{minipage}[b]{\linewidth}
  \centering
  \includegraphics[width=\linewidth]{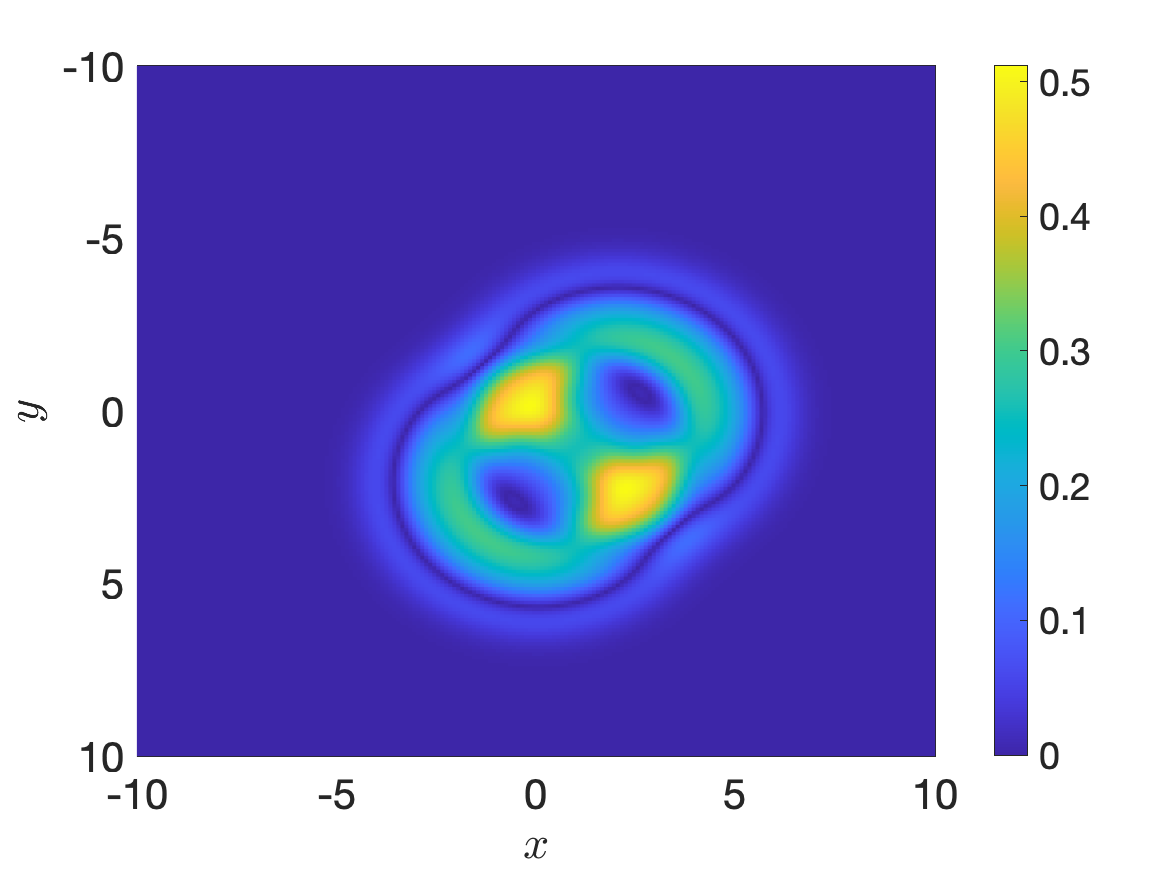}
  \subcaption{$t=3$}
\end{minipage} &

\begin{minipage}[b]{\linewidth}
  \centering
  \includegraphics[width=\linewidth]{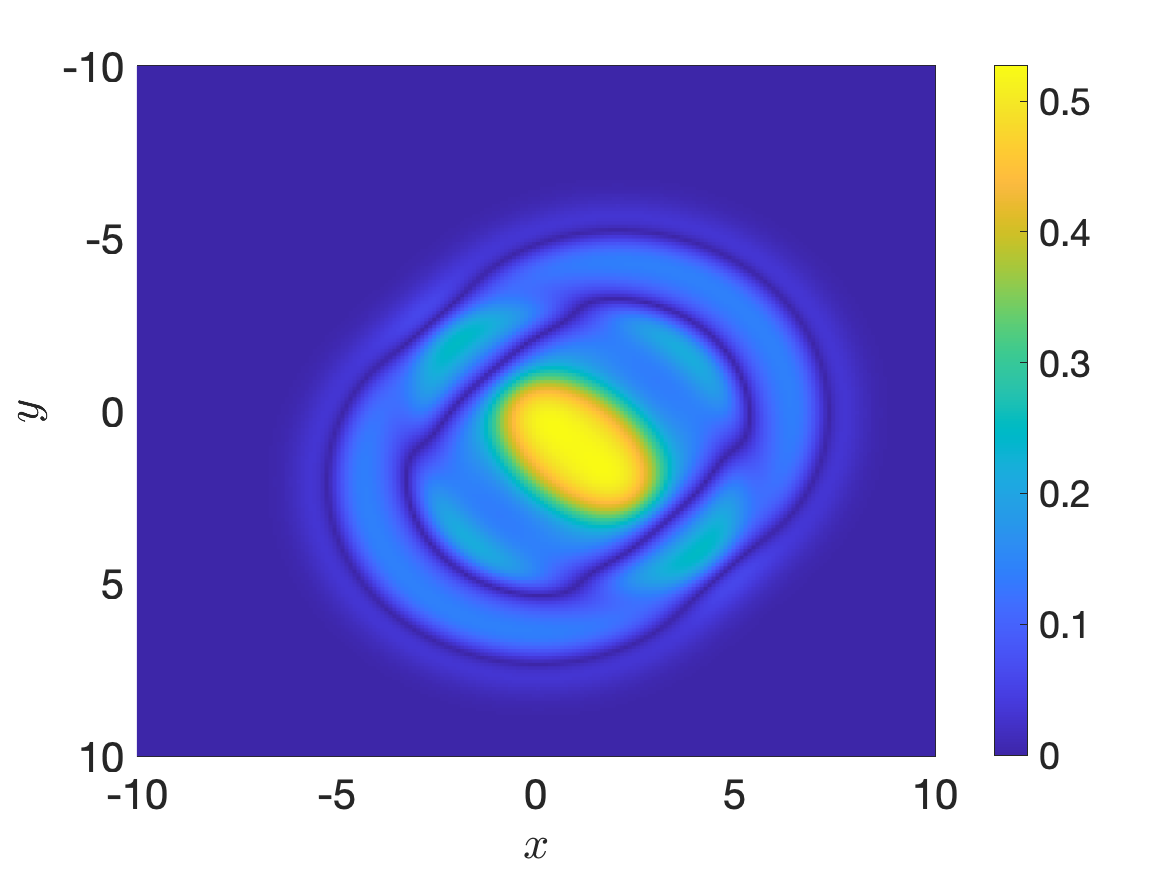}
  \subcaption{$t=4.5$}
\end{minipage}
\\

\end{tabular}
    \caption{Klein-Gordon-Zakharov equations. Plots compare the the absolute value of the complex-valued field $\psi(x,y,t)$ at selected time instances $t\in \{1.5, 3, 4.5\}$ for the nonlinear conservative FOM (top row) and the structure-preserving quadratic ROM (bottom row). The quadratic ROM of dimension $7r=420$ provides accurate approximate solutions for $|\psi(x,y,t)|$ at all three time instances. }
\label{fig:kgz_psi}
\end{figure}
\begin{figure}[tbp]
\centering

\newcolumntype{C}[1]{>{\centering\arraybackslash}m{#1}}

\setlength{\tabcolsep}{8pt}
\renewcommand{\arraystretch}{1.0}

\begin{tabular}{C{0.20\textwidth} C{0.23\textwidth} C{0.23\textwidth} C{0.23\textwidth}}

\makecell{\small Conservative \\ \small \Rc{Nonlinear} FOM} &

\begin{minipage}[b]{\linewidth}
  \centering
  \includegraphics[width=\linewidth]{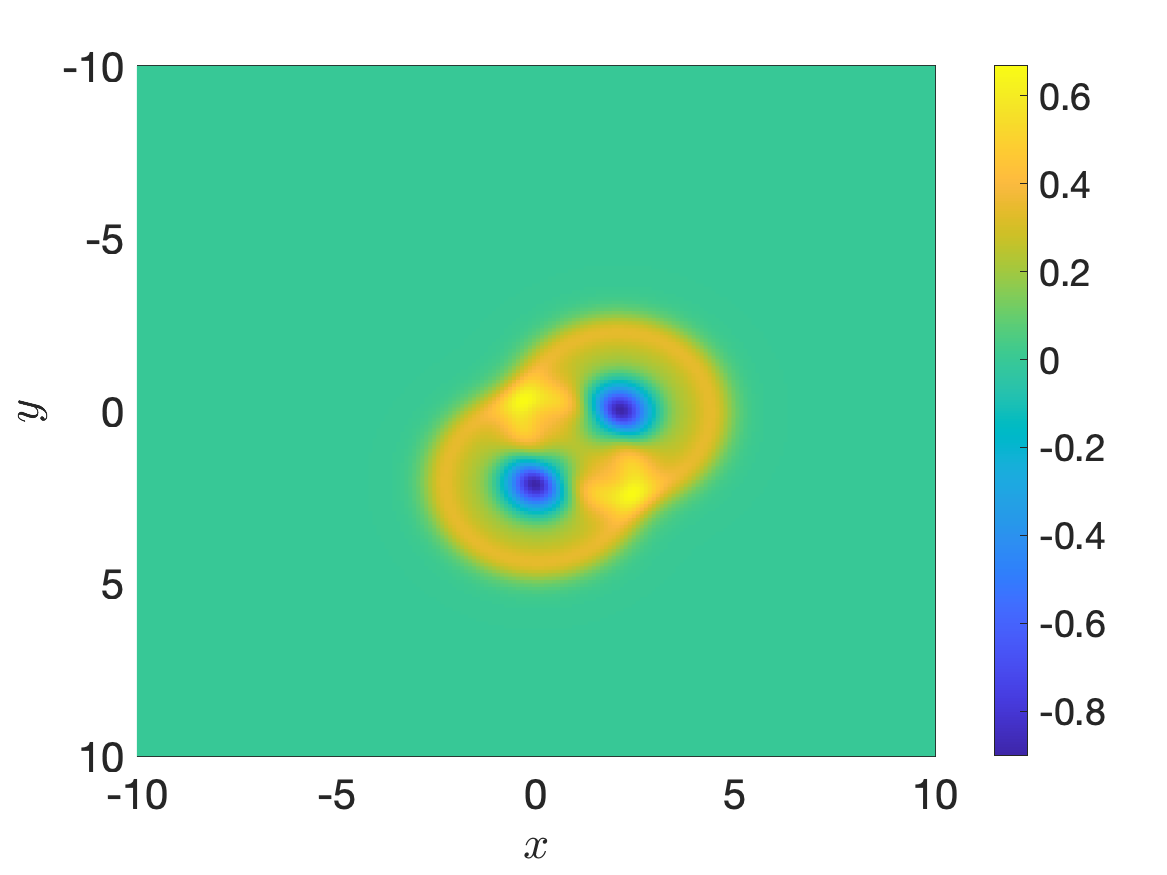}
\end{minipage} &

\begin{minipage}[b]{\linewidth}
  \centering
  \includegraphics[width=\linewidth]{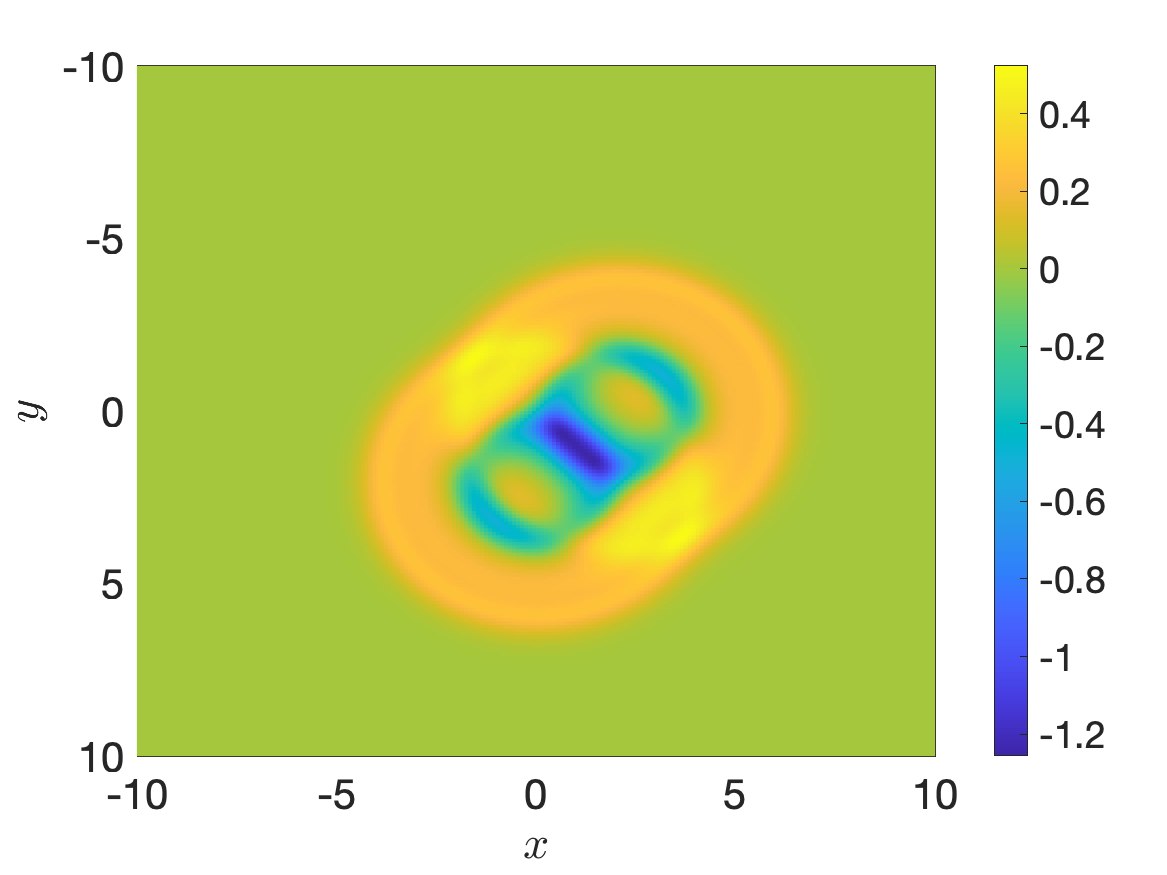}
\end{minipage} &

\begin{minipage}[b]{\linewidth}
  \centering
  \includegraphics[width=\linewidth]{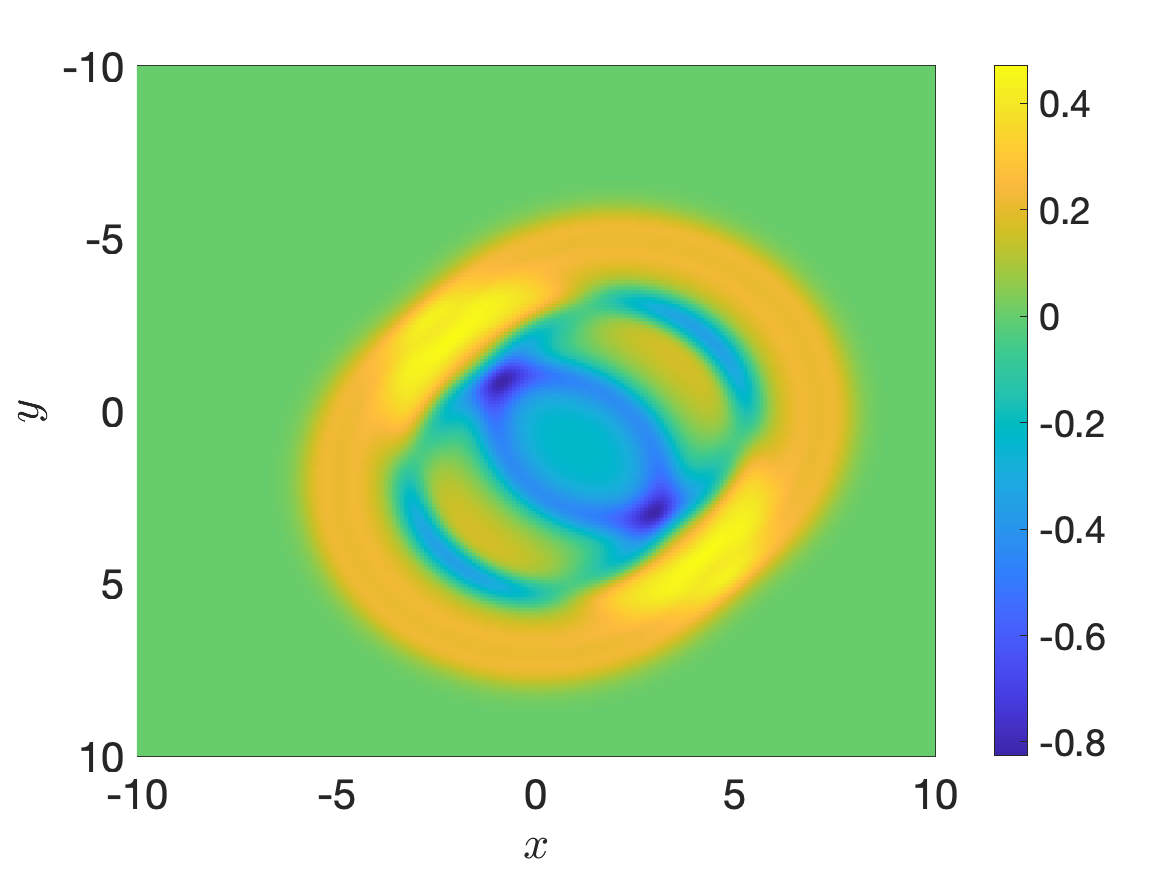}
\end{minipage}
\\[0.6em] 

\makecell{\small \Rc{Structure-preserving} \\ \small Quadratic ROM} &

\begin{minipage}[b]{\linewidth}
  \centering
  \includegraphics[width=\linewidth]{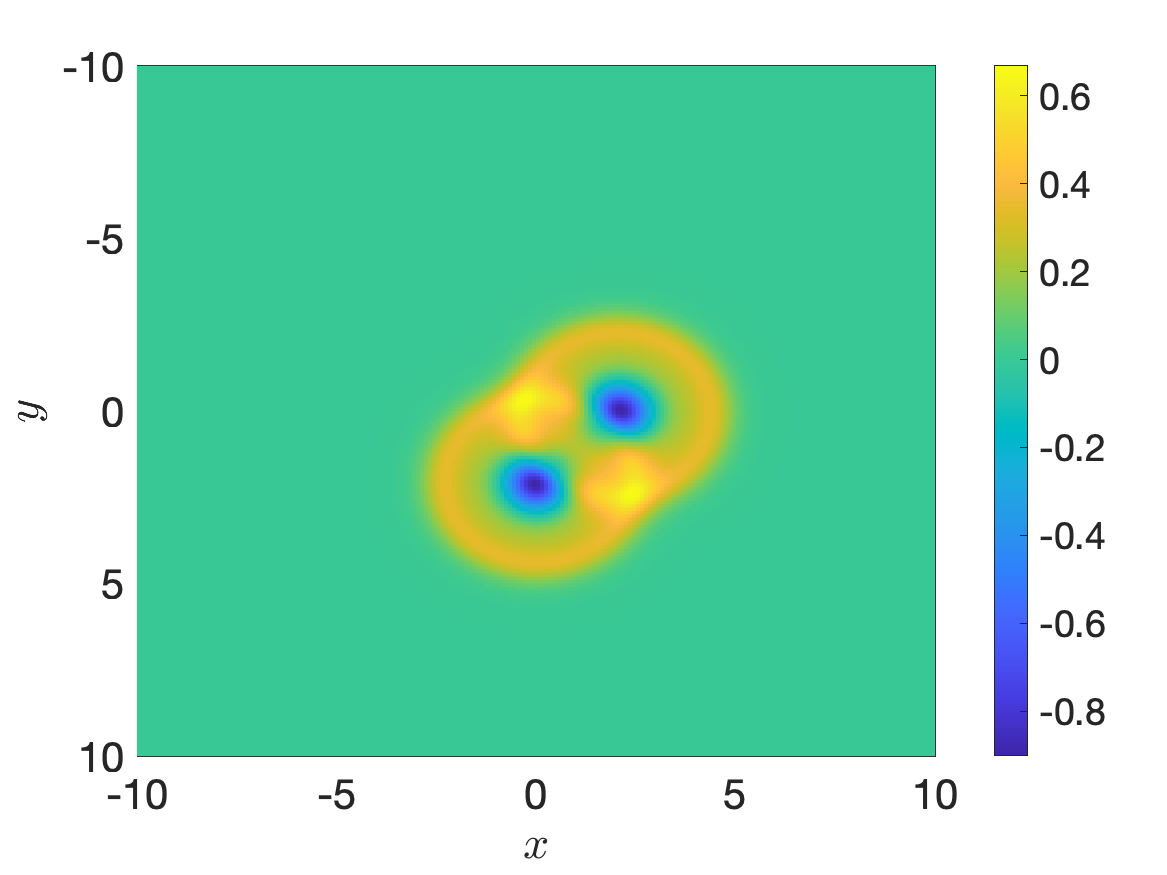}
  \subcaption{$t=1.5$}
\end{minipage} &

\begin{minipage}[b]{\linewidth}
  \centering
  \includegraphics[width=\linewidth]{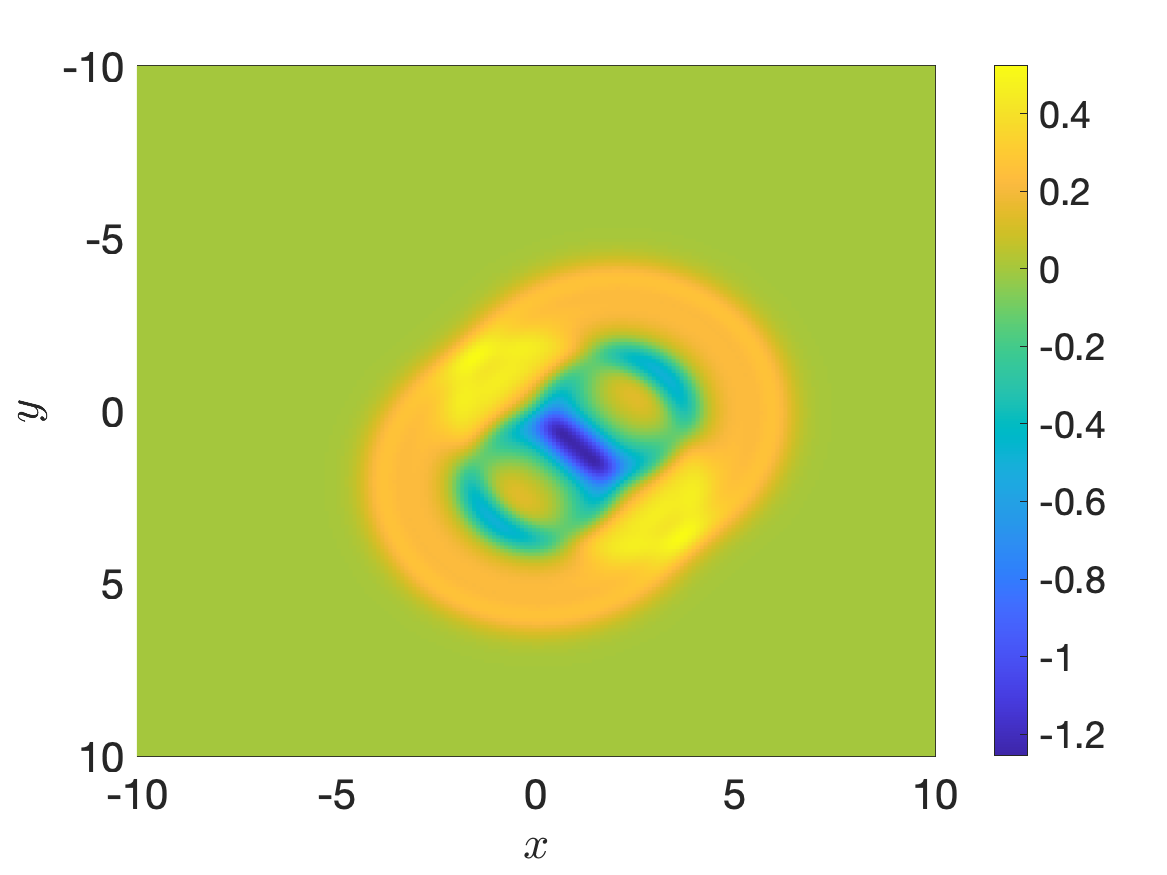}
  \subcaption{$t=3$}
\end{minipage} &

\begin{minipage}[b]{\linewidth}
  \centering
  \includegraphics[width=\linewidth]{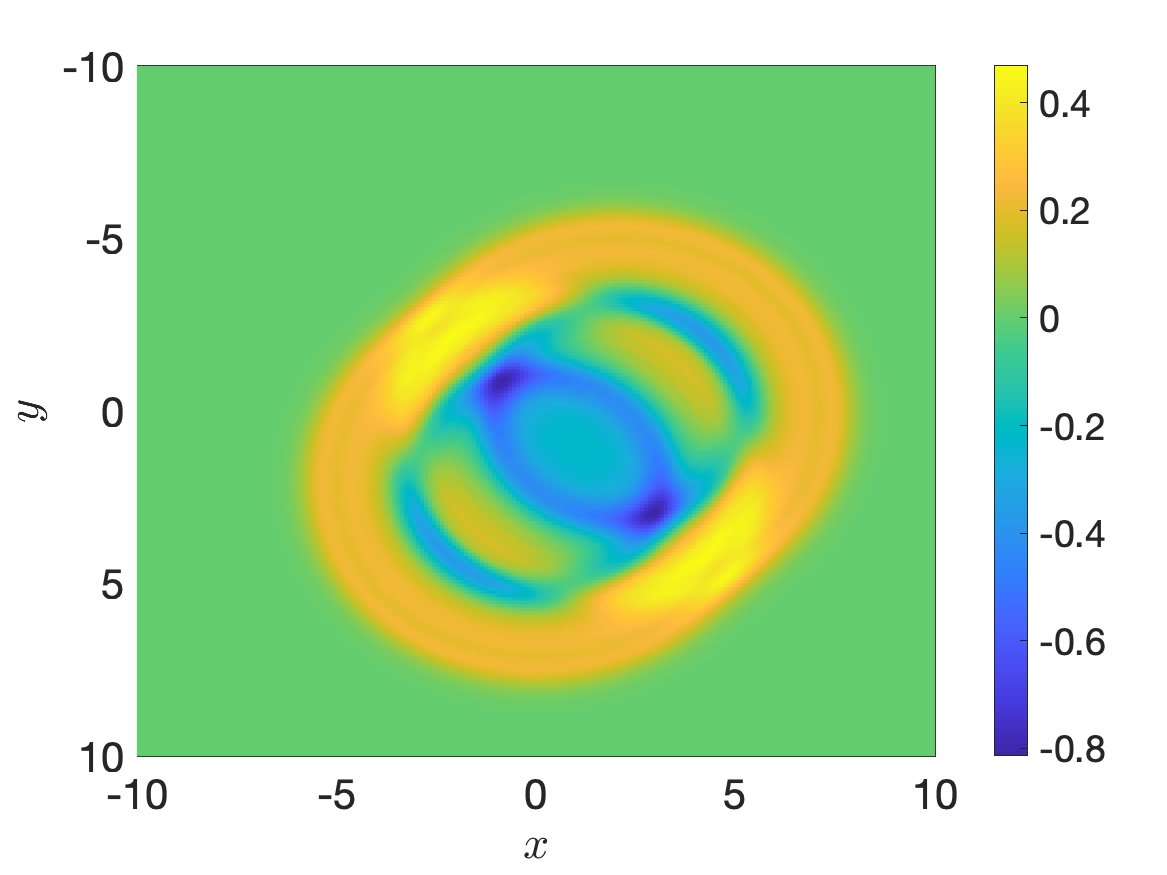}
  \subcaption{$t=4.5$}
\end{minipage}
\\

\end{tabular} 
    \caption{Klein-Gordon-Zakharov equations. Plots compare the the evolution of the scalar field $\phi(x,y,t)$ at selected time instances $t\in \{1.5, 3, 4.5\}$ for the nonlinear conservative FOM (top row) and the structure-preserving quadratic ROM (bottom row). The quadratic ROM of dimension $7r=420$ accurately captures the time-evolution of $\phi(x,y,t)$ at a substantially lower computational cost.}
\label{fig:kgz_phi}
\end{figure}
\section{Conclusions}
\label{sec:conclusion}
We have presented an energy-quadratization strategy for structure-preserving model reduction of nonlinear wave equations via lifting transformations. The proposed approach first transforms the nonlinear conservative FOM with nonlinear energy to a quadratic FOM with a quadratic energy in the lifted variables and then projects them onto a reduced space to derive quadratic ROMs. We presented a theoretical result that shows that the aforementioned quadratic ROMs  conserve the lifted FOM energy exactly. The key advantage of this approach is that it provides a constructive way to derive lifting transformations that respect the key physical properties of the original system. Moreover, this work also provides a computationally efficient alternative strategy to the gradient-preserving hyper-reduction method that is considered the state-of-the-art in structure-preserving model reduction of nonlinear conservative PDEs. 

The numerical examples demonstrate that the structure-preserving lifting approach yields energy-conserving quadratic ROMs that provide accuracy similar to nonlinear PSD ROMs while achieving computational efficiency comparable to the gradient-preserving hyper-reduction method~\cite{pagliantini2023gradient}. The numerical
results also show that the proposed structure-preserving lifting approach yields generalizable quadratic ROMs that provide accurate and stable predictions outside the training data regime in both time extrapolation and parameter extrapolation settings. The final numerical example with the Klein-Gordon-Zakharov equations, a system of six coupled nonlinear conservative PDEs, shows the wider applicability of the proposed approach for nonlinear conservative PDEs with conservation laws that can not be written in the canonical Hamiltonian form. 

Future research directions motivated by this work are: combining the proposed structure-preserving lifting approach with the nonintrusive Lift \& Learn method~\cite{qian2020lift} for learning structure-preserving ROMs of nonlinear wave equations directly from data; {\Ra{extending the proposed structure-preserving lifting approach to problems where the contangent lift basis is not useful}}; studying the connection between qualitative properties of the quadratic ROM and the nonlinear PSD ROM from a dynamical systems perspective, and encoding the proposed energy-quadratization strategy as a physics-based constraint into the optimal quadratization framework~\cite{bychkov2024exact} to automate the process of deriving structure-preserving lifting transformations.
\section*{Acknowledgments}
 H. Sharma and B. Kramer were in part financially supported by the Applied and Computational Analysis Program of the Office of Naval Research under award N000142212624 and NSF-CMMI award 2144023. We thank Iman Adibnazari for his help with running numerical experiments for the two-dimensional Klein-Gordon-Zakharov equations on the Triton Shared Computing Cluster (TSCC) allocation supported by ONR DURIP Award No. N000142412222.  
\bibliographystyle{vancouver}
\bibliography{main}
\end{document}